\documentclass[12pt]{article}

\usepackage{amsmath}
\usepackage{amssymb}
\usepackage{amsthm}
\usepackage{amsfonts}
\usepackage{amscd}
\usepackage{graphicx}
\usepackage[T1]{fontenc}    
\usepackage[nottoc]{tocbibind}  
\usepackage[usenames,dvipsnames]{color}
\usepackage{enumitem} 
\usepackage{caption}
\usepackage[numbers,sort&compress]{natbib}
\usepackage{hyperref}
\usepackage{adjustbox}
\usepackage{fancyvrb} 
\usepackage{tikz}
\usepackage{verbatim} 
\usepackage{framed}
\usepackage{geometry}
\usepackage[medium]{titlesec}
\usepackage{etoolbox} 
\usepackage{arydshln} 
\usepackage[capitalise]{cleveref} 
\usepackage{tocloft} 

\newtheorem{theorem}{Theorem}
\newtheorem{corollary}{Corollary}

\newtheorem{lemma}{Lemma}
\newtheorem{proposition}[lemma]{Proposition}
\newtheorem{conjecture}{Conjecture}
\newtheorem{theoremext}{Theorem}

\theoremstyle{definition}
\newtheorem{definition}[lemma]{Definition}
\newtheorem{example}[lemma]{Example}
\newtheorem{remark}[lemma]{Remark}
\newtheorem{notation}[lemma]{Notation}

\definecolor{colLinkBlue}{RGB}{23,111,192} 
\definecolor{colCiteGreen}{RGB}{8,144,8} 
\definecolor{colP}{RGB}{170,0,255} 
\definecolor{colLP}{RGB}{255,170,255} 
\definecolor{colBrown}{RGB}{156,99,49} 
\definecolor{colGray}{RGB}{128,128,128} 
\definecolor{colGreen}{RGB}{0,204,0} 
\definecolor{colO}{RGB}{255,170,0} 
\definecolor{colLB}{RGB}{143,189,211} 

\hypersetup{colorlinks,urlcolor=colLinkBlue,citecolor=colCiteGreen,linkcolor=colLinkBlue}
\makeatletter\patchcmd{\ttlh@hang}{\parindent\z@}{\parindent\z@\leavevmode}{}{}\patchcmd{\ttlh@hang}{\noindent}{}{}{}\makeatother 
\titlespacing*{\section}{0pt}{0mm}{0mm}
\titlespacing*{\subsection}{0pt}{0mm}{0mm}
\titlespacing*{\paragraph}{0pt}{0mm}{0mm}
\newcommand{\myspace}{\setlength{\abovedisplayskip}{1mm}\setlength{\belowdisplayskip}{0mm}}
\newenvironment{Mlist}{\begin{itemize}[topsep=0pt,itemsep=0pt,leftmargin=7mm]}{\end{itemize}}

\newenvironment{claims}{\begin{enumerate}[topsep=0pt,itemsep=0pt,leftmargin=8mm,label=\textnormal{\textbf{(\alph*)}},ref=(\alph*)]}{\end{enumerate}}

\usetikzlibrary{calc}
\def\centerarc[#1](#2)(#3:#4:#5:#6){\draw[#1] ($(#2)+({#5*cos(#3)},{#5*sin(#3)})$) arc (#3:#4:#5 and #6);}

\newcommand{\csep}[1]{\setlength{\tabcolsep}{#1}}
\newcommand{\fig}[3]{\includegraphics[height=#1cm, width=#2cm]{img/#3}}

\newcommand{\ASN}[1]{Assertion~\ref{#1}}
\newcommand{\EQN}[1]{(\ref{eqn:#1})}
\newcommand{\AXM}[1]{Axiom~\ref{#1}}
\newcommand{\AXMS}[2]{Axioms~\ref{#1} and~\ref{#2}}
\newcommand{\SEC}[1]{\textsection\ref{sec:#1}}

\newcommand{\END}{\hfill $\vartriangleleft$}
\newcommand{\RL}[2]{\Cref{lem:#1}\ref{lem:#1:#2}}
\newcommand{\RLS}[4]{Lemmas~\ref{lem:#1}\ref{lem:#1:#2} and~\ref{lem:#3}\ref{lem:#3:#4}}

\newcommand{\df}[1]{{\it #1}}
\newcommand{\Wlog}{without loss of generality }
\newcommand{\resp}{respectively}
\newcommand{\st}{such that }
\newcommand{\wrt}{with respect to }

\renewcommand{\c}{\colon}

\newcommand{\set}[2]{\{#1:#2\}}
\newcommand{\dto}{\dasharrow}

\newcommand{\bas}[1]{\langle #1\rangle}
\newcommand{\aut}{\operatorname{Aut}}
\newcommand{\Sing}{\operatorname{Sing}}

\newcommand{\id}{\operatorname{id}}
\newcommand{\radius}{\operatorname{radius}}
\newcommand{\plane}{\operatorname{plane}}
\newcommand{\cent}{\operatorname{center}}

\newcommand{\VisPoints}{\operatorname{VisPoints}}

\renewcommand{\P}{{\mathbb{P}}}
\renewcommand{\S}{{\mathbb{S}}}
\newcommand{\C}{{\mathbb{C}}}
\newcommand{\R}{{\mathbb{R}}}
\newcommand{\Z}{{\mathbb{Z}}}
\newcommand{\Q}{{\mathbb{Q}}}
\newcommand{\U}{{\mathbb{U}}}
\newcommand{\E}{{\mathbb{E}}}

\newcommand{\bA}{{\mathbf{A}}}
\newcommand{\bO}{{\mathbf{O}}}
\newcommand{\bU}{{\mathbf{U}}}
\newcommand{\bS}{{\mathbf{S}}}

\newcommand{\cV}{{\mathcal{V}}}
\newcommand{\cW}{{\mathcal{W}}}
\newcommand{\cF}{{\mathcal{F}}}
\newcommand{\cG}{{\mathcal{G}}}
\newcommand{\cE}{{\mathcal{E}}}

\newcommand{\cX}{\mathcal{X}}
\newcommand{\cZ}{\mathcal{Z}}
\newcommand{\cU}{\mathcal{U}}
\newcommand{\cH}{\mathcal{H}}

\renewcommand{\k}{\kappa}

\renewcommand{\l}{\ell}
\newcommand{\p}{\varepsilon}
\newcommand{\tb}{\tilde{b}}
\newcommand{\bp}{b'}

\newcommand{\pr}{\pi}

\newcommand{\pp}{\mathfrak{p}}
\newcommand{\op}{{\overline{\mathfrak{p}}}}
\newcommand{\qq}{\mathfrak{q}}
\newcommand{\oq}{{\overline{\mathfrak{q}}}}
\renewcommand{\aa}{\mathfrak{a}}
\newcommand{\oa}{{\overline{\mathfrak{a}}}}
\newcommand{\bb}{\mathfrak{b}}
\newcommand{\ob}{{\overline{\mathfrak{b}}}}
\newcommand{\vv}{\mathfrak{v}}
\newcommand{\rr}{\mathfrak{r}}
\newcommand{\rs}{\mathfrak{s}}
\newcommand{\rt}{\mathfrak{t}}

\newcommand{\oL}{\overline{L}}
\newcommand{\oR}{\overline{R}}

\newcommand{\SingType}{\operatorname{SingType}}
\newcommand{\VisType}{\operatorname{VisType}}
\newcommand{\PseudoType}{\operatorname{PseudoType}}

\newcommand{\Piii}{\hyperref[tab:P]{\operatorname{P3}}}
\newcommand{\Piv}{\hyperref[tab:P]{\operatorname{P4}}}
\newcommand{\Pv}{\hyperref[tab:P]{\operatorname{P5}}}

\newcommand{\Di}{\hyperref[tab:s]{\operatorname{D1}}}
\newcommand{\Dii}{\hyperref[tab:s]{\operatorname{D2}}}
\newcommand{\Diii}{\hyperref[tab:s]{\operatorname{D3}}}
\newcommand{\Div}{\hyperref[tab:s]{\operatorname{D4}}}
\newcommand{\Dv}{\hyperref[tab:s]{\operatorname{D5}}}

\newcommand{\Bi}{\hyperref[tab:A4]{\operatorname{B1}}}
\newcommand{\Bii}{\hyperref[tab:A4]{\operatorname{B2}}}
\newcommand{\Biii}{\hyperref[tab:A4]{\operatorname{B3}}}

\newcommand{\Ei}{\hyperref[tab:E]{\operatorname{E1}}}
\newcommand{\Eii}{\hyperref[tab:E]{\operatorname{E2}}}

\newcommand{\Dia}{\hyperref[tab:D123]{\operatorname{D1a}}}
\newcommand{\Diia}{\hyperref[tab:D123]{\operatorname{D2a}}}

\newcommand{\Dva}{\hyperref[tab:D45]{\operatorname{D5a}}}
\newcommand{\Dvb}{\hyperref[tab:D45]{\operatorname{D5b}}}
\newcommand{\Dvc}{\hyperref[tab:D45]{\operatorname{D5c}}}

\newcommand{\DI}{
\begin{tikzpicture}[scale=0.45]
\draw[thick, colBrown] (0,0) to [out=180, in=90] (-2,-2)
                              to [out=-90, in=175] (-1,-3.5)
                              to [out=5, in=270] (-1.5,-1.5)
                              to [out=90, in=180] (0,0);
\draw[thick, colBrown] (0,0) to [out=0, in=90] (2,-2)
                              to [out=-90, in=5] (1,-3.5)
                              to [out=175, in=270] (1.5,-1.5)
                              to [out=90, in=0] (0,0);
\draw[thick, red] (-2,0) -- (2,0);
\draw[thick, red] (0,2)  -- (0,-2);
\draw[thick, red] (-1.2,-1.2) -- (1.5,1.5);
\draw[thick, red] (1.2,-1.2) -- (-1.5,1.5);
\draw[draw=black, fill=colO] (0,0) circle [radius=0.2] node[black] {};
\draw[draw=white, fill=white] (0,-4) circle [radius=0.2] node[white] {};
\end{tikzpicture}}

\newcommand{\DII}{
\begin{tikzpicture}[scale=0.45]
\draw[thick, colGray] (0,0) to [out=180, in=90] (-2,-2)
                              to [out=-90, in=180] (0,-4)
                              to [out=0, in=-90] (2,-2)
                              to [out=90, in=0] (0,0);
\draw[thick, colGreen] (0,0) to [out=180, in=180] (0,-4)
                                        to [out=0, in=0] (0,0);
\draw[thick, red] (-2,0) -- (2,0);
\draw[thick, red] (0,2)  -- (0,-2);
\draw[thick, red] (-0.8,-0.8) -- ( 1.5,1.5);
\draw[thick, red] ( 0.8,-0.8) -- (-1.5,1.5);
\draw[draw=black, fill=colO] (0,0) circle [radius=0.2] node[black] {};
\draw[draw=black, fill=colO] (0,-4) circle [radius=0.2] node[black] {};
\end{tikzpicture}}

\newcommand{\DIII}{
\begin{tikzpicture}[scale=0.35]
\draw[thick, red] (-3.5, 2) -- ( 3.5, 2);
\draw[thick, red] (-3.5,-2) -- ( 3.5,-2);
\draw[thick, red] (-2, 4) -- (-2,-4);
\draw[thick, red] ( 2, 4) -- ( 2,-4);
\draw[thick, blue] (-2,2) to [out=45, in=180] (0,4)
                               to [out=0, in=135] (2,2)
                               to [out=-45, in=90] (4,0)
                               to [out=270, in=45] (2,-2)
                               to [out=225, in=0] (0,-4)
                               to [out=180, in=-45] (-2,-2)
                               to [out=135, in=270] (-4,0)
                               to [out=90, in=225] (-2,2);
\draw[draw=black, fill=colO] (-2, 2) circle [radius=0.27] node[black] {};
\draw[draw=black, fill=colO] ( 2, 2) circle [radius=0.27] node[black] {};
\draw[draw=black, fill=colO] (-2,-2) circle [radius=0.27] node[black] {};
\draw[draw=black, fill=colO] ( 2,-2) circle [radius=0.27] node[black] {};
\end{tikzpicture}}

\newcommand{\DIV}{
\begin{tikzpicture}[scale=0.35]
\draw[thick, red] (-3.5, 2) -- ( 3.5, 2);
\draw[thick, red] (-3.5,-2) -- ( 3.5,-2);
\draw[thick, red] (-2, 4) -- (-2,-4);
\draw[thick, red] ( 2, 4) -- ( 2,-4);
\draw[thick, colGreen] (-2,2) to [out=-5, in=135] (0,1)
                                         to [out=-45, in=135] (1,0)
                                         to [out=-45, in=95] (2,-2);
\draw[thick, colGreen] (2,-2) to [out=175, in=-45] (0.1,-1.1);
\draw[thick, colGreen] (-0.1,-0.9) to [out=135, in=-45] (-0.9,-0.1 );
\draw[thick, colGreen] (-1.1,0.1) to [out=135, in=275] (-2,2);
\draw[thick, colGray] (-2,-2) to [out=85, in=225] (-1,0)
                                to [out=225, in=225] (0,1)
                                to [out=45, in=185] (2,2);
\draw[thick, colGray] (-2,-2) to [out=5, in=225] (0,-1)
                                to [out=225, in=225] (0.9,-0.1);
\draw[thick, colGray] (1.1,0.1) to [out=45, in=265] (2,2);
\draw[draw=black, fill=colO] (-2, 2) circle [radius=0.27] node[black] {};
\draw[draw=black, fill=colO] ( 2, 2) circle [radius=0.27] node[black] {};
\draw[draw=black, fill=colO] (-2,-2) circle [radius=0.27] node[black] {};
\draw[draw=black, fill=colO] ( 2,-2) circle [radius=0.27] node[black] {};
\draw[draw=black, fill=colO] (0,1  ) circle [radius=0.27] node[black] {};
\end{tikzpicture}}

\newcommand{\DV}{
\begin{tikzpicture}[scale=0.35]
\draw[thick, red] (-3.5, 2) -- ( 3.5, 2);
\draw[thick, red] (-3.5,-2) -- ( 3.5,-2);
\draw[thick, red] (-2, 4) -- (-2,-4);
\draw[thick, red] ( 2, 4) -- ( 2,-4);
\draw[thick, colGreen] (0,2) to [out=180, in=180] (0,-2)
                              to [out=0, in=0] (0,2);
\draw[draw=black, fill=colO] (-2, 2) circle [radius=0.27];
\draw[draw=black, fill=colO] ( 2, 2) circle [radius=0.27];
\draw[draw=black, fill=colO] (-2,-2) circle [radius=0.27];
\draw[draw=black, fill=colO] ( 2,-2) circle [radius=0.27];
\draw[draw=black, fill=colO] (0, 2) circle [radius=0.27];
\draw[draw=black, fill=colO] (0,-2) circle [radius=0.27];
\end{tikzpicture}}

\newcommand{\DIa}{
\begin{tikzpicture}[xscale=0.9,yscale=1.4]
\centerarc[thick](0,0)(0:360:2:1)
\centerarc[thick](0,0)(0:360:1:0.4)
\draw[]               (-1,0) arc (0:180:0.5  and 0.2);
\draw[densely dotted] (-1,0) arc (0:-180:0.5 and 0.2);
\draw[]               (2,0) arc (0:180:0.5   and 0.2);
\draw[densely dotted] (2,0) arc (0:-180:0.5  and 0.2);
\draw[very thick,red] (0,-1) to (0,-0.4);
\draw[very thick,red] (0,1) to (0,0.4);
\end{tikzpicture}}

\newcommand{\DIb}{
\begin{tikzpicture}[xscale=0.75,yscale=1.5]
\centerarc[thick](0,0)(0:360:2:1)
\centerarc[thick](0,0)(0:360:1:0.6)
\draw[densely dotted] (-1,0) arc (0:180:0.5  and 0.1);
\draw[              ] (-1,0) arc (0:-180:0.5 and 0.1);
\draw[densely dotted] (2,0) arc (0:180:0.5   and 0.1);
\draw[              ] (2,0) arc (0:-180:0.5  and 0.1);
\draw[very thick,red] (-0.2, 0.6) to (0.2, 1);
\draw[very thick,red] (-0.2,-1) to (0.2,-0.6);
\end{tikzpicture}}

\newcommand{\DIIa}{
\begin{tikzpicture}[xscale=0.9,yscale=0.7]
\draw[thick] (0, 2)
             to [out=0,in=100] (2, 0)
             to [out=-80,in=0] (0.5, -2)
             to [out=180,in=-90] (-2,0.5)
             to [out=90,in=180] (0,2);
\draw[thick] (-1, 1)
             to [out=90,in=170] (0.7,1.87)
             to [out=-120,in=0]  (-1,1);
\draw[thick,densely dotted]         ( -1,   1-2.5)
             to [out=90,in=170] (0.7,1.87-2.5)
             to [out=-120,in=0]  ( -1,   1-2.5);
\draw[thick]                      (-1, 1)
             to [out=-120,in=130] (-1,-1.5)
             to [out=70,in=-60]   (-1,1);
\draw[thick,densely dotted]               (0.7,1.87)
             to [out=-120,in=130] (0.7,-1.5+0.87)
             to [out=70,in=-60]   (0.7,1.87);
\draw[thick,densely dotted,red] (0.7,1.87) to (-1,-1.5);
\draw[thick,densely dotted,red] (-1,1) to (0.7,1.87-2.5);
\end{tikzpicture}}

\newcommand{\DIIb}{
\begin{tikzpicture}[xscale=0.7,yscale=0.7]
\draw[thick] (0, 2)
             to [out=0,in=100] (2, 0)
             to [out=-80,in=0] (0.5, -2)
             to [out=180,in=-90] (-2,0.5)
             to [out=90,in=180] (0,2);
\draw[thick] (-1, 1)
             to [out=90,in=170] (0.7,1.87)
             to [out=-120,in=0]  (-1,1);
\draw[thick,densely dotted]         ( -1,   1-2.5)
             to [out=90,in=170] (0.7,1.87-2.5)
             to [out=-120,in=0]  ( -1,   1-2.5);
\draw[thick]                      (-1, 1)
             to [out=-120,in=130] (-1,-1.5)
             to [out=70,in=-60]   (-1,1);
\draw[thick,densely dotted]               (0.7,1.87)
             to [out=-120,in=130] (0.7,-1.5+0.87)
             to [out=70,in=-60]   (0.7,1.87);
\draw[thick,densely dotted,red] (0.7,1.87) to (-1,-1.5);
\draw[thick,densely dotted,red] (-1,1) to (0.7,1.87-2.5);
\end{tikzpicture}}

\newcommand{\DIIIa}{
\begin{tikzpicture}[xscale=0.5,yscale=0.5]
\draw (-4,0) to [out=90,in=90] (-3,0) to [out=90,in=90] (-2,0);
\draw[dotted] (-4,0) to [out=270,in=270] (-3,0) to [out=270,in=270] (-2,0);
\draw[thick] (-2,0) to [out=90, in=180] (0,1);
\draw[thick,red] (-3,0) to [out=90, in=180] (0,2);
\draw[thick] (-4,0) to [out=90, in=180] (0,3);
\draw[thick] (-2,0) to [out=270, in=180] (0,-1) to [out=0, in=190] (0.95,-0.9);
\draw[dotted] (0.95,-0.9) to [out=10, in=270] (2,0);
\draw[thick] (-4,0) to [out=270, in=180] (0,-3);
\draw[thick] (0,1) to [out=0, in=90] (2,0);
\draw[thick] (0,2) to [out=0, in=90] (3,0);
\draw[dotted] (0,-2) to [out=0,in=190] (1.6,-1.8);
\draw[thick] (1.65,-1.8) to [out=15,in=270] (3,0);
\draw[dotted] (0,2) to [out=0,in=100] (0.9,0.9);
\draw[thick] (0.9,0.9) to [out=280,in=90] (1,0) to [out=270,in=0] (0,-2);
\draw[thick] (0,3) to [out=0,in=100] (1.8,1.65);
\draw[dotted] (1.8,1.65) to [out=280,in=90] (2,0);
\draw[thick] (2,0) to [out=270,in=0] (0,-3);
\draw[thick,red] (-3,0) to [out=270, in=180] (0,-2);
\draw[thick,red,densely dotted] (0,2) to [out=0,in=90] (2,0) to [out=270, in=0] (0,-2);
\end{tikzpicture}}

\newcommand{\DIIIb}{
\begin{tikzpicture}[xscale=0.9,yscale=1.4]
\centerarc[thick](0,0)(0:360:2:1)
\centerarc[thick](0,0)(0:360:1:0.4)
\draw[              ] (-1,0) arc (0: 180:0.5 and 0.2);
\draw[densely dotted] (-1,0) arc (0:-180:0.5 and 0.2);
\draw[              ] (2 ,0) arc (0: 180:0.5 and 0.2);
\draw[densely dotted] (2 ,0) arc (0:-180:0.5 and 0.2);
\draw[very thick,red] (0,-1) to (0,-0.4);
\draw[very thick,red] (0,1) to (0,0.4);
\end{tikzpicture}}

\newcommand{\DIIIc}{
\begin{tikzpicture}[xscale=-0.9,yscale=1.4]
\centerarc[thick](0,0)(0:360:2:1)
\centerarc[thick,densely dotted](0,-0.455)(49:79:1:0.4)
\centerarc[thick](0,0)(180:79:1:0.4)
\centerarc[thick](0,0)(-180:-41:1:0.4)
\draw[thick]        (1,0) arc (0:232:0.5 and 0.5);
\draw[thick,densely dotted] (1,0) arc (360:232:0.5 and 0.5);
\draw[densely dotted] (0.75,0.26) arc (60:180:0.5 and 0.3);
\draw[]        (0.75,-0.26) arc (-60:-180:0.5 and 0.3);
\draw[thick,densely dotted,red] (1,0) arc (0:60:0.5 and 0.3);
\draw[thick,red]        (1,0) arc (0:-60:0.5 and 0.3);
\draw[]        (-1,0) arc (0:180:0.5 and 0.5);
\draw[densely dotted] (-1,0) arc (0:-180:0.5 and 0.5);
\end{tikzpicture}}

\newcommand{\DIIId}{
\begin{tikzpicture}[xscale=0.9,yscale=1.4]
\centerarc[thick](0,0)(0:360:2:1)
\centerarc[thick](0,0)(-190:10:1:0.4)
\centerarc[thick](0,-0.4)(160:20:0.8:0.4)
\end{tikzpicture}}

\newcommand{\DIIIe}{
\begin{tikzpicture}[xscale=-1.2,yscale=1.3]
\draw[thick] (-2,0) to [out=90,in=90] (1,0) to [out=270, in=270] (-2,0);
\centerarc[thick,densely dotted](0,-0.455)(49:79:1:0.4)
\centerarc[thick](0,0)(180:79:1:0.4)
\centerarc[thick](0,0)(-180:-41:1:0.4)
\draw[thick]        (1,0) arc (0:232:0.5 and 0.5);
\draw[thick,densely dotted] (1,0) arc (360:232:0.5 and 0.5);
\draw[dotted] (0.75,0.26) arc (60:180:0.5 and 0.3);
\draw[]        (0.75,-0.26) arc (-60:-180:0.5 and 0.3);
\draw[thick,densely dotted,red] (1,0) arc (0:60:0.5 and 0.3);
\draw[thick,red]        (1,0) arc (0:-60:0.5 and 0.3);
\draw[]        (-1,0) arc (0:180:0.5 and 0.5);
\draw[densely dotted] (-1,0) arc (0:-180:0.5 and 0.5);
\end{tikzpicture}}

\newcommand{\DIVa}{
\begin{tikzpicture}[xscale=0.9,yscale=0.9]
\centerarc[fill=white,thick](0,0)(0:360:1.5:1)
\centerarc[fill=white,thick](1,0)(0:360:1.5:1.5)
\centerarc[line width=0.5mm,white,densely dotted](1,0)(180:221:1.5:1.5)
\centerarc[thick,densely dotted](0,0)(0:95:1.5:1)
\centerarc[red,very thick](0,0)(0:-180:0.5:0.2)
\centerarc[red,densely dotted,thick](0,0)(0:180:0.5:0.2)
\centerarc[red,thick](1,0)(0:180:0.5:0.5)
\centerarc[densely dotted,red,thick](1,0)(180:360:0.5:0.5)
\centerarc[thick](0,0.5)(-95:0:1.5:1)
\end{tikzpicture}}

\newcommand{\DIVb}{
\begin{tikzpicture}[xscale=0.9,yscale=0.7]
\draw[thick] (0, 2)
             to [out=0,in=100] (2, 0)
             to [out=-80,in=0] (0.5, -2)
             to [out=180,in=-90] (-2,0.5)
             to [out=90,in=180] (0,2);
\draw[thick] (-1, 1)
             to [out=90,in=170] (0.7,1.87)
             to [out=-120,in=0]  (-1,1);
\draw[thick,densely dotted]         ( -1,   1-2.5)
             to [out=90,in=170] (0.7,1.87-2.5)
             to [out=-120,in=0]  ( -1,   1-2.5);
\draw[thick]                      (-1, 1)
             to [out=-120,in=130] (-1,-1.5)
             to [out=70,in=-60]   (-1,1);
\draw[thick,densely dotted]               (0.7,1.87)
             to [out=-120,in=130] (0.7,-1.5+0.87)
             to [out=70,in=-60]   (0.7,1.87);
\draw[thick,densely dotted,red] (0.7,1.87) to (-1,-1.5);
\draw[thick,densely dotted,red] (-1,1) to (0.7,1.87-2.5);
\end{tikzpicture}}

\newcommand{\DIVc}{
\begin{tikzpicture}[xscale=0.9,yscale=1.1]
\centerarc[thick](0,0)(0:360:2:1)
\centerarc[](-1.5,0)(0:180:0.5:0.5)
\centerarc[densely dotted](-1.5,0)(180:360:0.5:0.5)
\centerarc[densely dotted,red,thick](0,0)(0:180:1:0.4)
\centerarc[red,thick](0,0)(180:360:1:0.4)
\centerarc[red,thick](1.5,0)(110:180:0.5:0.5)
\centerarc[densely dotted,red,thick](1.5,0)(180:250:0.5:0.5)
\centerarc[line width=0.5mm,white,densely dotted](0,-0.78)(51:100:2:1)
\draw[thick] (-1,0) to [out=90,in=180] (0.5,1.2) to [out=0,in=90] (1.34,0.48);
\draw[thick,densely dotted] (-1,0) to [out=270,in=150] (-0.3,-1);
\draw[thick] (-0.3,-1) to [out=-30,in=270] (1.34,-0.48);
\end{tikzpicture}}

\newcommand{\DIVd}{
\begin{tikzpicture}[xscale=1.8,yscale=1.1]
\draw[very thick] (-0.5,0) to [out=90,in=180] (0.5,1.3) to [out=0, in=90] (1.5,0);
\draw[very thick] (-0.5,0) to [out=-90,in=180] (0.5,-1.3) to [out=0, in=-90] (1.5,0);
\draw[thick,black,densely dotted] (1.5,0) to [out=80,in=90] (-0.5,0);
\centerarc[red,very thick](1,0)(0:180:0.5:0.5)
\centerarc[densely dotted,red,very thick](1,0)(180:360:0.5:0.5)
\centerarc[thick,red](0,0)(0:-180:0.5:0.2)
\centerarc[thick,red,densely dotted](0,0)(0:180:0.5:0.2)
\draw[thick,black] (1.5,0) to [out=-80,in=-90] (-0.5,0);
\end{tikzpicture}}

\newcommand{\DVa}{
\begin{tikzpicture}[xscale=0.5,yscale=0.5]
\draw (-4,0) to [out=90,in=90] (-3,0) to [out=90,in=90] (-2,0);
\draw[dotted] (-4,0) to [out=270,in=270] (-3,0) to [out=270,in=270] (-2,0);
\draw[thick] (-2,0) to [out=90, in=180] (0,1);
\draw[thick,red] (-3,0) to [out=90, in=180] (0,2);
\draw[thick] (-4,0) to [out=90, in=180] (0,3);
\draw[thick] (-2,0) to [out=270, in=180] (0,-1) to [out=0, in=190] (0.95,-0.9);
\draw[dotted] (0.95,-0.9) to [out=10, in=270] (2,0);
\draw[thick] (-4,0) to [out=270, in=180] (0,-3);
\draw[thick] (0,1) to [out=0, in=90] (2,0);
\draw[thick] (0,2) to [out=0, in=90] (3,0);
\draw[dotted] (0,-2) to [out=0,in=190] (1.6,-1.8);
\draw[thick] (1.65,-1.8) to [out=15,in=270] (3,0);
\draw[dotted] (0,2) to [out=0,in=100] (0.9,0.9);
\draw[thick] (0.9,0.9) to [out=280,in=90] (1,0) to [out=270,in=0] (0,-2);
\draw[thick] (0,3) to [out=0,in=100] (1.8,1.65);
\draw[dotted] (1.8,1.65) to [out=280,in=90] (2,0);
\draw[thick] (2,0) to [out=270,in=0] (0,-3);
\draw[thick,red] (-3,0) to [out=270, in=180] (0,-2);
\draw[thick,red,densely dotted] (0,2) to [out=0,in=90] (2,0) to [out=270, in=0] (0,-2);
\end{tikzpicture}}

\newcommand{\DVb}{
\begin{tikzpicture}[xscale=0.9,yscale=1.4]
\centerarc[thick,red](-0.5,0)(0:180:1.5:0.8)
\centerarc[densely dotted,thick,red](-0.5,0)(180:360:1.5:0.8)
\centerarc[thick](0,0)(0:360:2:1)
\centerarc[thick](0,0)(-190:10:1:0.4)
\centerarc[thick](0,-0.4)(160:20:0.8:0.4)
\end{tikzpicture}}

\newcommand{\DVc}{
\begin{tikzpicture}[xscale=0.9,yscale=1.4]
\centerarc[thick](0,0)(0:360:2:1)
\centerarc[thick](0,0)(-190:10:1:0.4)
\centerarc[thick](0,-0.4)(160:20:0.8:0.4)
\end{tikzpicture}}

\newcommand{\Bohemianquartet}{
\begin{tikzpicture}[scale=0.5]
\draw[blue] (-2,0)       -- (2,0)      node[right      ] {$R$};
\draw[blue] (0,2)        -- (0,-2)     node[below      ] {$\oR$};
\draw[red] (-1.2,-1.2)  -- (1.5,1.5)  node[above right] {$L$};
\draw[red] (1.2,-1.2)   -- (-1.5,1.5) node[above left ] {$\oL$};
\draw[draw=black, fill=green!20] (0,0) circle [radius=0.5] node[black] {$\vv$};
\end{tikzpicture}}

\newcommand{\Cliffordianquartet}{
\begin{tikzpicture}[scale=0.3]
\draw[blue] (-4, 2) -- ( 4, 2) node[right] {$R$};
\draw[blue] (-4,-2) -- ( 4,-2) node[right] {$\oR$};
\draw[red] (-2, 4) -- (-2,-4) node[below] {$L$};
\draw[red] ( 2, 4) -- ( 2,-4) node[below] {$\oL$};
\draw[draw=black, fill=green!20]    (-2, 2) circle [radius=0.85] node {$\pp$};
\draw[draw=black, fill=blue!10]     ( 2, 2) circle [radius=0.85] node {$\oq$};
\draw[draw=black, fill=orange!40]   (-2,-2) circle [radius=0.85] node {$\qq$};
\draw[draw=black, fill=magenta!10]  ( 2,-2) circle [radius=0.85] node {$\op$};
\end{tikzpicture}}

\newcommand{\PIII}{
\begin{tikzpicture}[scale=0.4]
\draw[thick, red] (-3.5, 2) -- ( 4.5, 2) node[right] {$R$};
\draw[thick, red] (-3.5,-2) -- ( 4.5,-2) node[right] {$\oR$};
\draw[thick, red] (-2, 4) -- (-2,-4) node[below] {$L$};
\draw[thick, red] ( 2, 4) -- ( 2,-4) node[below] {$\oL$};
\draw[thick, blue] (-2,2) to [out=45, in=180] (0,4) to [out=0, in=135] (2,2) to [out=-45, in=90] (4,0) to [out=270, in=45] (2,-2) to [out=225, in=0] (0,-4) to [out=180, in=-45] (-2,-2) to [out=135, in=270] (-4,0) to [out=90, in=225] (-2,2);
\draw[draw=black, fill=white] (-2, 2) circle [radius=0.27];
\draw[draw=black, fill=white] ( 2, 2) circle [radius=0.27];
\draw[draw=black, fill=white] (-2,-2) circle [radius=0.27];
\draw[draw=black, fill=white] ( 2,-2) circle [radius=0.27];
\node[blue] at (5,0) {$V$};
\end{tikzpicture}}

\newcommand{\PIV}{
\begin{tikzpicture}[scale=0.4]
\draw[thick, red] (-3.5, 2) -- ( 3.5, 2) node[right] {$R$};
\draw[thick, red] (-3.5,-2) -- ( 3.5,-2) node[right] {$\oR$};
\draw[thick, red] (-2, 4) -- (-2,-4) node[below] {$L$};
\draw[thick, red] ( 2, 4) -- ( 2,-4) node[below] {$\oL$};
\draw[thick, colGreen] (-2,2) to [out=-5, in=135] (0,1) to [out=-45, in=135] (1,0) to [out=-45, in=95] (2,-2);
\draw[thick, colGreen] (2,-2) to [out=175, in=-45] (0.1,-1.1);
\draw[thick, colGreen] (-0.1,-0.9) to [out=135, in=-45] (-0.9,-0.1 );
\draw[thick, colGreen] (-1.1,0.1) to [out=135, in=275] (-2,2);
\draw[thick, colGray] (-2,-2) to [out=85, in=225] (-1,0) to [out=225, in=225] (0,1) to [out=45, in=185] (2,2);
\draw[thick, colGray] (-2,-2) to [out=5, in=225] (0,-1) to [out=225, in=225] (0.9,-0.1);
\draw[thick, colGray] (1.1,0.1) to [out=45, in=265] (2,2);
\draw[draw=black, fill=white] (-2, 2) circle [radius=0.27];
\draw[draw=black, fill=white] ( 2, 2) circle [radius=0.27];
\draw[draw=black, fill=white] (-2,-2) circle [radius=0.27];
\draw[draw=black, fill=white] ( 2,-2) circle [radius=0.27];
\draw[draw=black, fill=white] ( 0, 1) circle [radius=0.27];
\node[colGray] at (-0.5,-3) {$V_2$};
\node[colGreen] at (-1,3) {$V_1$};
\end{tikzpicture}}

\newcommand{\PV}{
\begin{tikzpicture}[scale=0.4]
\draw[thick, red] (-3.5, 2) -- ( 3.5, 2) node[right] {$R$};
\draw[thick, red] (-3.5,-2) -- ( 3.5,-2) node[right] {$\oR$};
\draw[thick, red] (-2, 4) -- (-2,-4) node[below] {$L$};
\draw[thick, red] ( 2, 4) -- ( 2,-4) node[below] {$\oL$};
\draw[thick, colGreen] (0,2) to [out=180, in=180] (0,-2) to [out=0, in=0] (0,2);
\draw[draw=black, fill=white]  (-2, 2) circle [radius=0.27];
\draw[draw=black, fill=white]  ( 2, 2) circle [radius=0.27];
\draw[draw=black, fill=white]  (-2,-2) circle [radius=0.27];
\draw[draw=black, fill=white]  ( 2,-2) circle [radius=0.27];
\draw[draw=black, fill=white] ( 0, 2) circle [radius=0.27];
\draw[draw=black, fill=white] ( 0,-2) circle [radius=0.27];
\node[colGreen] at (0,0) {$C$};
\end{tikzpicture}}

\newcommand{\DIIIcLoop}{
\begin{tikzpicture}[scale=0.6]
\centerarc[thick,fill=white   ](0,-2)(-180:0:0.3:0.3);
\centerarc[thick,fill=black!10](0,-2)(180:0:0.3:0.3);
\centerarc[thick](-2.3,0)(0:360:0.3:0.3);
\centerarc[thick]( 2.3,0)(0:360:0.3:0.3);
\centerarc[thick](0,2.5)(0:360:0.3:0.3);
\centerarc[very thick, dotted](0,0)(0:360:2:2);
\centerarc[very thick](0,0)(-180:0:2:2);
\draw[fill=red,red] (-2,0)          circle [radius=0.1];
\draw[fill=red,red] (2,0)           circle [radius=0.1];
\draw[fill=red,red] (-0.29,-1.98)   circle [radius=0.1];
\draw[fill=red,red] (0.29,-1.98)    circle [radius=0.1];
\end{tikzpicture}}

\newif\ifhide
\ifhide
\renewcommand{\DI}{D}\renewcommand{\DII}{D}\renewcommand{\DIII}{D}\renewcommand{\DIV}{D}\renewcommand{\DV}{D}
\renewcommand{\DIa}{D}\renewcommand{\DIb}{D}\renewcommand{\DIIa}{D}\renewcommand{\DIIb}{D}
\renewcommand{\DIIIa}{D}\renewcommand{\DIIIb}{D}\renewcommand{\DIIIc}{D}\renewcommand{\DIIId}{D}\renewcommand{\DIIIe}{D}
\renewcommand{\DIVa}{D}\renewcommand{\DIVb}{D}\renewcommand{\DIVc}{D}\renewcommand{\DIVd}{D}
\renewcommand{\DVa}{D}\renewcommand{\DVb}{D}\renewcommand{\DVc}{D}
\renewcommand{\Bohemianquartet}{D}\renewcommand{\Cliffordianquartet}{D}
\renewcommand{\PIII}{D}\renewcommand{\PIV}{D}\renewcommand{\PV}{D}
\renewcommand{\DIIIcLoop}{D}
\renewenvironment{tikzpicture}{D\comment}{\endcomment}
\renewcommand{\fig}[3]{F}
\fi

\begin{document}
\myspace

\begin{center}
\LARGE
~\\[5mm]
Self-intersections of surfaces that contain two circles through each point
\\[5mm]\large
Niels Lubbes
\\[5mm]\large
\today
\end{center}

\begin{abstract}
We classify the singular loci of real surfaces in three-space that contain two circles through each point. We characterize how a circle in such a surface meets this loci as it moves in its pencil and as such provide insight into the topology of the surface.
\\[2mm]
\\[2mm]
%
\end{abstract}

\section{Introduction}
\label{sec:intro}

\subsection{Prologue}
\label{sec:pro}

A surface in $\R^3$ that contains two lines through each point must
be a doubly ruled quadric and thus has no self-intersections.
In particular,
Sir Christopher Wren discovered in \citep[1669]{1669} that the one-sheeted hyperboloid
is doubly ruled and therefore could be used for grinding hyperbolic lenses.
Although a circle is perhaps the most elementary curve after the line,
the following problem is surprisingly difficult:

\textit{What are the possible self-intersections of a surface in $\R^3$
that contains two circles through almost each point?}

In \Cref{fig:intro1}, we see two examples of such self-intersections
together with diagrammatic representations of these loci.
By applying the main result of this article, we find that
the self-intersection loci is in both cases contained in two circles
of multiplicity two that meet at double point (see \Cref{thm:s}).
Moreover, almost each circle in the surface meets each double circle
in one point.

\begin{figure}[!ht]
\centering
\csep{2mm}
\begin{tabular}{ccc}
\fig{2.5}{3.5}{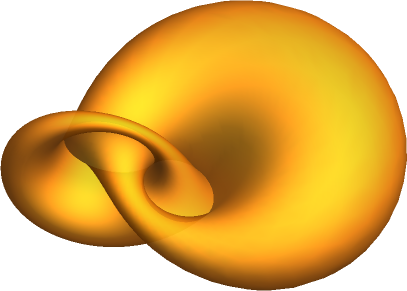} &
\fig{2.5}{3.5}{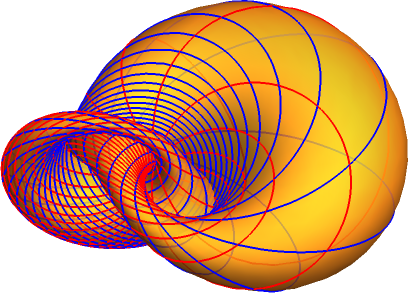} &
\DIVa
\\
\fig{2.5}{3.5}{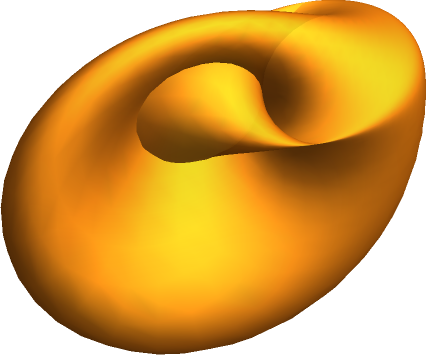} &
\fig{2.5}{3.5}{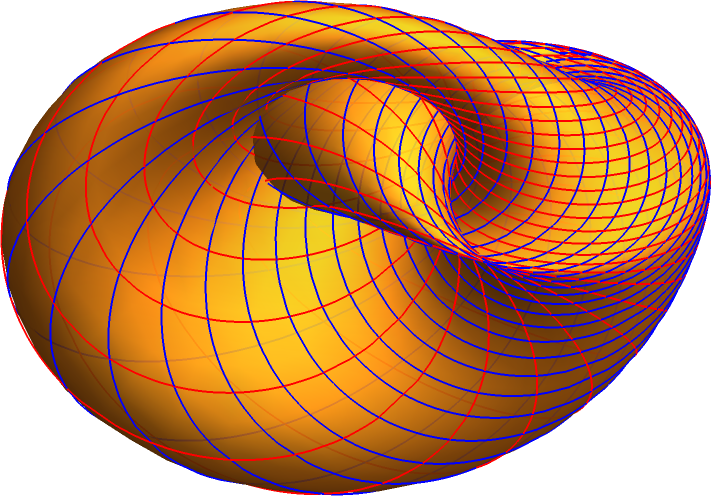} &
\DIVc
\end{tabular}
\caption{Surfaces that contain two circles through almost each point.}
\label{fig:intro1}
\end{figure}

In \Cref{fig:intro2}, we see another example where
the self-intersection locus is an arc that is
contained in a quartic rational curve.
The diagram illustrates how a circle meets this arc as it moves in its pencil.

\begin{figure}[!ht]
\centering
\csep{3mm}
\begin{tabular}{ccc}
\fig{3}{3.5}{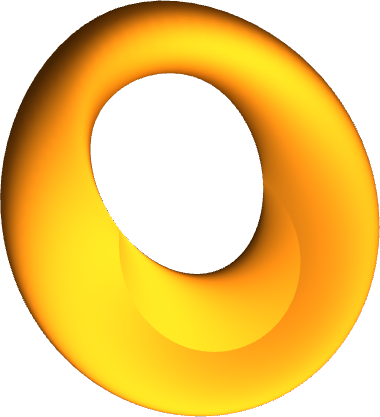}  &
\fig{3}{3.5}{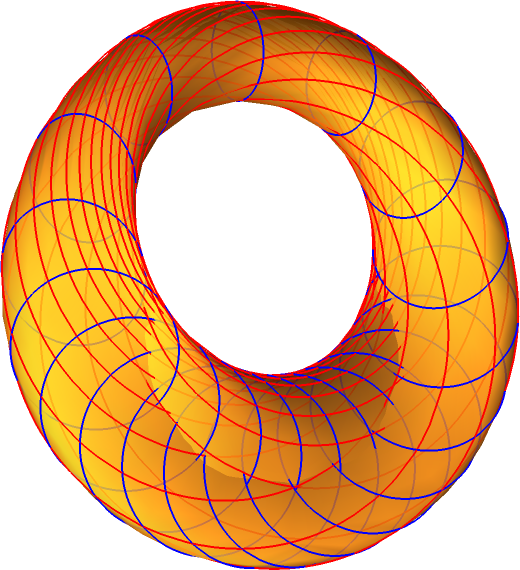}  &
\DIIIcLoop
\end{tabular}
\caption{The self-intersection locus consists of an arc.}
\label{fig:intro2}
\end{figure}

Both surfaces in \Cref{fig:intro3} are disjoint unions of red circles
and contain a double circle.
As a blue circle continuously moves in its pencil,
it either meets the double circle of the left surface in two points, or
the double circle of the right surface tangentially in one point.

\begin{figure}[!ht]
\centering
\csep{1mm}
\begin{tabular}{cccc}
\fig{3}{3.5}{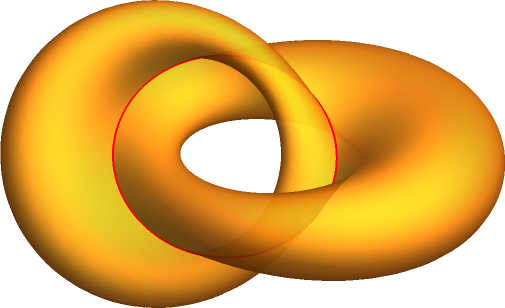} &
\fig{3}{3.5}{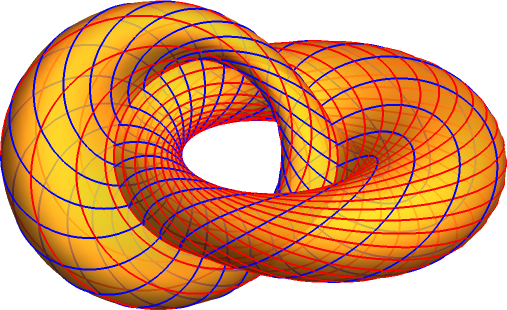} &
\fig{3}{3.5}{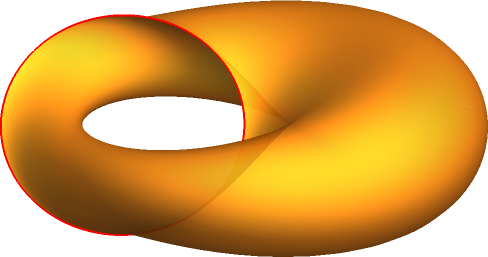} &
\fig{3}{3.5}{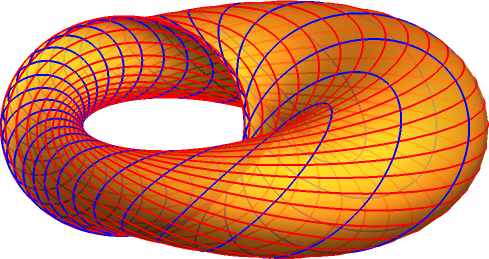}
\end{tabular}
\caption{Both surfaces that are disjoint unions of red circles. A blue circle
meets the double circle in two points (left) or one point tangentially (right).}
\label{fig:intro3}
\end{figure}

We will show that the self-intersection locus of a surface that contains two circles
through each point is homeomorphic to the symbol
$=$, $-$, $+$, $\circ$, $\infty$, $\alpha$ or $\sigma$.
More examples with explicit parametrizations are provided at \Cref{tab:D,tab:D123,tab:D45}.
We discuss some related literature in \SEC{state}
and present our main results in~\SEC{main}.
In \SEC{overview}, we conclude this introduction with the overview
of the article.

\subsection{Historical and recent developments}
\label{sec:state}

Analytic surfaces that contain at least two circles through almost each point are called \df{celestial surfaces}.
Schicho shows in \citep[Theorem~1]{2001} that such surfaces must be algebraic and thus the zero set of
polynomials with real coefficients.

Celestial surfaces are of interest in
differential geometry \cite{1995,2000},
discrete differential geometry \cite{2012ort,2017sup},
geometric modeling~\cite{2012web},
kinematics~\cite{2018kin,2020kin},
computer vision~\cite{2023vision} and
architecture~\cite{2011arc}.

If $Z$ is a celestial surface in the unit-sphere~$S^3$,
then its degree is either 2, 4 or~8, and
contains at most~6 circles through a general point \cite{2021circle,1980,2000}.
If $Z$ has degree~4, then it is a \df{Darboux cyclide}
and admits at most isolated singularities \cite{2021circle}.
The investigation of such cyclides can be traced back
to Dupin~\citep[1822]{1822}.

It follows from the recent result \citep[Theorem~1.1]{2019sko}
by Skopenkov and Krasauskas that celestial surfaces of degree~8
are up to M\"obius equivalence and stereographic projection either
the pointwise sum of circles in $\R^3$
or
the pointwise product of circles in the unit-quaternions~$S^3$.
Geometrically such surfaces are the translation of a circle along another circle
with respect to the Euclidean or spherical metric.
The points in $\R^3$ or $S^3$ that during such a translation are reached more than one time
by the translated circle correspond to the self-intersection locus.
We remark that translations of $S^3$ are also known as \df{Clifford translations}
or \df{isoclinic rotations}.
Clifford translations are closely related to the factorization problem for
bivariate quaternionic polynomials (see \citep[Lemma~2.9]{2019sko}).
These results from \cite{2019sko} are,
together with Frischauf and Schr\"ocker, generalized in \cite{2024fact}.

From an algebro geometric point of view, the current article
investigates real birational morphisms from $\P^1\times\P^1$
into a hyperquadric in~$\P^4$ of signature~$(1,4)$
\st its components are of bidegree~$(2,2)$.
We classify such biquadratic maps in terms of the singular loci of their images.

A classical counterpart to this result is the classification of
images of birational morphisms from $\P^2$ into $\P^3$
\st the components of the maps are of degree~$2$.
Such an image is called a \df{Steiner surface} and its singular locus
consists of either one, two or three complex double lines that each pass through
a triple point (see \cite{1996}, \citep[Section~2.1.1]{2012dol} and \citep[VI.\textsection 46, page~303]{1990}).
In case all three double lines are real it is called a \df{Roman surface}
as Jacob Steiner discovered some of its properties during a stay at Rome in~$1844$.
Steiner surfaces admit pinch points
and thus, unlike \df{Boy's surfaces}, do not define an immersion of the real points of~$\P^2$ into $\R^3\subset \P^3$
(see \cite{1987}).

Recently, Koll\'ar provides in \citep[Theorem~3]{2018kol} a birational classification
of varieties whose points correspond to
morphisms from~$\P^2$ into a 4-dimensional hyperquadric whose components have degree~$2$.
As an application, \citep[Theorem~23]{2018kol} states that
a surface in a unit-sphere~$S^n$ that contains a circle through any two points,
is either~$S^2$ or a quartic Veronese surface in $S^4$ that is not contained in a hyperplane.

The classification of complex surfaces in $\P^3$ that are singular along curves
is completed for cubic surfaces by Bruce and Wall~\cite{1979},
and for quartic surfaces by Urabe~\cite{1986}.
Piene provides in~\cite{2005} upper bounds for the degree of the singular loci
of rational surfaces in~$\P^3$.

From a topological point of view, we consider continuous maps $S^1\times S^1\to S^3$
\st the parameter lines are circles.
Their images may have pinch points and thus are not immersions.
In case the image is covered by great circles, then its self-intersection locus
and topological type is classified in \citep[Theorem~I and Proposition~25]{2024great}.
The classification of immersions of compact surfaces into $\R^3$ up to regular homotopy
is done by Pinkall in~\citep[Theorem~4]{1985}.

\subsection{Presentation of the main results}
\label{sec:main}

In order to investigate the self-intersection loci of surfaces,
we investigate curves at complex infinity.
To uncover these hidden curves we
define a \df{real variety} $X$ to be a complex irreducible variety together with
an antiholomorphic involution $\sigma_X\c X\to X$
called the \df{real structure} of $X$.
We denote its real points by
\[
X_\R:=\set{p\in X}{\sigma_X(p)=p}.
\]
Such varieties can always be defined by polynomials with real coefficients.

\begin{notation}
\label{ntn:real}
In what follows, points, curves, surfaces and projective spaces $\P^n$ are real algebraic varieties
and maps between such varieties
are compatible with their real structures
unless explicitly stated otherwise.
Conics are real and reduced by default, but may be reducible.
By default, we assume that the real structure $\sigma_{\P^n}\c\P^n\to\P^n$ sends $x$
to $(\overline{x_0}:\ldots:\overline{x_n})$,
where $\overline{\,\cdot\,}$ denotes the complex conjugate.
\END
\end{notation}

The \df{M\"obius quadric} is defined as
\[
\S^3:=\set{x\in\P^4}{-x_0^2+x_1^2+x_2^2+x_3^2+x_4^2=0}.
\]
A \df{circle}~$C\subset\S^3$ is an irreducible conic \st $C_\R\neq\varnothing$.
A \df{Viviani curve}~$C\subset\S^3$ is a quartic rational curve with one singular point.
A \df{twisted quartic}~$C\subset\S^3$ is a quartic rational normal curve.

We call a surface $X\subset\S^3$ \df{$\lambda$-circled}
if it contains at least $\lambda\in \Z_{\geq 0}\cup\{\infty\}$ circles through a general point.
If $\lambda\in\Z_{\geq0}$, then we assume that $X$ is not $(\lambda+1)$-circled.
If $\lambda\geq 2$, then we call $X$ \df{celestial}.

A \df{singular component} of a surface $X\subset\P^n$
is a complex irreducible component
of the \df{singular locus}~$\Sing X$.
We call a 1-dimensional singular component a (complex) \df{double curve} if a general complex point
in this (complex) curve has multiplicity two.
In this article, $|\cdot|$ always denotes the set-theoretic cardinality.

A complex curve $W\subset \P^1\times\P^1$ has \df{bidegree}~$(\alpha,\beta)$
if
almost all fibers of the projection of $\P^1\times\P^1$ to its second $\P^1$ component
intersects $W$ in $\alpha\in \Z_{\geq 0}$ complex points
and
almost all fibers of the projection of $\P^1\times\P^1$ to its first $\P^1$ component
intersects $W$ in $\beta\in \Z_{\geq 0}$ complex points.
Thus, $W$ can be defined as the zero set of an irreducible polynomial in $\C[x_0,x_1,y_0,y_1]$ that is
homogeneous and of degree~$\alpha$ in the variables $x_0,x_1$ and
homogeneous and of degree~$\beta$ in the variables $y_0,y_1$.

\begin{definition}
\label{def:class}
Suppose that $X\subset\S^3$ is a surface and $\varphi\c \P^1\times\P^1\to X$
a birational morphism. Let $N(X):=\bas{\l_0,\l_1}_\Z$ be an additive group
together with the unimodular intersection product~$\_\cdot\_\c N(X)\times N(X)\to\Z$ defined by
\[
\l_0^2=\l_1^2=\l_0\cdot \l_1-1=0.
\]
The class~$[C]\in N(X)$ \wrt $\varphi$ of a curve or 1-dimensional singular component
is defined as $\alpha\,\l_0+\beta\,\l_1$
if $\varphi^{-1}(C)$ has bidegree~$(\alpha,\beta)$.
\END
\end{definition}

\begin{definition}[singular type]
\label{def:st}
The \df{singular type}~$\SingType X$ of a surface~$X\subset\S^3$
is defined as $\Di$, $\Dii$, $\Diii$, $\Div$ or $\Dv$,
if each of the following five properties holds (and $\SingType X:=$``Undefined'' otherwise):
\begin{enumerate}[topsep=0pt,itemsep=0pt,leftmargin=8mm,label=(\roman*),ref=(\roman*)]
\item\label{i}
The surface $X$ is 2-circled and of degree eight.

\item\label{ii}
The incidence relations of the singular components of~$X$
are characterized by the corresponding diagram in \Cref{tab:s},
where
\begin{Mlist}
\item a line segment represents a non-real double line,
\item a loop represents a double curve with specified name, and
\item a disc represents the complex incidence point between the corresponding
complex double curves.
\end{Mlist}
\begin{table}[!ht]
\caption{Incidences between singular components for each singular type.}
\label{tab:s}
\centering
\csep{3mm}
\begin{tabular}{@{}ccccc@{}}
{\bf D1} & {\bf D2} & {\bf D3} & {\bf D4} & {\bf D5}\\
\DI & \DII & \DIII & \DIV & \DV
\\[-2mm]
Viviani curve & two circles & twisted quartic & two circles & conic
\end{tabular}
\end{table}

\item\label{iii}
If~$p\in \Sing X$ is a complex incidence point between singular components
\st $p$ has multiplicity $m$ and lies on $r$ complex double lines,
then
\[
(m,\,r,\,|X_\R\cap \{p\}|)\in \{(4,4,1),\,(3,2,0),\,(2,1,0),\,(2,0,1)\}.
\]

\item\label{iv}
There exists a birational morphism~$\P^1\times\P^1\to X$ \st the following holds.
\begin{Mlist}
\item A hyperplane section of $X$ has class $2\,\l_0+2\,\l_1$.
\item The set of classes of circles in $X$ containing a general point is $\{\l_0,\l_1\}$.
\item A complex double line in $X$ represented by a vertical or horizontal line segment
has class $\l_0$ and $\l_1$, \resp.
\item The double Viviani curve at~$\Di$ has class $2\,\l_0+2\,\l_1$.
\item The double circles at~$\Dii$ and~$\Div$ each have class $\l_0+\l_1$.
\item The double twisted quartic at~$\Diii$ has class $2\,\l_0+2\,\l_1$.
\item The double conic at~$\Dv$ has class either $2\,\l_0$ or $\l_0$.
\end{Mlist}

\item\label{v}
If $C\subset X$ is a circle containing a general point and
$V\subset X$ a singular component, then $|C\cap V|=[C]\cdot [V]$.
\END
\end{enumerate}
\end{definition}

\begin{theorem}
\label{thm:s}
Suppose that $X\subset\S^3$ is a celestial surface.
\begin{claims}
\item\label{thm:s:a}
$\deg X\in\{2,4,8\}$.
\item\label{thm:s:b}
If $\deg X=4$, then $|\Sing X|\leq 4$ and $|\Sing X_\R|\leq 2$.
\item\label{thm:s:c}
If $\deg X=8$, then $\SingType X\in\{\Di,\,\Dii,\,\Diii,\,\Div,\,\Dv\}$.
\end{claims}
\end{theorem}

\begin{remark}
\label{rmk:s}
\Cref{thm:s}\ref{thm:s:c} is the main contribution of the current article
and builds on \cite{2001} (see \Cref{thm:mc}).
Theorems~\ref{thm:s}\ref{thm:s:a} and \ref{thm:s}\ref{thm:s:b}
follow from
\citep[Theorem~1 and Corollary~5]{2021circle}.
\END
\end{remark}

The (complex) \df{stereographic projection}~$\pr_p\c \S^3\dto \P^3$
is defined as the complex linear projection from (complex) center~$p\in\S^3$.

The \df{Euclidean model}~$\bU(X)\subset\R^3$ of $X\subset\S^3$ is defined as~$(\iota\circ\pr_\vv)(X_\R)$,
where $\vv=(1:0:0:0:1)$ and
$\iota\c\P^3_\R\dto \R^3$ sends $(y_0:y_1:y_2:y_3)$ to $(y_1,y_2,y_3)/y_0$.
It was already known to Hipparchus (190--120~BCE) that the Euclidean model of a circle
is either a circle or a line.
We call $X\subset\S^3$ \df{Bohemian} if there exists circles $A,B\subset\R^3$ \st
$\bU(X)$ is the Zariski closure of
\[
\set{a+b}{a\in A,~b\in B}.
\]
The \df{spherical model}~$\bS(X)$ in the 3-dimensional \df{unit-sphere}~$S^3\subset\R^4$ is defined
as~$\iota(X_\R)$, where
$\iota\c\S^3_\R\to S^3$ sends $(y_0:y_1:y_2:y_3:y_4)$ to $(y_1,y_2,y_3,y_4)/y_0$.
We call $X\subset\S^3$ \df{Cliffordian} if there exists circles $A,B\subset S^3$ \st
\[
\bS(X)=\set{a\star b}{a\in A,~b\in B},
\]
where we identified $S^3$ with the unit-quaternions and
$\_\star\_\c S^3\times S^3\to S^3$
denotes the \df{Hamiltonian product}.
We call $X\subset\S^3$ \df{great} if $\bS(X)$ contains a great circle through a general point.

The \df{M\"obius transformations}~$\aut\S^3$
are defined as the projective transformations of~$\P^4$
that leave~$\S^3$ invariant.
If $X,X'\subset\S^3$ are M\"obius equivalent,
then there exists an angle preserving rational map~$\R^3\dto \R^3$
that sends $\bU(X)$ to $\bU(X')$.

\begin{corollary}
\label{cor:BC}
If $X\subset \S^3$ is a celestial surface \st $\deg X\notin\{2,4\}$,
then the following holds:
\begin{Mlist}
\item
$\SingType X\in\{\Di,\,\Dii,\,\Diii,\,\Div,\,\Dv\}$.
\item
$\SingType X\in\{\Di,\,\Dii\}$ if and only if
$X$ is M\"obius equivalent to a Bohemian surface.
\item
$\SingType X\in\{\Diii,\,\Div,\,\Dv\}$ if and only if
$X$ is M\"obius equivalent to a Cliffordian surface.
\item
If $X$ is M\"obius equivalent to a great surface, then $\SingType X=\Dv$.
\item
If $C\subset X$ is a circle \st $C_\R\cap\Sing X=\varnothing$, then $\SingType X\notin\{\Dii,\Div\}$.
\end{Mlist}
\end{corollary}

\begin{conjecture}
If $\SingType X=\Dv$,
then $X$ is M\"obius equivalent to a great surface.
\end{conjecture}

\begin{remark}
\Cref{cor:BC}
depends on \citep[Theorem~II]{2024fact} and \citep[Proposition~25]{2024great}.
This corollary refines \citep[Theorem~1.1]{2019sko} by Skopenkov and Krasauskas:
non-quartic celestial surfaces
are up to M\"obius equivalence either Bohemian or
Cliffordian.
\END
\end{remark}

\begin{corollary}
\label{cor:deg}
If $X\subset \S^3$ is a celestial surface, then $\deg\bU(X)\in\{1,2,3,4,6,7,8\}$.
\end{corollary}

\begin{example}
\label{exm:deg}
Suppose that $X\subset\S^3$ is the Bohemian celestial surface \st
its Euclidean model~$\bU(X)$ is parametrized by $(\cos\alpha,\sin\alpha,0)+(\cos\beta,0,\sin\beta)$
with $0\leq\alpha,\beta<2\pi$.
It follows from \Cref{cor:BC} that $\SingType X\in\{\Di,\Dii\}$.
We know from \ref{iii} at \Cref{def:st} the multiplicities of the complex points in $\Sing X$.
See \Cref{fig:deg} for renderings of $\bU(\gamma_m(X))$,
where
$m\in\{0,1,2,4\}$ and
$\gamma_m\in\aut\S^3$ is a M\"obius transformation
\st $(1:0:0:0:1)$ has multiplicity~$m$ in~$\gamma_m(X)$.
Notice that $\deg\bU(\gamma_m(X))=8-m$.
Since $\bU(\gamma_4(X))$ contains two lines,
we find that $\SingType X=\Dii$.
\END
\end{example}

\begin{figure}[!ht]
\centering
\csep{3.5mm}
\begin{tabular}{cccc}
\fig{3}{3}{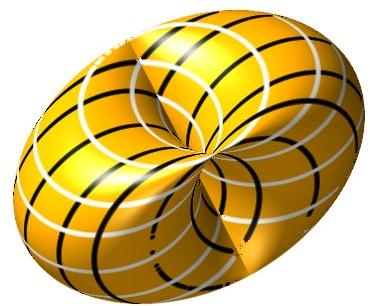} &
\fig{3}{3}{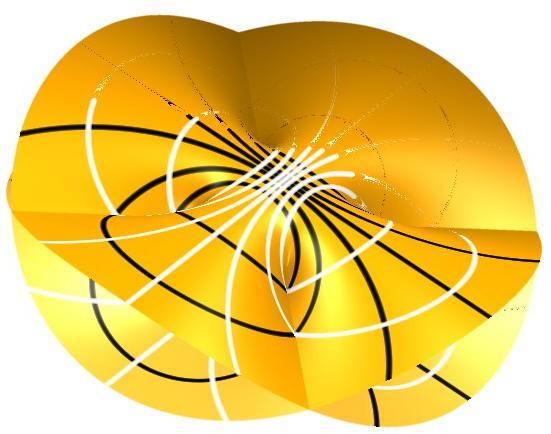} &
\fig{3}{3}{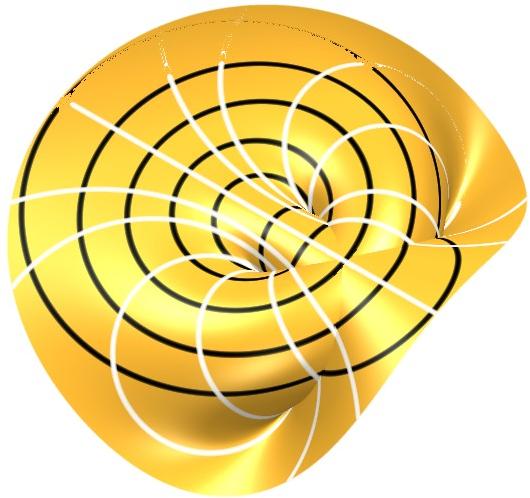} &
\fig{3}{3}{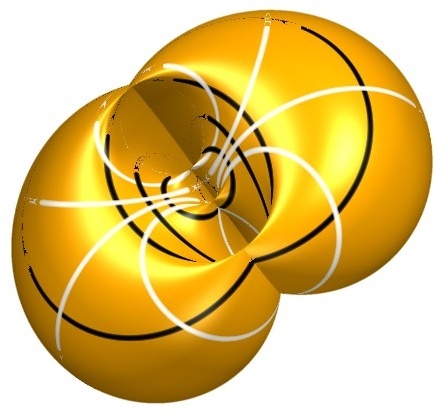}
\\
degree 4 & degree 6 & degree 7 & degree 8
\end{tabular}
\caption{M\"obius transformations of a Bohemian dome of singular type~$\Dii$.}
\label{fig:deg}
\end{figure}

We call a point in a surface \df{visible} if
it is real and non-isolated in almost all real hyperplane sections
that contain this point.
The \df{visible singular locus} of a surface is defined
as the set of visible points in the singular locus.
We remark that an end point of an arc in the visible singular locus
is a \df{pinch point}.
The visible points of a celestial surface $X\subset\S^3$
is obtained by taking the closure of each circle in $X\setminus\Sing X$.

The \df{visible type} $\VisType X$ of a surface~$X\subset\S^3$
is defined as $\SingType X$ together with a symbol
\st the visible singular locus of $X$ is homeomorphic to this symbol.
The symbol is encapsulated by the left square bracket~$[$ and right square bracket~$]$.
No symbol means that the visible singular locus is the empty set.
In \Cref{exm:deg}, $\VisType X=\Dii[+]$ because $\SingType X=\Dii$ and the visible singular locus
of $X$ is homeomorphic to two line segments that form a cross.
If for example $\VisType X=\Dv[~]$, then $X$ has no visible singular locus.
We remark that the symbols $\alpha$ and $\sigma$ are homeomorphic to the set
$\set{(t^2,t^3-t)\in\R^2}{-11+c\leq 10\,t\leq 11}$ where $c=0$ and $c=1$, \resp.

\begin{corollary}
\label{cor:VisType}
If $X\subset\S^3$ is a celestial surface \st $\deg X\notin\{2,4\}$,
then
\begin{gather*}
\VisType X
\in
\{
\Di[=],\,
\Dii[+],\,
\Diii[~],\,
\Diii[\circ],\,
\Diii[=],\,
\Diii[-],\,
\\
\Div[\infty],\,
\Div[\alpha],\,
\Div[\sigma],\,
\Div[+],\,
\Div[=],\,
\Div[-],\,
\Dv[~],\,
\Dv[\circ]
\}.
\end{gather*}
\end{corollary}

\begin{conjecture}
\label{cnj:VisType}
If $X\subset\S^3$ is a celestial surface, then
\[
\VisType X\notin\{\Div[\sigma],\,\Div[=],\,\Div[-]\}.
\]
\end{conjecture}

\begin{example}
The surfaces in \Cref{fig:intro1} at \SEC{pro}
have visible types $\Div[\infty]$ and $\Div[\alpha]$,
and the surface in \Cref{fig:intro2} has visible type $\Diii[-]$.
Both surfaces in \Cref{fig:intro3} have visible types $\Dv[\circ]$,
but the class of the double circle is $2\,\l_0$ on the left,
and $\l_0$ on the right. The classes of the red and blue circles
are $\l_0$ and $\l_1$, \resp.
\END
\end{example}

\begin{definition}
\label{def:pmz}
The \df{parametric type} of a celestial surface $X\subset\S^3$ is defined as tuple
\[
(\tau_0,\ldots,\tau_3;\mu_0,\ldots,\mu_3)\in\R^4\times\R^4,
\]
which represents the rational map $\psi\c [0,2\pi]^2\dto \bU(X)\subset \R^3$ that is defined as
\[
(\alpha,\beta)\mapsto
(
a\,e - b\,f - c\,g - d\,h,~
a\,f + b\,e + c\,h - d\,g,~
a\,g - b\,h + c\,e + d\,f
)/q,
\]
where $q:=- a\,h - b\,g + c\,f - d\,e + (d + 1) (h + 1)$ and
\[
\setlength{\arraycolsep}{1mm}
\begin{array}{cccc}
a := \tau_1 + \tau_0 \cos\alpha,&
b := \tau_2 + \tau_0 \sin\alpha,&
c := \tau_3,                    &
d := (a^2 + b^2 + c^2 - 1)/2,
\\
e := \mu_1 + \mu_0 \cos\beta,   &
f := \mu_2 + \mu_0 \sin\beta,   &
g := \mu_3,                     &
h := (e^2 + f^2 + g^2 - 1)/2.
\\[2mm]
\end{array}
\]
It follows from \Cref{rmk:pmz} that the parameter lines of $\psi$
are circles.
\END
\end{definition}

The purpose of parametric types is to encode explicit examples
of celestial surfaces.
In \Cref{tab:D} is an overview of examples
of singular, visible and parametric types of celestial surfaces of degree~$8$.
See \SEC{id} for more details and \cite{2024celest} to
\href{https://github.com/niels-lubbes/celestial-surfaces#experiment-with-cliffordian-surfaces}{experiment}.

\begin{table}[!ht]
\caption{Visible and parametric types of celestial surfaces of degree eight.
The examples at the first two rows correspond to the surface parametrized by
$
(\cos\alpha,\sin\alpha,0)+(r\cdot \cos\beta,0,r\cdot\sin\beta).
$
The last three rows correspond to great surfaces.
}
\label{tab:D}
\centering
$
\begin{array}{l@{\hspace{1cm}}l@{}r@{~}r@{~}r@{~}r@{~~}r@{~}r@{~}r@{~}r@{\hspace{1cm}}l}
\Di[=]         &  &              &r=            & 1   &     &              &             &              &       & \text{\Cref{tab:D123}:D1a} \\\hdashline

\Dii[+]        &  &              &r=            & 2   &     &              &             &              &       & \text{\Cref{tab:D123}:D2a} \\\hdashline

\Diii[\circ]   & (&1           , & -1         , & 0, & 0 ; & 1          , & 0          , & \frac{3}{2}, & 0 )   & \text{\Cref{tab:D123}:D3a} \\
\Diii[=]       & (&\frac{1}{10}, & \frac{1}{2}, & 0, & 0 ; & \frac{1}{2}, & 0          , & \frac{1}{2}, & 0 )   & \text{\Cref{tab:D123}:D3b} \\
\Diii[-]       & (&\frac{3}{10}, & \frac{1}{2}, & 0, & 0 ; & \frac{1}{2}, & 0          , & \frac{1}{2}, & 0 )   & \text{\Cref{tab:D123}:D3c} \\
\Diii[~]       & (&1           , & -4         , & 0, & 0 ; & 1          , & 0          , & 0          , & -1)   & \text{\Cref{tab:D123}:D3d} \\
\Diii[-]       & (&\frac{1}{5} , & \frac{1}{2}, & 0, & 0 ; & \frac{1}{2}, & 0          , & \frac{1}{2}, & 0 )   & \text{\Cref{tab:D123}:D3e} \\\hdashline

\Div[\infty]   & (&1           , & 1          , & 0, & 0 ; & 1          , & 0          , & 1          , & 0 )   & \text{\Cref{tab:D45}:D4a}  \\
\Div[+]        & (&\frac{1}{2} , & 0          , & 1, & 0 ; & \frac{1}{2}, & 1          , & 0          , & 0 )   & \text{\Cref{tab:D45}:D4b}  \\
\Div[\alpha]   & (&\frac{1}{2} , & \frac{1}{2}, & 0, & 0 ; & \frac{1}{2}, & 0          , & \frac{1}{2}, & 0 )   & \text{\Cref{tab:D45}:D4c}  \\
\Div[\infty]   & (&1           , &-\frac{3}{2}, & 0, & 0 ; & 1          , & 0          , &-\frac{3}{2}, & 0 )   & \text{\Cref{tab:D45}:D4d}  \\\hdashline

\Dv[\circ]     & (&1           , & 0          , & 0, & 0 ; & 1          , & \frac{3}{2}, & 0          , & 0 )   & \text{\Cref{tab:D45}:D5a}  \\
\Dv[\circ]     & (&1           , & 0          , & 0, & 0 ; & 1          , & 2          , & 0          , & 0 )   & \text{\Cref{tab:D45}:D5b}  \\
\Dv[~]         & (&1           , & 0          , & 0, & 0 ; & 1          , & \frac{5}{2}, & 0          , & 0 )   & \text{\Cref{tab:D45}:D5c}  \\
\end{array}
$
\end{table}

\begin{table}[p]
\centering
\caption{See \Cref{tab:D} and \SEC{id}.}
\label{tab:D123}
\csep{4mm}
\begin{tabular}{ccccc}
\raisebox{15mm}[0pt][0pt]{\bf D1a} &
\fig{3}{3.5}{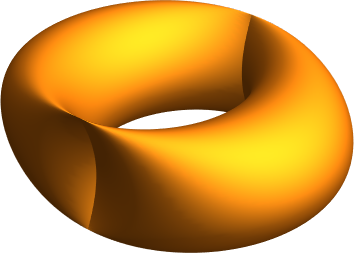} &
\fig{3}{3.5}{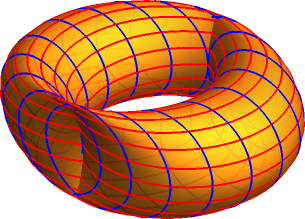} &
\DIa
\\
\raisebox{15mm}[0pt][0pt]{\bf D2a} &
\fig{3}{3.5}{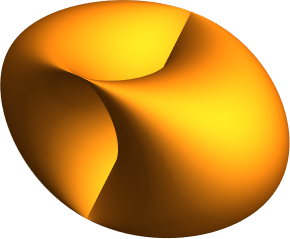} &
\fig{3}{3.5}{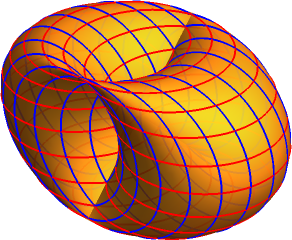} &
\DIIa
\\
\raisebox{15mm}[0pt][0pt]{\bf D3a} &
\fig{3}{3.5}{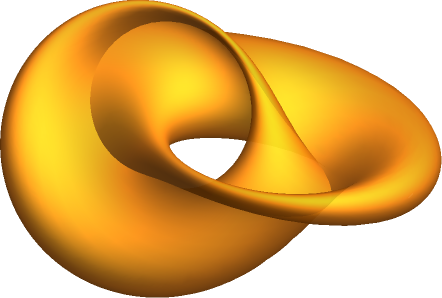}   &
\fig{3}{3.5}{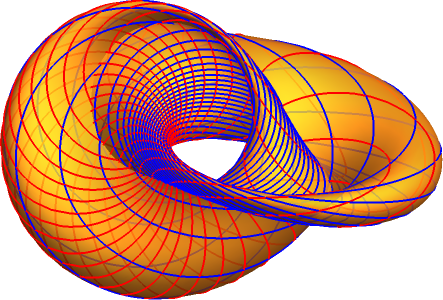}   &
\DIIIa
\\
\raisebox{15mm}[0pt][0pt]{\bf D3b} &
\fig{3}{3.5}{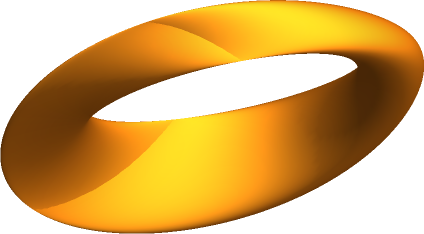} &
\fig{3}{3.5}{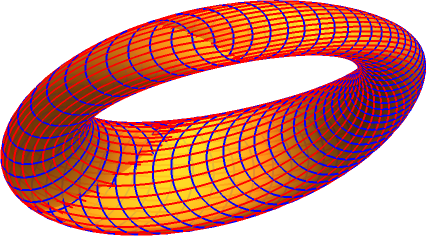} &
\DIIIb
\\
\raisebox{15mm}[0pt][0pt]{\bf D3c} &
\fig{3}{3.5}{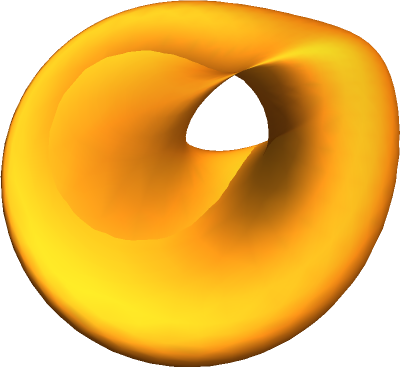}  &
\fig{3}{3.5}{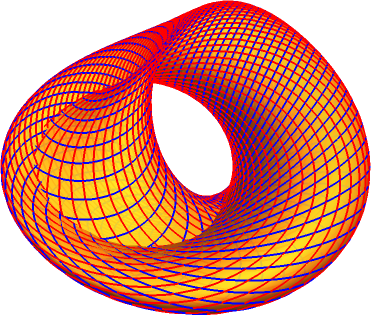}  &
\DIIIc
\\
\raisebox{15mm}[0pt][0pt]{\bf D3d} &
\fig{3}{3.5}{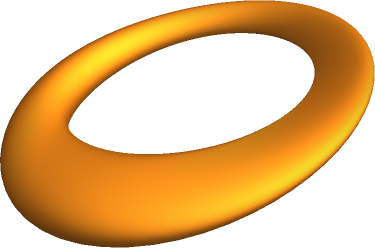} &
\fig{3}{3.5}{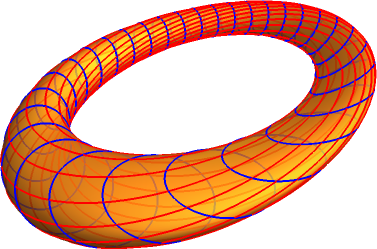} &
\DIIId
\\
\raisebox{15mm}[0pt][0pt]{\bf D3e} &
\fig{3}{3.5}{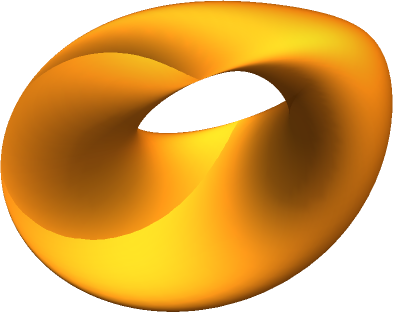} &
\fig{3}{3.5}{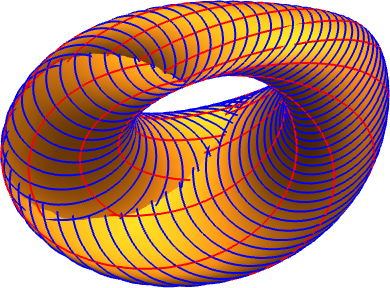} &
\DIIIe
\end{tabular}
\end{table}

\begin{table}[p]
\centering
\caption{See \Cref{tab:D} and \SEC{id}.}
\label{tab:D45}
\csep{4mm}
\begin{tabular}{ccccc}
\raisebox{10mm}[0pt][0pt]{\bf D4a} &
\fig{3}{3.5}{D4a-1___0-0-0_1-0-0_1___0-0-0_0-1-0_1} &
\fig{3}{3.5}{D4a-2___0-0-0_1-0-0_1___0-0-0_0-1-0_1} &
\DIVa
\\
\raisebox{10mm}[0pt][0pt]{\bf D4b} &
\fig{3}{3.5}{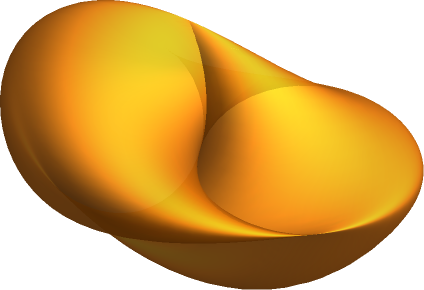} &
\fig{3}{3.5}{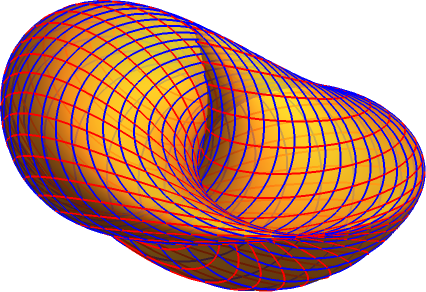} &
\DIVb
\\
\raisebox{10mm}[0pt][0pt]{\bf D4c} &
\fig{3}{3.5}{D4c-1___0-0-0_05-0-0_05___0-0-0_0-05-0_05} &
\fig{3}{3.5}{D4c-2___0-0-0_05-0-0_05___0-0-0_0-05-0_05} &
\DIVc
\\
\raisebox{15mm}[0pt][0pt]{\bf D4d} &
\fig{3}{3.5}{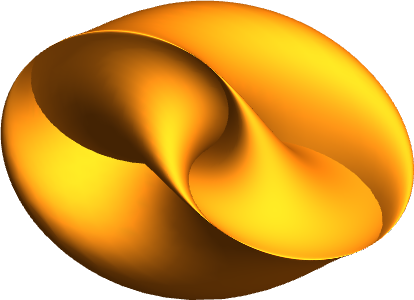} &
\fig{3}{3.5}{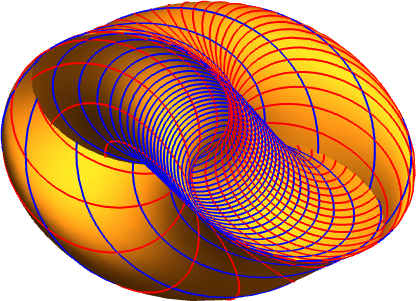} &
\DIVd
\\
\raisebox{10mm}[0pt][0pt]{\bf D5a} &
\fig{3}{3.5}{D5a-1___0-0-0_0-0-0_1___0-0-0_15-0-0_1} &
\fig{3}{3.5}{D5a-2___0-0-0_0-0-0_1___0-0-0_15-0-0_1} &
\DVa
\\
\raisebox{10mm}[0pt][0pt]{\bf D5b} &
\fig{3}{3.5}{D5b-1___0-0-0_0-0-0_1___0-0-0_2-0-0_1}    &
\fig{3}{3.5}{D5b-2___0-0-0_0-0-0_1___0-0-0_2-0-0_1}    &
\DVb
\\
\raisebox{10mm}[0pt][0pt]{\bf D5c} &
\fig{3}{3.5}{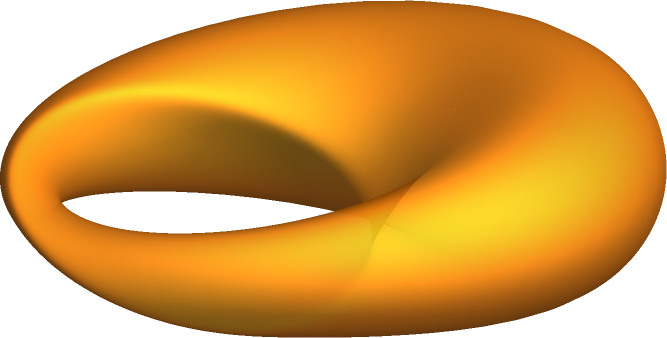} &
\fig{3}{3.5}{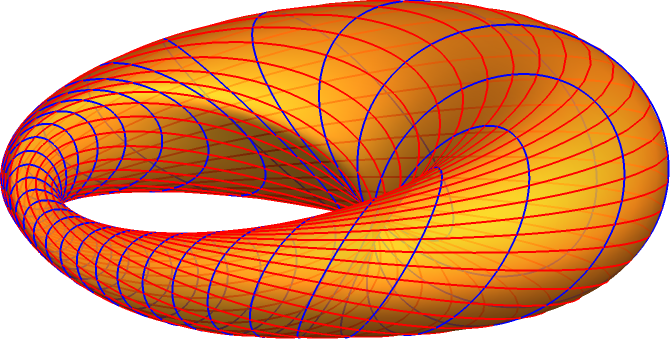} &
\DVc
\end{tabular}
\end{table}

\begin{remark}
\label{rmk:pmz}
The rational map $\psi\c [0,2\pi]^2\dto \R^3$
associated to a parametric type
\[
(\tau_0,\ldots,\tau_3;\mu_0,\ldots,\mu_3)\in\R^4\times\R^4
\]
is constructed as follows.
The stereographic projection $f\c S^3\dto \R^3$
with center~$(0,0,0,1)$
and its inverse are defined as follows:
\begin{gather*}
f\c y\mapsto (y_1,y_2,y_3)/(1-y_4),
\\
f^{-1}\c x\mapsto
(
2\,x_1,
2\,x_2,
2\,x_3,
x_1^2 + x_2^2 + x_3^2 - 1)/(x_1^2 + x_2^2 + x_3^2 + 1).
\end{gather*}
The Hamiltonian product $\_\star\_\c S^3\times S^3\to S^3$ sends $(u,v)$ to
\begin{gather*}
(
u_1\,v_1 - u_2\,v_2 - u_3\,v_3 - u_4\,v_4,~
u_1\,v_2 + u_2\,v_1 + u_3\,v_4 - u_4\,v_3,\\
u_1\,v_3 + u_3\,v_1 + u_4\,v_2 - u_2\,v_4,~
u_1\,v_4 + u_4\,v_1 + u_2\,v_3 - u_3\,v_2
).
\end{gather*}
We consider the following parametrizations of circles in $S^3$:
\begin{gather*}
A(\alpha):=f^{-1}(\tau_0\cos\alpha+\tau_1,~\tau_0\sin\alpha+\tau_2,~\tau_3),
\\
B(\beta):=f^{-1}(\mu_0\cos\beta+\mu_1,~\mu_0\sin\beta+\mu_2,~\mu_3).
\end{gather*}
The map $\psi$ is now defined as
$(\alpha,\beta)\mapsto f\bigl(A(\alpha)\star B(\beta)\bigr)$.
Thus, for all angles $0\leq\alpha,\beta<2\pi$
both $\psi(\alpha,\cdot)$ and $\psi(\cdot,\beta)$ parametrize circles.
See \cite{2024celest} for a
\href{https://github.com/niels-lubbes/celestial-surfaces#parametric-type}{verification}
that this map is equal to the map~$\psi$ at \Cref{def:pmz}.
\END
\end{remark}

\subsection{Overview and table of contents}
\label{sec:overview}

Suppose that~$X\subset\S^3$ is a
celestial surface \st $\deg X\notin\{2,4\}$.
Such a surface must be of degree eight
and a linear projection of a del Pezzo surface.
In~\SEC{smooth}, we extract
some results about del Pezzo surfaces
from the existing literature.
In \SEC{pr} we consider
complex linear projections of~$X$ into~$\P^3$
in order to classify
concurrent complex double lines
and the multiplicities of singular points in~$X$.
We show in~\SEC{pseudo} under some technical hypothesis
that $X$ has singular type either~$\Diii$, $\Div$ or~$\Dv$.
In \SEC{part2}, we prove the main result~\Cref{thm:s} and \Cref{cor:BC,cor:deg} by showing that either the hypothesis for \SEC{pseudo}
holds, or that $X$ has singular type~$\Di$ or~$\Dii$.
We prove \Cref{cor:VisType} in \SEC{VisType}.
In \SEC{id}, we present computational methods for
determining $\SingType X$ and $\VisType X$.

\makeatletter
\renewcommand{\@cftmaketoctitle}{}
\makeatother
\vspace{2mm}
\begingroup
\def\addvspace#1{\vspace{-1mm}}
\tableofcontents
\endgroup

\section{Divisor classes on projected del Pezzo surfaces}
\label{sec:smooth}

In this section, we aim to provide an interface for non-experts to
some known properties of del Pezzo surfaces.
We require these results for characterizing the incidences between curves
and singular components in low degree surfaces.
We define everything for the real setting,
but implicitly these definitions also apply to the non-real setting.

Suppose that $f\c X \dto Y\subset\P^n$ is a rational map that is not defined at~$U\subset X$.
By abuse of notation, we denote the Zariski closure of $f(X\setminus U)\subseteq Y$ by $f(X)$.
If $U=\varnothing$, then $f$ is called a \df{morphism}.

A \df{smooth model}~$\bO(X)$ of a surface $X\subset\P^n$ is a nonsingular surface
\st there exists a birational morphism~$\varphi\c\bO(X)\to X$
that does not contract complex $(-1)$-curves.
We refer to $\varphi$ as a \df{desingularization}.

The \df{N\'eron-Severi lattice}~$N(X)$ of a surface is
an additive group  defined by the divisor classes on~$\bO(X)$ up to numerical equivalence.
This group comes with an intersection product~$\_\cdot\_\c N(X)\times N(X)\to \Z$
and the involution~$\sigma_*\c N(X)\to N(X)$ induced by the real structure~$\sigma_X\c X\to X$.
If $X$ is non-real, then $\sigma_*=\id$.
We denote by $\aut N(X)$ the additive group automorphisms that are compatible with both $\cdot$ and~$\sigma_*$.

Suppose that $C\subset X$ is a complex and possibly reducible or non-reduced curve.
Let $\widetilde{C}\subset \bO(X)$ denote the union of complex curves
in~$\varphi^{-1}(C)$ that are not contracted to complex points by the morphism~$\varphi$.
The \df{class}~$[C]\in N(X)$ of~$C$ is defined as the divisor class of~$\widetilde{C}$.
This notion generalizes and is compatible with \Cref{def:class}.

We call a surface $X\subset\P^n$ a \df{dP surface} if the following three properties hold:
\begin{Mlist}
\item $N(X)\cong\bas{\l_0,\l_1,\p_1,\ldots,\p_r}_\Z$
\st $r=8-\deg X$ and the nonzero intersections between
the generators are $\l_0\cdot\l_1=1$ and $\p_1^2=\ldots=\p_r^2=-1$.

\item The class of a hyperplane section of~$X$
is equal to the \df{anticanonical class}
\[
-\k:=2\,\l_0+2\,\l_1-\p_1-\ldots-\p_r.
\]

\item $-\k\cdot [C]\geq 0$ for all complex curves $C\subset X$.
\end{Mlist}

If $X$ is a dP surface, then its \df{anticanonical model}
is defined as
\[
\bA(X):=\psi_{-\k}(\bO(X))\subset\P^{\deg X},
\]
where $\psi_{-\k}\c\bO(X)\to\P^{\deg X}$ is the morphism associated to the class~$-\k$.
An \df{anticanonical projection} is defined as
a birational linear projection $\eta\c \bA(X)\to X$ \st $\deg\bA(X)=\deg X$.

\begin{proposition}
\label{prp:OA}
Suppose that $X$ is a dP surface of degree $3\leq d\leq 9$.
\begin{Mlist}
\item The smooth model $\bO(X)$ is unique up to biregular isomorphisms
and there exists a desingularization~$\varphi\c\bO(X)\to X$.
\item The anticanonical model $\bA(X)\subset\P^d$ is unique up to
projective automorphisms of $\P^d$ and there exists an
anticanonical projection~$\eta\c\bA(X)\to X$.
\end{Mlist}
\end{proposition}

\begin{proof}
See \citep[Theorem~2.16]{2018kol} for the existence and uniqueness
up to biregular isomorphisms of $\bO(X)$ and $\varphi$ for any surface.
See \citep[Definition~8.1.18 and Theorem~8.3.2]{2012dol} for the existence of $\bA(X)$.
As a direct consequence of the definitions, $\bA(X)$ and $\eta$ are unique up to
projective transformations \citep[\textsection II.7]{1977}.
\end{proof}

\begin{theoremext}[Schicho, 2001]
\label{thm:mc}
If $X\subset\P^n$ is a
complex surface that is not covered by complex lines
and contains $2\leq\lambda<\infty$ complex conics that pass through a general complex point in~$X$,
then $X$ is a complex dP surface of degree $3\leq d\leq 8$.
\end{theoremext}

\begin{proof}
Direct consequence of \citep[Theorems~5--8 and Proposition~1]{2001}.
See the proof of \citep[Theorem~A]{2024fact} for more details.
\end{proof}

Let $X\subset\P^n$ be a dP surface with desingularization~$\varphi\c\bO(X)\to X$.
We consider the following subsets of the N\'eron-Severi lattice $N(X)$:
\begin{Mlist}
\item
$B(X)$ is the set of classes of complex irreducible curves $C\subset \bO(X)$
\st $\varphi(C)\in X$ is a complex point.

\item $G(X)$ is the set of classes of complex irreducible conics in~$X$
that are not singular components of~$X$.

\item $E(X)$ is the set of classes of complex curves $C\subset\bO(X)$
that are mapped to complex lines in the anticanonical model~$\bA(X)$.
\end{Mlist}
We use the following shorthand notation for classes of
complex curves on a dP surface~$X$ (these will be elements in $B(X)\cup G(X)\cup E(X)$):
\[
\begin{array}{@{}l@{\hspace{1cm}}l@{\hspace{1cm}}l@{}}
\tb_{ij}:=\p_i-\p_j, & b_{ij}:=\l_0-\p_i-\p_j   & b_0:=\l_0+\l_1-\p_1-\p_2-\p_3-\p_4, \\
                           & b_{ij}':=\l_1-\p_i-\p_j. &
\end{array}
\]
\[
\begin{array}{@{}l@{\hspace{1cm}}l@{\hspace{1cm}}l@{}}
g_0:=\l_0, & g_2:=2\,\l_0+ \l_1-\p_1-\p_2-\p_3-\p_4, & g_{ij}:=\l_0+\l_1-\p_i-\p_j,\\
g_1:=\l_1, & g_3:= \l_0+2\,\l_1-\p_1-\p_2-\p_3-\p_4. &
\end{array}
\]
\[
e_i=\p_i,\qquad e_{ij}=\l_i-\p_j,\qquad e'_i=b_0+\p_i.
\]
If
$c\in N(X)$
and
$\Psi\subset N(X)$, then we write $c\sim \Psi$
if $c$ is up to permutation of the generators in~$\{\p_1,\ldots,\p_4\}$ and up to switching generators $\l_0$ and $\l_1$, equal
to an element in $\Psi$.
For example, $g_{34} \sim \{g_{12}\}$, $g_0\sim\{g_1\}$ and $g_2\sim\{g_3\}$.

We call $W\subset B(X)$ a \df{component} if it defines
a maximal connected subgraph of the graph with vertex set $B(X)$ and edge set $\set{\{a,b\}}{a\cdot b>0}$.

We write $c\cdot W\succ 0$ for $c\in N(X)$ and $W\subset N(X)$, if there exists $w\in W$ \st $c\cdot w>0$.
We write $c\cdot W\nprec 0$, if there does not exists $w\in W$ \st $c\cdot w<0$.

\begin{lemma}
\label{lem:BGE}
Suppose that $X$ is a complex dP surface of degree $4\leq d\leq 8$.
\begin{claims}
\item\label{lem:BGE:a}
$B(X)=\set{b\in B(X)}{b\sim\{\tb_{12},b_{12},b_0\}}$,
\\$B(X)\cap \set{\tb_{ij}\in B(X)}{i>j}=\varnothing$ and
$\set{a\cdot b}{a,b\in B(X)}\subseteq\{-2,0,1\}$.

\item\label{lem:BGE:b}
$G(X)=\set{c\in N(X)}{c\sim\{g_0,g_2,g_{12}\},~ c\cdot B(X)\nprec 0}$ and
\\
$E(X)=\set{c\in N(X)}{c\sim\{e_1,e_1',e_{01}\},~ c\cdot B(X)\nprec 0}$.
\end{claims}
\end{lemma}

\begin{proof}
See \citep[Lemma~1]{2021circle}.
\end{proof}

\begin{remark}
If $X$ is a complex dP surface, then
$\bO(X)$ is a ``weak del Pezzo surface'' \citep[Definition~8.1.18]{2012dol},
$B(X)$ is the set of classes of ``$(-2)$-curves'' \citep[\textsection8.2.7]{2012dol},
$E(X)$ is the set of classes of ``$(-1)$-curves'' \citep[\textsection8.2.6]{2012dol} and
$\bA(\P^1\times\P^1)$ is a ``Veronese-Segre surface'' \citep[\textsection8.4.1]{2012dol}.
In \citep[\textsection8.2]{2012dol}, classes are considered with respect to the basis
$(\l_0+\l_1-\p_1,\l_0-\p_1,\l_1-\p_1,\p_2,\ldots,\p_r)$ instead of
$(\l_0,\l_1,\p_1,\ldots,\p_r)$.
\END
\end{remark}

A \df{pencil} on a surface $X\subset\S^3$ is defined as
a rational map~$F\c X\dto\P^1$.
The \df{member}~$F_i\subset X$ for index $i\in \P^1$
is defined as the Zariski closure of the fiber~$F^{-1}(i)$.
We call a complex point $p\in X$ a \df{base point} of~$F$
if $p\in F_i$ for all~$i\in\P^1$.
We call $F$ a \df{pencil of conics} if
$F_i$ is a complex irreducible conic for almost all~$i\in \P^1$.
We call $F$ a \df{pencil of circles} if it is a pencil of conics \st
$F_i$ is a circle for infinitely many~$i\in \P^1_\R$.

\begin{lemma}
\label{lem:C}
Suppose that $X$ is a (complex) dP surface of degree $4\leq d\leq 8$. Let
\begin{Mlist}
\item $\eta\c\bA(X)\to X$ be an anticanonical projection,
\item $\cW(X)$ be the set of components in~$B(X)$,
\item $\cF(X)$ be the set of (complex) pencils of conics on~$X$,
\item $\cG(X):=\set{g\in G(X)}{\sigma_*(g)=g}$, and
\item $\cE(X)$ be the set of complex lines in $X$ that are not contained in~$\Sing X$.
\end{Mlist}
Suppose that
$C,C'\subset X$ are different complex curves,
$p\in X$ a complex point,
$W,W'\in\cW(X)$ components, $F,F'\in\cF(X)$ pencils,
$e\in E(X)$ a class and $i,j\in\P^1$ general indices.
\begin{claims}
\item\label{lem:C:a}
There exists a surjective function
$
\Gamma\c\cW(X)\to\eta(\Sing\bA(X))\subset\Sing X
$
that satisfies the following properties:
\begin{Mlist}
\item If $[C]\cdot W\succ 0$, then $\Gamma(W)\in C$.
\item $\sigma_*(W)=W'$ and $X$ is real if and only if $\sigma_X(\Gamma(W))=\Gamma(W')$.
\item If $e\cdot W\succ 0$, $e\cdot W'\succ 0$ and $W\neq W'$, then $\Gamma(W)\neq\Gamma(W')$.
\end{Mlist}

\item\label{lem:C:b}
The map $\Lambda\c\cF(X)\to\cG(X)$ that sends $F$ to $[F_t]$ with $t:=(0:1)$ is a
bijection that satisfies the following properties:
\begin{Mlist}
\item
$p$ is a base point of $F$
if and only if
$\Lambda(F)\cdot W\succ 0$ and $\Gamma(W)=p\in\Sing X$.

\item
$\Lambda[F]\cdot\Lambda[F']=|F_i\cap F'_j\setminus \Sing X|$.

\item
If $C\nsubseteq\Sing X$, then
$|F_i\cap C\setminus \Sing X|\leq \Lambda[F]\cdot[C]$.

\item
If $\Lambda[F]\cdot[C]>0$,
then for general $c\in C$ there exists $s\in\P^1$
\st $c\in F_s$.
\end{Mlist}

\item\label{lem:C:c}
The map $\xi\c\cE(X)\to E(X)$ that sends $L$ to $[L]$ is injective.

\item\label{lem:C:d}
If $\deg C\leq 2$ and $\deg C'\leq 2$, then the following holds:
\begin{Mlist}
\item
If $[C]\cdot [C']\neq 0$, then $C\cap C'\neq \varnothing$.
\item
If $[C]\cdot [C']=0$, then $C\cap C'\subset \Sing X$.
\item
If $[C]\cdot W\succ0$ and $[C']\cdot W\succ0$, then $\Gamma(W)\in C\cap C'$.
\end{Mlist}
\end{claims}
\end{lemma}

\begin{proof}
Let us first suppose that $X=\bA(X)$.
In this case, each assertion follows from \citep[Proposition~13]{2024darboux}
(the hypothesis assumes that $d=4$, but the same proof applies when $5\leq d\leq 8$).
Moreover, the functions~$\Gamma$ and $\xi$ are bijective, and $C\cap C'\neq \varnothing$
if and only if
either $[C]\cdot [C']\neq 0$,
or there exists $V\in \cW(X)$
\st both $[C]\cdot V\succ0$ and $[C']\cdot V\succ0$.

Now suppose that $X\neq \bA(X)$.
The regular birational linear projection~$\eta$
does not contract complex curves to complex points (see the proof of \citep[Proposition~20]{2024great}).
The hypothesis of the third property at \ASN{lem:C:a} and the previous paragraph imply
that $\Gamma(W)$ and $\Gamma(W')$ are projections of two isolated singularities
in $\bA(X)$ that lie on a complex line~$L\subset \bA(X)$.
Since $\eta$ does not contract $L$ to a complex point, it follows that $\Gamma(W)\neq \Gamma(W')$.
Each assertion is now a straightforward consequence
of the previous paragraph, and the fact that $\bA(X)\setminus\eta^{-1}(\Sing X)\cong X\setminus\Sing X$.
\end{proof}

\begin{lemma}
\label{lem:l01}
If $X$ is a complex dP surface of degree $4\leq d\leq 8$ and
$u,v\in G(X)$ \st $u\cdot v=1$, then we may assume up to $\aut N(X)$
that $u=g_0$ and $v=g_1$.
\end{lemma}

\begin{proof}
By \RL{C}{b} there exist two pencils $F,F'\c X\dto \P^1$
\st  $u=[F_i]$ and $v=[F'_j]$ for general $i,j\in\P^1$.
Moreover,
there exists a complex point $p$ \st $F_i\cap F_j\setminus\Sing X=\{p\}$.
Hence, there exists a birational map $\mu\c\P^1\times\P^1\dto X$
that sends $(i,j)$ to $p$.
The parameter lines are conics and thus the components of~$\mu$ are of bidegree~$(2,2)$.
After resolving the locus of indeterminancy,
we obtain a smooth surface~$Y$
and
birational morphisms
$\alpha\c Y\to \P^1\times\P^1$
and
$\beta\c Y\to \bO(X)$
\st
$\mu\circ\alpha=\varphi\circ\beta$
for some desingularization~$\varphi\c \bO(X)\to X$.
The surface $\P^1\times\P^1$ is a dP surface of degree 8 and
the classes of the fibers of the projections of $\P^1\times\P^1$
to its first or second component are $\l_0$ and $\l_1$, \resp.
It follows from \citep[Propositions~V.3.2 and~V.3.6]{1977}
that $\{\l_0,\l_1\}\subset N(\P^1\times \P^1)$ corresponds via $\beta_*\circ\alpha^*$
to $\{\l_0,\l_1\}\subset N(X)$.
The parameter lines of $\mu$ define the pencils $F$ and $F'$ and thus
we may assume up to $\aut N(X)$ that
$u=(\beta_*\circ\alpha^*)(\l_0)=\l_0=g_0$ and $v=(\beta_*\circ\alpha^*)(\l_1)=\l_1=g_1$
\end{proof}

\begin{lemma}
\label{lem:dp4}
If $X$ is a dP surface of degree~4
\st
$\sigma_*(\l_0)=\l_0$ and $\sigma_*(\l_1)=\l_1$,
and
there exist different components~$W,W'\subset B(X)$ \st
\[
\l_0\cdot W \succ 0,\qquad
\l_1\cdot W'\succ 0,\qquad
\sigma_*(W) \neq W,\qquad
\sigma_*(W')\neq W',
\]
then, up to $\aut N(X)$, we have
$\sigma_*(\p_1)=\p_2$, $\sigma_*(\p_3)=\p_4$,
\[
B(X)=\{b_{13},b_{24},\bp_{14},\bp_{23}\},\quad
G(X)=\{g_0,g_1,g_{12},g_{34}\}
~~\text{and}~~
E(X)=\{e_1,e_2,e_3,e_4\}.
\]
\end{lemma}

\begin{proof}
It follows from
\citep[Theorem~2 (see in particular row 67 in Table~7)]{2021web}
with $\sigma_*(W) \neq W$ and $\sigma_*(W')\neq W'$
that $|B(X)|=4$ and $B(X)$ is unique up to $\aut N(X)$.
We conclude from \Cref{lem:BGE} that $B(X)$, $G(X)$ and $E(X)$ are as asserted
(see alternatively \cite{2024celest} for an
\href{https://github.com/niels-lubbes/celestial-surfaces#automatic-verification-for-lemma-15}{automatic verification}).
By assumption, $\sigma_*(g_0)=g_0$, $\sigma_*(g_1)=g_1$,
$\sigma_*(b_{13})=b_{24}$ and $\sigma_*(\bp_{14})=\bp_{23}$,
and thus $\sigma_*$ must be as asserted.
\end{proof}

\begin{proposition}
\label{prp:BGE8}
If $X\subset\S^3$ is a celestial surface of degree~8,
then $X$ is a dP~surface \st
$\P^1\times\P^1\cong\bO(X)\cong\bA(X)$,
$B(X)=E(X)=\varnothing$,
$G(X)=\{g_0,g_1\}$
and $X$ has no isolated singularities.
\end{proposition}

\begin{proof}
See \citep[Proposition~17]{2024great}.
\end{proof}

\section{Concurrent complex double lines and projections}
\label{sec:pr}

In this section, we provide a characterization
for the linear singular components in celestial surfaces of degree eight.
Our strategy is to analyze complex projections
of these surfaces \st the center of
projection coincides with a complex point
in the singular locus.
The N\'eron-Severi lattices
of the complex projections encode more geometry
than the N\'eron-Severi lattice of the celestial surface itself.

\begin{notation}
\label{ntn:X}
In what follows, we let $\cX\subset\S^3$ denote a celestial surface of degree 8
and $\cZ_p\subset\P^3$ for $p\in\S^3$ denotes the stereographically projected complex surface~$\pr_p(\cX)$.
Notice that $\cZ_p$ is real if and only if $p\in\S^3_\R$.
\END
\end{notation}

We denote the complex tangent hyperplane section of $\S^3$ at
a complex point $p\in\S^3$
by $\U_p\subset\S^3$.
Thus $\U_p$ is a complex quadric cone with complex vertex $p$.

The \df{absolute conic}~$\cU_p$ is defined as
the complex irreducible conic~$\pr_p(\U_p)\subset\P^3$.
The multiplicity of $\cU_p\cap Z$ in a complex variety~$Z\subset\P^3$
is called the \df{cyclicity} of~$Z$.
The \df{plane at infinity}~$\cH_p\subset\P^3$
is defined as the complex plane spanned by $\cU_p$.
Notice that $\pr_p\c\S^3\dto\P^3$ is a complex biregular isomorphism outside $\U_p$ and
$\cH_p$.

The \df{concurrency}~$\theta_p(\cX)\in\Z_{\geq 0}$ denotes
the number of complex lines in~$\cX$
that pass through the complex point~$p\in \cX$.

\begin{lemma}
\label{lem:d}
For all $p\in\Sing\cX$ the complex projection $\cZ_p$ is a complex dP surface of
cyclicity $\deg \cZ_p-4$ and degree $4\leq\deg\cZ_p\leq 6$.
\end{lemma}

\begin{proof}
It follows from \Cref{prp:BGE8} and \RL{C}{b} that
$\cX$ is covered by two base point free pencils of circles.
Hence, $\cZ_p$ is covered by two families of complex conics
and not covered by complex lines.
\Cref{thm:mc} now implies that $\cZ_p$ is a dP surface.
It follows from \citep[Proposition~5]{2021circle}
that the absolute conic $\cU_p$ has multiplicity~$\deg \cZ_p-4$ in~$\cZ_p$
and thus $\deg\cZ_p\in\{4,5,6\}$ by B\'ezout's theorem.
\end{proof}

\begin{lemma}
\label{lem:dpencil}
For all $p\in\Sing\cX$ there exist
base point free pencils of circles~$M$ and~$M'$ on~$\cX$
and
complex pencils of conics~$F$ and~$F'$
on the complex dP surface $\cZ_p$
\st up to $\aut N(\cX)$ and $\aut N(\cZ_p)$, and for general $i\in\P^1$, we have
\[
F_i=\pr_p(M_i),     \qquad
F'_i=\pr_p(M'_i),   \qquad
[F_i]=[M_i]=g_0,    \qquad
[F'_i]=[M'_i]=g_1,
\]
and $|M_i\cap \U_p|=|M'_i\cap\U_p|=|F_i\cap \cU_p|=|F'_i\cap \cU_p|=2$.
\end{lemma}

\begin{proof}
Recall from \Cref{lem:d} that $\cZ_p$ is a complex dP surface \st $4\leq \deg\cZ_p\leq 8$.
We know from \Cref{prp:BGE8} and \RL{C}{b} that
$[M_i]=g_0$ and $[M'_i]=g_1$ up to $\aut N(\cX)$ so that $[M_i]\cdot [M'_i]=1$.
It follows from \RL{C}{b} and $\pr_p$ being birational
that $|M_i\cap M'_i\setminus \Sing\cX|=|F_i\cap F'_i\setminus \Sing\cZ_p|=1$ and thus $[F_i]\cdot[F'_i]=1$.
We conclude from \Cref{lem:l01} that $[F_i]=g_0$ and $[F'_i]=g_1$ up to $\aut N(\cZ_p)$.
Since the complex quadratic cone~$\U_p\subset\S^3$ is a hyperplane section and
$M$ is a base point free pencil of circles,
it follows from B\'ezout's theorem that $|M_i\cap \U_p|=2$,
which implies that $|F_i\cap \cU_p|=2$.
\end{proof}

\begin{lemma}
\label{lem:d4}
If $\deg \cZ_p=4$ for some $p\in\Sing\cX$, then $\theta_p(\cX)=4$ and $p\in\Sing\cX_\R$.
\end{lemma}

\begin{proof}
Suppose that $M$, $M'$, $F$, $F'$ are as in \Cref{lem:dpencil}
so that for general $i\in\P^1$:
\[
F_i\cap\cU_p=\{\aa,\oa\},\qquad
F'_i\cap\cU_p=\{\bb,\ob\},\qquad
[F_i]=\l_0,\qquad
[F'_i]=\l_1.
\]
Let $L:=\pr_p^{-1}(\aa)\cap\cX$, $\oL:=\pr_p^{-1}(\oa)\cap\cX$,
$R:=\pr_p^{-1}(\bb)\cap\cX$ and $\oR:=\pr_p^{-1}(\ob)\cap\cX$.
It follows from \Cref{lem:d} that the complex dP surface $\cZ_p$
does not contain the absolute conic~$\cU_p$,
which implies that $\aa$ and $\oa$ are base points of~$F$.
Since the pencil of circles~$M$ on~$\cX$ is base point free,
$L$ and $\oL$ must be complex lines in the complex quadratic cone~$\U_p$.
A circle that belongs to~$M$ meets~$L$ and~$\oL$ in complex conjugate
points and thus $\sigma_{\cX}(L)=\oL$.
This implies $L\cap\oL$ is real so that $p\in\Sing\cX_\R$.
Therefore, $\cZ_p$ is real and $\aa,\oa\in\cZ_p$ are complex conjugate base points of~$F$.
We apply \RLS{C}{a}{C}{b}, and deduce that there exists
a component $W_\aa\subset B(\cZ_p)$ \st
\[
\Gamma(W_\aa)=\aa,\qquad
\Gamma(\sigma_*(W_\aa))=\oa,\qquad
\Lambda(F)\cdot W_\aa\succ 0
\quad\text{and}\quad
\sigma_*(W_\aa)\neq W_\aa.
\]
Similarly, there exists
a component $W_\bb\subset B(\cZ_p)$ \st
\[
\Gamma(W_\bb)=\bb,\qquad
\Gamma(\sigma_*(W_\bb))=\ob,\qquad
\Lambda(F')\cdot W_\bb\succ 0
\quad\text{and}\quad
\sigma_*(W_\bb)\neq W_\bb.
\]
The hypothesis of \Cref{lem:dp4} is satisfied so that
\[
W_\aa=\{\bp_{14}\},~~
\sigma_*(W_\aa)=\{\bp_{23}\},~~
W_\bb=\{b_{13}\},~~
\sigma_*(W_\bb)=\{b_{24}\}
~~\text{and}~~
e_1,e_4\in E(\cZ_p).
\]
It follows from \RL{C}{a} that $\aa\notin\{\bb,\ob\}$, because
\[
e_1\cdot W_\aa\succ 0,\qquad
e_1\cdot W_\bb\succ 0,\qquad
e_4\cdot W_\aa\succ 0
\quad\text{and}\quad
e_4\cdot \sigma_*(W_\bb)\succ 0.
\]
This implies that $\{\pr_p(L),\pr_p(\oL)\}\neq\{\pr_p(R),\pr_p(\oR)\}$
and thus $\U_p\cap\cX$ contains at least four complex lines.
Since $E(\cX)=\varnothing$ by \Cref{prp:BGE8}, it follows from \RL{C}{c}
that $L,\oL,R,\oR\subseteq\U_p\cap\Sing\cX$ are singular components of multiplicity at least two.
By B\'ezout's theorem $\U_p\cap\Sing\cX$ consists of two pairs of complex conjugate double lines
and thus $\theta_p(\cX)=4$ with $p\in\Sing\cX_\R$.
\end{proof}

\begin{lemma}
\label{lem:d5}
If $\deg \cZ_p=5$ for some $p\in\Sing\cX$, then $\theta_p(\cX)=2$ and
\[
\pr_p(\Sing\cX)=\Sing\cZ_p.
\]
\end{lemma}

\begin{proof}
We know from \Cref{lem:d} that $\cZ_p$ is a dP surface
and we assume \Wlog that $\cZ_p$ is non-real so that $\sigma_*=\id$ by convention.
We apply \Cref{lem:BGE} and find that
\begin{gather*}
B(\cZ_p)\subseteq\set{b_{st},\bp_{st},\tb_{st}}{(s,t)\in \{(1,2),(1,3),(2,3)\}},
\\
G(\cZ_p)\subseteq\{g_0,g_1,g_{12},g_{13},g_{23}\}
\quad\text{and}\quad
E(\cZ_p)\subseteq\{e_1,e_2,e_3,e_{01},e_{02},e_{03},e_{11},e_{12},e_{13},e'_4\}.
\end{gather*}
See \cite{2024celest} for a
\href{https://github.com/niels-lubbes/celestial-surfaces#intersection-numbers-for-lemma-21}{table of intersection numbers}
between these classes.
We consider the bijection $\Lambda\c\cF(\cZ_p)\to G(\cZ_p)$ at \RL{C}{b}.
Let $i\in\P^1$ be general.
By \Cref{lem:dpencil} there exist
complex pencils of conics~$F$ and~$F'$ on $\cZ_p$ \st
\[
\Lambda(F)=g_0,\qquad
\Lambda(F')=g_1
\quad\text{and}\quad
|F_i\cap \cU_p|=|F'_i\cap \cU_p|=2.
\]
It follows from \Cref{lem:d} that $\cZ_p$ has cyclicity one
so that $\cU_p\nsubseteq\Sing\cZ_p$ and thus $[\cU_p]\in G(\cZ_p)$.
Since $\set{u\cdot v}{u,v\in G(\cZ_p)}=\{0,1\}$,
we deduce from the third property at \RL{C}{b} that
\[
|F_i\cap\cU_p\setminus\Sing\cZ_p|\leq\Lambda(F)\cdot [\cU_p]\leq 1
\quad\text{and}\quad
|F'_i\cap\cU_p\setminus\Sing\cZ_p|\leq\Lambda(F)\cdot [\cU_p]\leq 1.
\]
Hence, it follows from the first property at \RL{C}{b} that
there exist components $W_\aa,W_\bb\subset B(\cZ_p)$
corresponding to the base points $\aa,\bb\in\cU_p$
\st $\Lambda(F)\cdot W_\aa\succ 0$ and $\Lambda(F')\cdot W_\bb\succ 0$.
Since $\Lambda(F)=\l_0$ and $\Lambda(F)=\l_1$, we find that
\[
W_\aa\cap\{\bp_{12},\bp_{13},\bp_{23}\}\neq\varnothing
\quad\text{and}\quad
W_\bb\cap\{b_{12},b_{13},b_{23}\}\neq\varnothing.
\]
We may assume that $b_{12}\in W_\bb$ up to an element of $\aut N(\cZ_p)$ that permutes the generators
in~$\{\p_1,\p_2,\p_3\}$.
By \RL{BGE}{a}, we require that $w\cdot b_{12}\neq -1$ for all $w\in W_\aa$,
and thus $\bp_{13}\in W_\aa$ up to permuting the generators in~$\{\p_1,\p_2\}$.
Applying \RL{BGE}{a} again shows that
\[
B(\cZ_p)\subseteq\{b_{12},\bp_{13},\tb_{12},\tb_{13}\}
\quad\text{and}\quad
G(\cZ_p)\subseteq\{g_0,g_1,g_{23}\}.
\]
Since $\Lambda(F)\cdot w=0$ for all $w\in  B(\cZ_p)\setminus \{\bp_{13}\}$,
we deduce that $F$ has only one base point on $\cU_p$.
Hence, $|F_i\cap\cU_p\setminus\Sing\cZ_p|>0$
so that $[\cU_p]\cdot \Lambda(F)>0$ by the third property of \RL{C}{b}.
Similarly, $[\cU_p]\cdot \Lambda(F')>0$ and thus
\[
[\cU_p]=g_{23}.
\]
Since $g_{23}\cdot \tb_{12}=g_{23}\cdot \tb_{13}=-1$
it follows from \Cref{lem:BGE} that
\[
B(\cZ_p)=\{b_{12},\bp_{13}\},\quad
G(\cZ_p)=\{g_0,g_1,g_{23}\}
\quad\text{and}\quad
E(\cZ_p)=\{e_1,e_2,e_3,e_{03},e_{12}\}.
\]
We established that $W_\aa=\{\bp_{13}\}$ and $W_\bb=\{b_{12}\}$.
Since $[\cZ_p\cap\cH_p]=-\k$, $[\cU_p]=g_{23}$ and
$-\k-g_{23}-b_{12}-\bp_{13}=e_1+e_2+e_3$,
it follows from \RLS{C}{c}{C}{d}, that the hyperplane section~$\cZ_p\cap\cH_p$ consists aside $\cU_p$
of three complex lines \st their incidences
are diagrammatically illustrated by one of
the two diagrams in \Cref{fig:d5}.
In particular, $\aa\neq\bb$ and no 1-dimensional component of $\Sing\cZ_p$ is
contained in the plane at infinity~$\cH_p$.
Since $\pr_p$ is a biregular isomorphism outside~$\U_p$ and~$\cH_p$,
we conclude that $\pr_p(\Sing\cX)=\Sing\cZ_p$.
\begin{figure}[!ht]
\centering
\setlength{\tabcolsep}{1cm}
\begin{tabular}{cc}
\begin{tikzpicture}
\draw[blue] (0,0) circle [radius=1];
\node[blue, below] at (0,-1) {$g_{23}$};
\draw[thick, red, densely dotted] (-1.3,0.7) -- (1.3,0.7) node[right] {$e_1$};
\draw[thick, red, densely dotted] (-1,1) -- (1,-1)  node[below right] {$e_3$};
\draw[thick, red, densely dotted] ( 1,1) -- (-1,-1) node[below left] {$e_2$};
\draw[draw=black, fill=colGreen] (-0.7,0.7)  circle [radius=0.1] node[above=5, colCiteGreen] {$\{\bp_{13}\}$};
\draw[draw=black, fill=colGreen] ( 0.7,0.7)  circle [radius=0.1] node[above=5, colCiteGreen] {$\{b_{12}\}$};
\draw[draw=black, fill=colO    ] (0,0)       circle [radius=0.1];
\draw[draw=black, fill=white   ] (-0.7,-0.7) circle [radius=0.1];
\draw[draw=black, fill=white   ] ( 0.7,-0.7) circle [radius=0.1];
\end{tikzpicture}
&
\begin{tikzpicture}
\draw[blue] (0,0) circle [radius=1];
\node[blue] at (1,-1) {$g_{23}$};
\draw[thick, red, densely dotted] (-1.3,0.7) -- (1.3,0.7) node[right] {$e_1$};
\draw[thick, red, densely dotted] (-0.9,1.2) -- ( 0.25,-1.5);
\draw[thick, red, densely dotted] ( 0.9,1.2) -- (-0.25,-1.5);
\node[below right, red] at ( 0.25,-1.3) {$e_3$};
\node[below left,  red] at (-0.25,-1.3) {$e_2$};
\draw[draw=black, fill=colGreen] (-0.7,0.7) circle [radius=0.1] node[above left=2, colCiteGreen] {$\{\bp_{13}\}$};
\draw[draw=black, fill=colGreen] ( 0.7,0.7) circle [radius=0.1] node[above right=2, colCiteGreen] {$\{b_{12}\}$};
\draw[draw=black, fill=colO    ] (0,-1)     circle [radius=0.1];
\end{tikzpicture}
\end{tabular}
\caption{The classes of the irreducible components
in the hyperplane section $\cZ_p\cap\cU_p$ and the base points $\aa,\bb\in\cU_p$,
where $[\cU_p]=g_{23}$, $W_\aa=\{\bp_{13}\}$ and $W_\bb=\{b_{12}\}$.}
\label{fig:d5}
\end{figure}

The preimages $A:=\pr_p^{-1}(\aa)\cap\cX$ and $B:=\pr_p^{-1}(\bb)\cap \cX$
of the base points $\aa$ and $\bb$ are either
0-dimensional or complex lines.
By \Cref{lem:dpencil} there exist base point free pencils of circles~$M$ and~$M'$ on~$\cX$
\st
$F_i=\pr_p(M_i)$,
$F'_i=\pr_p(M'_i)$
$\Lambda(M)=\l_0$ and
$\Lambda(M')=\l_1$.
Hence, $|M_i\cap A|=|M'_i\cap B|=1$
so that $A,B\subset \cX\cap\U_p$ must be different complex lines.
Now suppose by contradiction that $C\subset \cX\cap \U_p$ is a complex line
\st $C\notin\{A,B\}$.
We may assume \Wlog that $\Lambda(M)\cdot [C]>0$ so that $M_i\cap C\neq\varnothing$ by \RL{C}{d}.
This implies that $F$ has both $\pr_p(A)$ and $\pr_p(C)$ as base points.
We arrived at a contradiction as $F$ has only one base point.
We conclude that $\theta_p(\cX)=2$.
\end{proof}

\begin{lemma}
\label{lem:d6}
If $\deg\cZ_p=6$ for some~$p\in\Sing\cX$, then $\theta_p(\cX)\in\{0,1\}$.
\end{lemma}

\begin{proof}
We know from \Cref{lem:d} that $\cZ_p$ is a complex dP surface of cyclicity two
and we assume \Wlog that $\cZ_p$ is non-real so that $\sigma_*=\id$ by convention.
Let $T$ denote the Zariski closure of $\cZ_p\cap \cH_p\setminus \cU_p$.
By B\'ezout's theorem $\cZ_p\cap \cH_p$ is of degree 6 when counted with multiplicity,
and thus $T$ is either a complex conic that is not a singular component or a complex double line.

It follows from \Cref{lem:BGE} that
$B(\cZ_p)\subset\{\tb_{12},b_{12},\bp_{12}\}$,
\[
G(\cZ_p)\subseteq\{g_0,g_1,g_{12}\}
\qquad\text{and}\qquad
E(\cZ_p)\subseteq\{e_1,e_2,e_{01},e_{02},e_{11},e_{12}\}.
\]
See \cite{2024celest} for a
\href{https://github.com/niels-lubbes/celestial-surfaces#intersection-numbers-for-lemma-22}{table of intersection numbers}
between these classes.
We consider the bijection $\Lambda\c\cF(\cZ_p)\to G(\cZ_p)$ from \RL{C}{b}.
Let $i\in\P^1$ be general.
By \Cref{lem:dpencil} there exist
complex pencils of conics~$F$ and~$F'$ on $\cZ_p$ \st
\[
\Lambda(F)=g_0,\qquad
\Lambda(F')=g_1
\quad\text{and}\quad
|F_i\cap \cU_p|=|F'_i\cap \cU_p|=2.
\]
Since $|F_i\cap\cH_p|=|F'_i\cap\cH_p|=2$ by B\'ezout's theorem, we find that
$|F_i\cap T\setminus\cU_p|=|F'_i\cap T\setminus\cU_p|=0$.
It follows from the fourth property at \RL{C}{d} that
\[
g_0\cdot[T]=g_1\cdot[T]=0.
\]
Since $[T]\notin G(\cZ_p)\subseteq\{g_0,g_1,g_{12}\}$,
we find that $T$ is either a reducible complex conic or a complex double line.
Therefore, $T=L\cup R$, where $L$ and $R$
are anticanonical projections of two complex lines in $\bA(\cZ_p)\subset\P^6$.
This implies that there exist $u,v\in E(\cZ_p)$ \st
$[T]=u+v$ and $u\cdot g_0=u\cdot g_1=v\cdot g_0=v\cdot g_1=0$,
and thus
\[
[T]=e_1+e_2.
\]
Since $g_0,g_1\in G(\cZ_p)$, $e_1,e_2\in E(\cZ_p)$
and $e_1\cdot \tb_{12}=e_2\cdot \tb_{21}=e_1\cdot \bp_{12}=-1$,
it follows from \Cref{lem:BGE} that
\[
B(\cZ_p)\subseteq\{b_{12}\}.
\]
First suppose that $B(\cZ_p)=\{b_{12}\}$.
It follows from \RL{C}{b}
that $F$ is base point free and $F'$ has exactly one base point~$\bb\in\cU_p$.
The preimage $B:=\pi^{-1}(\bb)\cap \cX$
is either 0-dimensional or a complex line.
By \Cref{lem:dpencil} there exist base point free pencils of circles~$M$ and~$M'$ on~$\cX$
\st
$F_i=\pr_p(M_i)$,
$F'_i=\pr_p(M'_i)$,
$\Lambda(M)=\l_0$ and
$\Lambda(M')=\l_1$.
Hence, $M'_i\cap B\neq \varnothing$ and thus $B\subset \cX\cap \U_p$
is a complex line.
Suppose by contradiction that $C\subset \cX\cap \U_p$ is a complex line
\st $C\neq B$.
In this case, either $\Lambda(M)\cdot [C]>0$
or $\Lambda(M')\cdot [C]>0$.
This implies that either $F$ or $F'$ has a base point $\pr_p(C)\neq\bb$.
We arrived at a contradiction and thus $\theta_p(\cX)=1$.
If $B(\cZ_p)=\varnothing$, then both $F$ and $F'$
are base point free so that $\theta_p(\cX)=0$ by a similar argument.
Hence, $\theta_p(\cX)\in\{0,1\}$ as was to be shown.
\end{proof}

\begin{proposition}
\label{prp:pr}
For all $p\in\Sing\cX$ the following holds:
\begin{Mlist}
\item $\bigl(\deg\cZ_p,~\theta_p(\cX)\bigr)\in\{(6,0),~(6,1),~(5,2),~ (4,4)\}$,
\item if $\theta_p(\cX)=4$, then $p\in\cX_\R$, and
\item if $\theta_p(\cX)=2$, then $\pr_p(\Sing\cX)=\Sing\cZ_p$.
\end{Mlist}
\end{proposition}

\begin{proof}
We know from \Cref{lem:d} that $4\leq\deg\cZ_p\leq 6$
and thus the proof is concluded by
\Cref{lem:d4,lem:d5,lem:d6}.
\end{proof}

\section{A reduction to pseudo singular types}
\label{sec:pseudo}

We introduce the ``pseudo singular type'' of $\cX$ whose data is a subset of~$\SingType\cX$.
If $\SingType\cX\in\{\Diii,\,\Div,\,\Dv\}$, then
this partial data turns out to be sufficient to recover $\SingType\cX$ completely.

Let $\vv:=(1:0:0:0:1)$ so that $\U_\vv=\set{x\in \S^3}{x_0-x_4=0}$.
We say that $\cX$ contains a \df{Bohemian quartet} $(L,R,\vv)$
if $\cX\cap \U_\vv$ consists of
two pairs of complex conjugate double lines
$L,\oL\subset\U_\vv$ and $R,\oR\subset\U_\vv$
\st
$[L]=[\oL]=\l_0$ and $[R]=[\oR]=\l_1$.
As a direct consequence of the definition, we find that
$L\cap R\cap\oL\cap\oR=\{\vv\}$.

Let $\E:=\set{x\in \S^3}{ x_0=0}$.
We say that $\cX$ contains a \df{Cliffordian quartet} $(L,R,\pp,\qq)$
if $\cX\cap\E$ consists of
two pairs of complex conjugate double lines
$R,\oR\subset\E$ and $L,\oL\subset\E$
\st
$L\cap R   =\{\pp\}$,
$\oL\cap\oR=\{\op\}$,
$L\cap\oR  =\{\qq\}$,
$\oL\cap R =\{\oq\}$,
$[L]=[\oL]=\l_0$ and $[R]=[\oR]=\l_1$.

We illustrated the incidences between the complex lines of Bohemian and Cliffordian quartets in \Cref{fig:quartet}.

\begin{figure}[!ht]
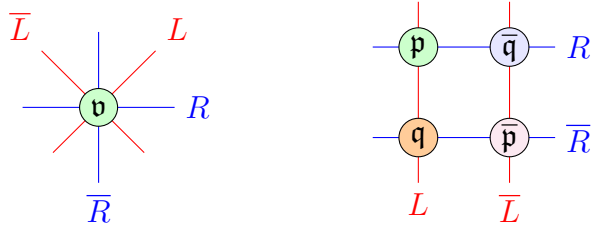

\centering
\csep{1cm}
\begin{tabular}{cc}
\Bohemianquartet & \Cliffordianquartet
\end{tabular}
\vspace{-3mm}
\caption{Incidences between complex lines in
Bohemian quartet (left) and Cliffordian quartet (right).}
\label{fig:quartet}
\end{figure}

\begin{theoremext}
\label{thm:quartet}
There exists a M\"obius transformation $\alpha\in\aut\S^3$ \st
$\alpha(\cX)$ contains either a Bohemian quartet or a Cliffordian quartet.
\end{theoremext}

\begin{proof}
See \citep[Theorem~A(b) and Proposition~22]{2024fact}.
\end{proof}

\begin{definition}
\label{def:pt}
The \df{pseudo singular type} $\PseudoType\cX$
defined as $\Piii$, $\Piv$ or $\Pv$ if
the incidence relations of the 1-dimensional singular components of~$X$
are characterized by the corresponding diagram in \Cref{tab:P},
where
\begin{Mlist}
\item the line segments $L$, $\oL$, $R$, $\oR$ represent complex double lines that form a Cliffordian quartet $(L,R,\pp,\qq)$,
\item the loops~$V$, $V_1$, $V_2$, $C$ represent double curves \st
$\deg V=\deg\pr_\pp(V)+1=4$ and $\deg V_1=\deg V_2=\deg C=2$, and
\item the discs represent complex incidence points. \END
\end{Mlist}
\begin{table}[!ht]
\centering
\caption{Incidences between singular components for each pseudo singular type.}
\label{tab:P}
\csep{6mm}
\begin{tabular}{ccc}
{\bf P3} & {\bf P4} & {\bf P5} \\
\PIII & \PIV & \PV\\
$\deg V=4$ & $\deg V_1=\deg V_2=2$ & $\deg C=2$\\
\end{tabular}
\end{table}
\end{definition}

\begin{remark}
\label{rmk:part1}
If $\SingType\cX$ is equal to $\Diii$, $\Div$ or $\Dv$,
then as a straightforward consequence of the definitions $\PseudoType\cX$ is equal to $\Piii$, $\Piv$ and $\Pv$, \resp.
We shall show in this section that the converse also holds by
specifying the classes and names of the singular components.
\END
\end{remark}

\begin{lemma}
\label{lem:V}
If $\cX$ contains a Cliffordian quartet $(L,R,\pp,\qq)$
and $V\subset\Sing\cX$ is a possibly reducible curve \st
$V\setminus\E=\Sing\cX\setminus\E$,
$V\cap\E=\{\pp,\op,\qq,\oq\}$
and $\deg V=4$,
then $[V]=2\,\l_0+2\,\l_1$.
\end{lemma}

\begin{proof}
Since $[L]=[\oL]=\l_0$ and $[R]=[\oR]=\l_1$,
we find that the preimage $E_\bA:=\eta^{-1}(\cX\cap\E)$ via the anticanonical projection~$\eta\c\bA(\cX)\to\cX$
consists of four complex irreducible conics whose four incidences
are as illustrated in the left diagram of \Cref{fig:eta}.
These complex conics
are projected 2:1 onto the complex lines in~$\E$ as is illustrated in the right diagram of \Cref{fig:eta}.
It follows from \Cref{prp:pr} that $\deg\pr_\pp(\cX)=5$ so that $\pp$ is a complex point of multiplicity three.
This implies that $\eta^{-1}(\pp)\subseteq \{\pp_0,\pp_1,\pp_2\}$
for some complex points~$\pp_0,\pp_1,\pp_2\in E_\bA$.
Using the same argument we find that
$\eta^{-1}(\op)\subseteq \{\op_0,\op_1,\op_2\}$,
$\eta^{-1}(\qq)\subseteq \{\qq_0,\qq_1,\qq_2\}$ and
$\eta^{-1}(\oq)\subseteq \{\oq_0,\oq_1,\oq_2\}$.
We illustrated the complex points in the complex fibers in \Cref{fig:eta}.
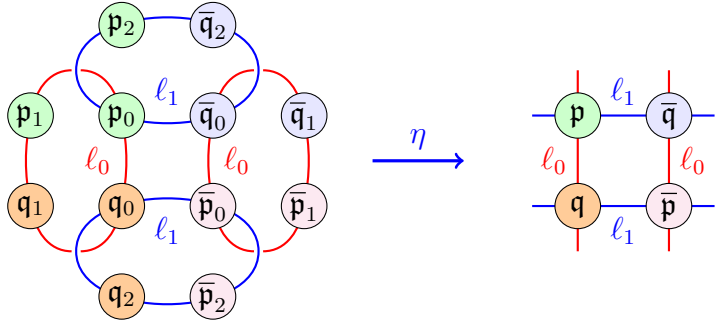
\begin{figure}[!ht]
\centering
\begin{tikzpicture}[scale=0.6]
\draw[red, thick] (-2.1,2) to [out=180, in=80] (-3,1) to [out=260, in=100] (-3,-1) to [out=280, in=180] (-2.1,-2);
\draw[red, thick] (-1.9,2) to [out=0, in=100] (-1,1)  to [out=280, in=80] (-1,-1)  to [out=260, in=0] (-1.9,-2);
\draw[red, thick] (2.1,2) to [out=0, in=100] (3,1)   to [out=280, in=80] (3,-1)   to [out=260, in=0] (2.1,-2);
\draw[red, thick] (1.9,2) to [out=180, in=80] (1,1) to [out=260, in=100] (1,-1) to [out=280, in=180] (1.9,-2);
\draw[blue, thick] (-2,2)  to [out=90, in=90] (2,2)  to [out=270, in=270] (-2,2);
\draw[blue, thick] (-2,-2) to [out=90, in=90] (2,-2) to [out=270, in=270] (-2,-2);
\node[blue,above] at (0,1)    {$\l_1$};
\node[blue,below] at (0,-1)   {$\l_1$};
\node[red,left] at  (-1,0)    {$\l_0$};
\node[red,right] at (+1,-0)   {$\l_0$};
\node[blue,above] at (3+7,1)  {$\l_1$};
\node[blue,below] at (3+7,-1) {$\l_1$};
\node[red,left]  at (3+6,0)   {$\l_0$};
\node[red,right] at (3+8,-0)  {$\l_0$};
\draw[blue, thick] (3+5,1)  -- (3+9,1);
\draw[blue, thick] (3+5,-1) -- (3+9,-1);
\draw[red, thick] (3+6,2)  -- (3+6,-2);
\draw[red, thick] (3+8,2)  -- (3+8,-2);
\draw[very thick, blue, ->] (4.5,0) -- node[above] {$\eta$} (6.5,0);
\draw[draw=black, fill=green!20] (-1,1) circle [radius=0.5] node {$\pp_0$};
\draw[draw=black, fill=green!20] (-3,1) circle [radius=0.5] node {$\pp_1$};
\draw[draw=black, fill=green!20] (-1,3) circle [radius=0.5] node {$\pp_2$};
\draw[draw=black, fill=blue!10] (1,1) circle [radius=0.5] node {$\oq_0$};
\draw[draw=black, fill=blue!10] (3,1) circle [radius=0.5] node {$\oq_1$};
\draw[draw=black, fill=blue!10] (1,3) circle [radius=0.5] node {$\oq_2$};
\draw[draw=black, fill=orange!40] (-1,-1) circle [radius=0.5] node {$\qq_0$};
\draw[draw=black, fill=orange!40] (-3,-1) circle [radius=0.5] node {$\qq_1$};
\draw[draw=black, fill=orange!40] (-1,-3) circle [radius=0.5] node {$\qq_2$};
\draw[draw=black, fill=magenta!10] (1,-1) circle [radius=0.5] node {$\op_0$};
\draw[draw=black, fill=magenta!10] (3,-1) circle [radius=0.5] node {$\op_1$};
\draw[draw=black, fill=magenta!10] (1,-3) circle [radius=0.5] node {$\op_2$};
\draw[draw=black, fill=green!20]       (3+6,1)  circle [radius=0.5] node {$\pp$};
\draw[draw=black, fill=blue!10]        (3+8,1)  circle [radius=0.5] node {$\oq$};
\draw[draw=black, fill=orange!40]      (3+6,-1) circle [radius=0.5] node {$\qq$};
\draw[draw=black, fill=magenta!10]     (3+8,-1) circle [radius=0.5] node {$\op$};
\end{tikzpicture}
\caption{Two pairs of complex conjugate conics in the anticanonical model~$\bA(\cX)$
are 2:1 linearly projected to
complex lines in the Cliffordian quartet $(L,R,\pp,\qq)$.
The fibers of the four complex incidence points have each cardinality at most three.
}
\label{fig:eta}
\end{figure}

Let $V_\bA:=\eta^{-1}(V)$.
It follows from \Cref{prp:pr} that $\deg\pr_v(\cX)=6$ for a general complex point~$v\in V$
and thus the possibly reducible curve~$V$ is of multiplicity two.
Since $\deg V=4$ we find that $\deg V_\bA\in\{4,8\}$.
By B\'ezout's theorem,
the four intersections of $V$ with the hyperplane section~$\E\subset\S^3$ must be transversal.
This implies that the intersections between $V_\bA$ and the hyperplane section~$E_\bA\subset \bA(\cX)$
must be transversal as well.
Hence, $|V_\bA\cap E_\bA|=\deg V_\bA$ by applying B\'ezout's theorem again.
Notice that $[V]\cdot (2\,\l_0+2\,\l_1)=\deg V_\bA$
because $2\,\l_0+2\,\l_1$ is the class of a hyperplane section of $\bA(\cX)\cong\P^1\times\P^1$.
We have $V_\bA\cap E_\bA\subset \set{\pp_i,\qq_i,\op_i,\oq_i}{i\in\{0,1,2\}}$,
as $V\cap\E=\{\pp,\op,\qq,\oq\}$ by assumption.
This implies that
$[V]\cdot\l_0=|V_\bA\cap\{\pp_0,\pp_1,\qq_0,\qq_1\}|=|V_\bA\cap\{\op_0,\op_1,\oq_0,\oq_1\}|$
and
$[V]\cdot\l_1=|V_\bA\cap\{\pp_0,\pp_2,\oq_0,\oq_2\}|=|V_\bA\cap\{\op_0,\op_2,\qq_0,\qq_2\}|$.
If $\deg V_\bA=8$, then it follows from the above characterization that
$V_\bA\cap E_\bA=\set{\pp_i,\qq_i,\op_i,\oq_i}{i\in\{1,2\}}$
so that $[V]=2\,\l_0+2\,\l_1$.
If $\deg V_\bA=4$, then we find for similar reasons that up to symmetry of the
diagram in \Cref{fig:eta} we have $V_\bA\cap E_\bA=\{\pp_1,\op_1,\qq_2,\oq_2\}$
so that $[V]=\l_0+\l_1$.

Now suppose by contradiction that $[V]=\l_0+\l_1$ so that $V_\bA\cap E_\bA=\{\pp_1,\op_1,\qq_2,\oq_2\}$.
In this case $V_\bA$ is projected 1:1 onto $V\subset\Sing \cX$.
Let $\tilde{\eta}\c\P^8\dto\P^4$ be the linear
projection that restricts to the anticanonical projection~$\eta$.
Since $\bA(\cX)$ is smooth by \Cref{prp:BGE8},
the fiber $\tilde{\eta}^{-1}(v)$ of a
general complex point $v\in V$
is a 4-dimensional complex plane that meets the center of projection
and that contains the complex tangent plane of $\bA(\cX)$ at a complex point in~$\eta^{-1}(v)\cap V_\bA$
(see \citep[pages 611--618]{1978} for background material).
Therefore, a general complex hyperplane section of~$\cX$ that contains $v$ has a cuspidal singularity at~$v$
(see \citep[page 292, Figure~10]{1978}).
Let the complex preimage~$L_\bA$ be defined as the Zariski closure of $\eta^{-1}(L\setminus\{\pp,\qq\})$.
Thus, $L_\bA$ is the leftmost conic in \Cref{fig:eta} that contains the complex points $\pp_0$, $\pp_1$, $\qq_0$ and $\qq_1$.
Suppose that $U$ is a complex analytic neighborhood of~$\bA(\cX)$ around~$\pp_1$.
We consider the complex analytic map $U\to\eta(U)$
and illustrated $V_\bA$, $L_\bA$ and their images \wrt this map diagrammatically in \Cref{fig:V4}.
\begin{figure}[!ht]
\centering
\begin{tabular}{ll}
\begin{tikzpicture}[scale=0.5]
\draw[thick] (-4,-2) to [out=50, in=270] (-2,0)
                     to [out=90, in=-50] (-4,2);

\draw[thick] (2,0) to [out=50, in=270] (4,2)
                   to [out=90, in=-50] (2,4);

\draw[thick] (-4,-2) to (2,0);
\draw[thick] (-4,2) to (2,4);

\draw[red] (0,-0.65) to [out=20, in=270] (1,1)
                        to [out=90, in=0] (-1,3);

\draw[blue] (-2,0) to (4,2);

\draw (1.4,2.2) node[red] {$L_\bA$};
\draw (-1.3, 0.8) node[blue] {$V_\bA$};

\draw[draw=black, fill=green!20] (1, 1)   circle [radius=0.5] node[black] {$\pp_1$};
\draw[draw=black, fill=green!20] (0, 2.7) circle [radius=0.5] node[black] {$\pp_0$};

\draw[very thick,->] (4.5,0) -- node[above] {$\eta$} (6.5,0);

\end{tikzpicture}
&
\begin{tikzpicture}[scale=0.5]
\draw[thick] (-4,-2) to [out=50, in=190] (-2,0)
                     to [out=170, in=-50] (-4,2);

\draw[thick] (2,0) to [out=50, in=190] (4,2)
                   to [out=170, in=-50] (2,4);

\draw[thick] (-4,-2) to (1,1) to (2,0);
\draw[thick] (-4, 2) to (-1,1.6) to (2,4);

\draw[very thick, red] (-1,1.6) to (1,1) ;
\draw[very thick, blue] (-2,0) to (4,2);

\draw (-1.5,0.6) node[blue] {$V$};
\draw (0,1.8) node[red] {$L$};

\draw[draw=black, fill=green!20] (1, 1) circle [radius=0.4] node[black] {$\pp$};

\end{tikzpicture}
\end{tabular}
\caption{A diagrammatic depiction of a complex analytic neighborhood~$U\subset \bA(\cX)$
around~$\pp_1$ and its anticanonical projection into~$\eta(U)$.
The complex conic $L_\bA$ is 2:1 linearly projected to the double line $L$.
The quartic curve $V_\bA$ is 1:1 linearly projected to the singular component~$V$
and thus $\pp_1$ must be a complex ramification point.}
\label{fig:V4}
\end{figure}
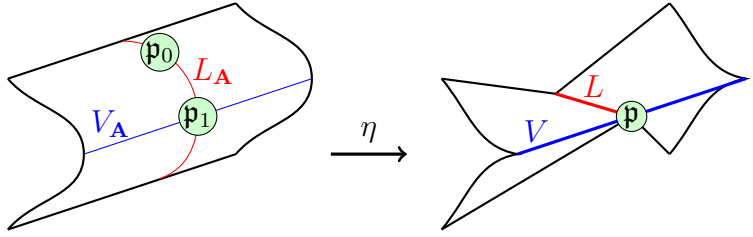

We find that $\pp_1\in V_\bA\cap E_\bA$ must be a complex ramification point of the 2:1 map $\eta|_{_{L_\bA}}\c L_\bA\to L$.
We arrived at a contradiction as $\pp_0,\pp_1\in L_\bA$ and $\eta(\pp_0)=\eta(\pp_1)=\pp\in L$.

We concluded the proof as $[V]=2\,\l_0+2\,\l_1$ is the only remaining option.
\end{proof}

\begin{proposition}
\label{prp:D345}
If $\PseudoType\cX\in\{\Piii,\,\Piv,\,\Pv\}$, then
\[
\SingType\cX\in\{\Diii,\,\Div,\,\Dv\}.
\]
\end{proposition}

\begin{proof}
It follows from \Cref{prp:BGE8} that the dP surface $\cX$ has no isolated singularities
and the class of a circle through a general point is in~$\{\l_0,\l_1\}$.
By definition, $-\k=2\,\l_0+2\,\l_1$ is the class of a hyperplane section of $\cX$.
Hence, \ref{i} and \ref{ii} and the first three items at \ref{iv} for \Cref{def:st}
are a direct consequence of the definitions of~$\cX$ and Cliffordian quartets.

It follows from \Cref{prp:BGE8} that
there exists a desingularization $\varphi\c \P^1\times\P^1\to\cX$.
Let $F_i$ for general $i\in\P^1_\R$ denote the fiber of the projection of $\P^1\times\P^1$
to the first component~$\P^1$.
Suppose that $W\subset \P^1\times\P^1$ is a possibly reducible curve \st $\varphi(W)\subset\Sing\cX$
is a singular component.
Since $\varphi(F_i)$ is isomorphic to a circle through a general point of $\cX$,
we find that $|W\cap F_i|=|\varphi(W\cap F_i)|$.
Since $|W\cap F_i|=[W]\cdot [F_i]$, it follows that \ref{v} is satisfied.

Let us consider the Clifford quartet $(L,R,\pp,\qq)$
and singular components~$V$, $V_1$, $V_2$ and $C$ at \Cref{def:pt}.
It follows from \Cref{prp:pr} that
a complex point~$p\in \Sing\cX$ has multiplicity $\max(2,1+\theta_p(\cX))$.
Thus, if $V_1\cap V_2\subset\cX_\R$, then \ref{iii} is satisfied.

It remains to show the last four items at \ref{iv}
and that $V_1\cap V_2\subset\cX_\R$.

First, let us suppose that $\PseudoType\cX=\Piii$ so that $\deg \pr_\pp(V)=3$.
We know from \Cref{lem:V} that $[V]=2\,\l_0+2\,\l_1$.
By B\'ezout's theorem, the complex double cubic~$\pr_\pp(V)$ is
not contained in a hyperplane section of
the quintic complex surface~$\pr_\pp(\cX)$ and thus must be rational by \citep[Example~IV.6.4.2]{1977}.
Therefore, $V\subset\P^4$ is a quartic rational curve that is not contained in hyperplane section,
which implies that $V$ is a twisted quartic and thus $\SingType\cX=\Diii$.

Next, we suppose that $\PseudoType\cX=\Piv$ so that $[V_1\cup V_2]=2\,\l_0+2\,\l_1$ by \Cref{lem:V}.
By assumption, $\sigma_\cX(V_1\cup V_2)=V_1\cup V_2$ and $|V_1\cap V_2|=1$.
Hence, $\sigma_\cX(V_1\cap V_2)=V_1\cap V_2$ so that $|(V_1)_\R|=|(V_2)_\R|=\infty$ and $V_1\cap V_2\subset\cX_\R$.
It follows that $V_1$ and $V_2$ are double circles \st $[V_1]+[V_2]=2\,\l_0+2\,\l_1$.
Thus, up to $\aut N(\cX)$, we have
$([V_1],[V_2])\in\{(\l_0+\l_1,\l_0+\l_1),\,(2\,\l_0+\l_1,\l_1),\,(2\,\l_0,2\,\l_1)\}$.
For all $i\in\{1,2\}$, the preimage $\eta^{-1}(V_i)$ is
projected either 2:1 or 1:1 onto $V_i$,
and thus $\deg \eta^{-1}(V_i)=-\k\cdot [V_i]\neq 3$.
We already established that the point $v\in V_1\cap V_2$ has multiplicity two,
which implies that $|\eta^{-1}(v)|=[V_1]\cdot [V_2]\leq 2$.
We deduce that $[V_1]=[V_2]=\l_0+\l_1$ and thus $\SingType\cX=\Div$.

Finally, we suppose that $\PseudoType\cX=\Pv$.
Since $|C\cap L|=0$, it follows from \RL{C}{d} that $[C]\cdot [L]=0$,
which implies that $[C]=\alpha\,\l_0$ for $\alpha\geq 1$.
Since $C$ is by assumption a double conic, we deduce that~$\alpha\leq 2$.
Hence, $\SingType\cX=\Dv$ and we concluded the proof.
\end{proof}

\section{The classification of singular types}
\label{sec:part2}

In this section, we show that if $\cX$ contains a Cliffordian quartet,
then
\[
\PseudoType\cX\in\{\Piii,\,\Piv,\,\Pv\}.
\]
Moreover, if $\cX$ contains a Bohemian quartet, then $\SingType\cX\in\{\Di,\Dii\}$.
Our method is based on the ``sectional delta invariant'', which
is defined as the sum of delta invariants of singularities of a general hyperplane section
(see \citep[\textsection3]{2024great}).
We conclude this section with a proof of the main result \Cref{thm:s}
and \Cref{cor:BC,cor:deg}.

\begin{lemma}
\label{lem:A4}
If $\cX$ contains a Cliffordian quartet~$(L,R,\pp,\qq)$
and $V_1,V_2\subset\Sing\cX$ are irreducible conics
\st $\Sing\cX\setminus\E=(V_1\cup V_2)\setminus\E$
and $(V_1\cup V_2)\cap\E=\{\pp,\op,\qq,\oq\}$,
then $\PseudoType\cX=\Piv$.
\end{lemma}

\begin{proof}
Since $[V_1\cup V_2]=[V_1]+[V_2]=2\,\l_0+2\,\l_1$ by \Cref{lem:V}, we find that
$[V_1]\cdot[V_2]>0$ and thus $|V_1\cap V_2|>0$.
As $\E$ is a hyperplane section of~$\S^3$
it follows from B\'ezout's theorem that
$(V_1\cap\E,~V_2\cap\E)$ is \Wlog equal to either
$(\{\pp,\qq\},~\{\op,\oq\})$ or
$(\{\pp,\op\},~\{\qq,\oq\})$.
We observe that either $\sigma_\cX(V_1)=V_2$ or $\sigma_\cX(V_i)=V_i$ for $i\in\{1,2\}$.
If $V_1\cap\E=\{\pp,\qq\}$, then $\sigma_\cX(V_1)=V_2$ and thus $|V_1\cap V_2|=2$.
It now follows that the incidences between the singular components
are characterized by diagram~$\Bi$, $\Bii$ or $\Biii$ in \Cref{tab:A4}.

\begin{table}[!ht]
\caption{Incidences between $V_1$, $V_2$, $L$, $\oL$, $R$ and $\oR$.}
\label{tab:A4}
\centering
\csep{5mm}
\begin{tabular}{cccc}
{\bf B1} & {\bf B2} & {\bf B3}
\\
\begin{tikzpicture}[scale=0.45]
\draw[thick, red] (-3.5, 2) -- ( 3.5, 2) node[right] {$R$};
\draw[thick, red] (-3.5,-2) -- ( 3.5,-2) node[right] {$\oR$};
\draw[thick, red] (-2, 3.5) -- (-2,-3.5) node[below] {$L$};
\draw[thick, red] ( 2, 3.5) -- ( 2,-3.5) node[below] {$\oL$};
\draw[thick, colGreen] (-2,2) to [out=180, in=180] (-2,-2) to [out=0, in=-90] (0.5,0) to [out=90, in=0] (-2,2);
\draw[thick, colGray] (2,2) to [out=0, in=0] (2,-2) to [out=180, in=-90] (-0.5,0) to [out=90, in=180] (2,2);
\draw[draw=black, fill=green!20]    (-2, 2) circle [radius=0.7] node {$\pp$};
\draw[draw=black, fill=blue!10]     ( 2, 2) circle [radius=0.7] node {$\oq$};
\draw[draw=black, fill=orange!40]   (-2,-2) circle [radius=0.7] node {$\qq$};
\draw[draw=black, fill=magenta!10]  ( 2,-2) circle [radius=0.7] node {$\op$};
\draw[draw=black, fill=white] (0,1) circle [radius=0.6] node[black] {$u$};
\draw[draw=black, fill=white] (0,-1) circle [radius=0.6] node[black] {$v$};
\node[colGreen] at (-3.8,0) {$V_1$};
\node[colGray] at (3.8,0) {$V_2$};
\end{tikzpicture}
&
\begin{tikzpicture}[scale=0.45]
\draw[thick, red] (-3.5, 2) -- ( 3.5, 2) node[right] {$R$};
\draw[thick, red] (-3.5,-2) -- ( 3.5,-2) node[right] {$\oR$};
\draw[thick, red] (-2, 3.5) -- (-2,-3.5) node[below] {$L$};
\draw[thick, red] ( 2, 3.5) -- ( 2,-3.5) node[below] {$\oL$};
\draw[thick, colGreen] (-2,2) to [out=-5, in=135] (0,1) to [out=-45, in=135] (1,0) to [out=-45, in=95] (2,-2);
\draw[thick, colGreen] (2,-2) to [out=175, in=-45] (0.1,-1.1);
\draw[thick, colGreen] (-0.1,-0.9) to [out=135, in=-45] (-0.9,-0.1 );
\draw[thick, colGreen] (-1.1,0.1) to [out=135, in=275] (-2,2);
\draw[thick, colGray] (-2,-2) to [out=85, in=225] (-1,0) to [out=225, in=225] (0,1) to [out=45, in=185] (2,2);
\draw[thick, colGray] (-2,-2) to [out=5, in=225] (0,-1) to [out=225, in=225] (0.9,-0.1);
\draw[thick, colGray] (1.1,0.1) to [out=45, in=265] (2,2);
\draw[draw=black, fill=green!20]    (-2, 2) circle [radius=0.7] node {$\pp$};
\draw[draw=black, fill=blue!10]     ( 2, 2) circle [radius=0.7] node {$\oq$};
\draw[draw=black, fill=orange!40]   (-2,-2) circle [radius=0.7] node {$\qq$};
\draw[draw=black, fill=magenta!10]  ( 2,-2) circle [radius=0.7] node {$\op$};
\draw[draw=black, fill=white] (0,1) circle [radius=0.6] node[black] {$u$};
\draw[draw=black, fill=white] (0,-1) circle [radius=0.6] node[black] {$v$};
\node[colGreen] at (-2.5,0.8) {$V_1$};
\node[colGray] at (-2.5,-0.8) {$V_2$};
\end{tikzpicture}
&
\begin{tikzpicture}[scale=0.45]
\draw[thick, red] (-3.5, 2) -- ( 3.5, 2) node[right] {$R$};
\draw[thick, red] (-3.5,-2) -- ( 3.5,-2) node[right] {$\oR$};
\draw[thick, red] (-2, 3.5) -- (-2,-3.5) node[below] {$L$};
\draw[thick, red] ( 2, 3.5) -- ( 2,-3.5) node[below] {$\oL$};
\draw[thick, colGreen] (-2,2) to [out=-5, in=135] (0,1) to [out=-45, in=135] (1,0) to [out=-45, in=95] (2,-2);
\draw[thick, colGreen] (2,-2) to [out=175, in=-45] (0.1,-1.1);
\draw[thick, colGreen] (-0.1,-0.9) to [out=135, in=-45] (-0.9,-0.1 );
\draw[thick, colGreen] (-1.1,0.1) to [out=135, in=275] (-2,2);
\draw[thick, colGray] (-2,-2) to [out=85, in=225] (-1,0) to [out=225, in=225] (0,1) to [out=45, in=185] (2,2);
\draw[thick, colGray] (-2,-2) to [out=5, in=225] (0,-1) to [out=225, in=225] (0.9,-0.1);
\draw[thick, colGray] (1.1,0.1) to [out=45, in=265] (2,2);
\draw[draw=black, fill=green!20]    (-2, 2) circle [radius=0.7] node {$\pp$};
\draw[draw=black, fill=blue!10]     ( 2, 2) circle [radius=0.7] node {$\oq$};
\draw[draw=black, fill=orange!40]   (-2,-2) circle [radius=0.7] node {$\qq$};
\draw[draw=black, fill=magenta!10]  ( 2,-2) circle [radius=0.7] node {$\op$};
\draw[draw=black, fill=white] (0,1) circle [radius=0.6] node[black] {$u$};
\node[colGreen] at (-2.5,0.8) {$V_1$};
\node[colGray] at (-2.5,-0.8) {$V_2$};
\end{tikzpicture}
\end{tabular}
\end{table}

Suppose by contradiction that $\Bi$ holds.
In this case $\pr_u(L)=\pr_u(V_1)$ and $\pr_u(\oL)=\pr_u(V_2)$.
This implies that $\pr_u(L)\cap\pr_u(\oL)=\{\pr_u(v)\}$
and thus $\pr_u(L),\pr_u(\oL)\subset\cZ_u$ are complex and coplanar lines of multiplicity at least three.
Since $\pr_u(\{\pp,\op,\qq,\oq\})\subset\pr_u(L\cup\oL)$, it follows that
$\pr_u(L\cup\oL\cup R\cup \oR)$ lies in a hyperplane section of $\cZ_u$,
where $\pr_u(R),\pr_u(\oR)\subset\cZ_u$ are complex double lines.
We arrived at a contradiction as $\deg\cZ_u\geq 3+3+2+2$ by B\'ezout's theorem.

Next, we suppose by contradiction that $\Bii$ holds.
We observe that $\pr_u(L\cup \oL\cup R\cup\oR\cup V_1\cup V_2)$
consist of six complex double lines that are
contained in a complex hyperplane section of~$\cZ_u$.
We arrived again at a contradiction, since $\cZ_u$
must be of degree at least twelve by B\'ezout's theorem.

We established that $\Biii$ holds and
thus we conclude that $\PseudoType\cX=\Piv$.
\end{proof}

Let $X\subset\P^n$ be a complex surface.
The \df{sectional arithmetic genus}~$a(X)$
and the \df{sectional geometric genus}~$g(X)$
are defined
as the arithmetic genus and geometric genus of a general hyperplane section of $X$, \resp.
The \df{total delta invariant} of~$X$ is defined as $\delta(X):=a(X)-g(X)$.

\begin{lemma}
\label{lem:delta}
For all $p\in \Sing\cX$, we have $g(\cX)=g(\cZ_p)=1$,
$\delta(\cX)=8$ and
\[
\delta(\cZ_p)=\tfrac{1}{2}\cdot\deg\cZ_p\cdot(\deg\cZ_p-3).
\]
\end{lemma}

\begin{proof}
See \citep[Lemmas~15 and~16]{2024great}
and recall from \Cref{prp:BGE8} and \Cref{lem:d} that both $\cX$ and $\cZ_p$ are dP surfaces.
\end{proof}

\newpage
\begin{definition}
\label{def:delta}
Let $X\subset\P^n$ be a complex surface
\st the 1-dimensional part of the singular locus of~$X$
admits the following decomposition
into irreducible complex curve components:
$
C_1\cup\cdots\cup C_r.
$
A \df{sectional delta invariant} for $X$ is a function
\[
\Delta_X\c \set{C_i}{1\leq i\leq r} \to \Z_{>0}
\]
that satisfies the following
axioms for all $1\leq i\leq r$, real structures $\sigma\c X\to X$,
and complex projective automorphisms~$\alpha\in \aut\P^n$:
\begin{enumerate}[topsep=0pt,itemsep=0pt,leftmargin=9mm,label=A\arabic*.,ref=A\arabic*]
\item\label{a1}
$\Delta_X(C_1)+\cdots+\Delta_X(C_r)=\delta(X)$.

\item\label{a2}
$\Delta_X( C_i)\geq \deg C_i$.

\item\label{a3}
$\Delta_X(C_i)=\Delta_X(\sigma(C_i))$ and $\Delta_X(C_i)=\Delta_{\alpha(X)}(\alpha(C_i))$.

\item\label{a4}
If $\rho\c X\dto Z\subset\P^m$ is a complex birational linear map
\st $g(X)=g(Z)$ and
$\rho|_{C_i}\c C_i\dto \rho(C_i)$ is birational, then
\[
\frac{\Delta_Z(\rho(C_i))}{\deg\rho(C_i)}=\frac{\Delta_X(C_i)}{\deg C_i}.
\]
\end{enumerate}
We write $\Delta(C)$ instead of~$\Delta_X(C)$
if it is clear from the context that $C\subset X$.
\END
\end{definition}

\begin{proposition}
\label{prp:delta}
If $X\subset \P^n$ is a complex surface, then
there exists a sectional delta invariant~$\Delta_X$.
\end{proposition}

\begin{proof}
\citep[Proposition~10]{2024great}.
\end{proof}

\begin{proposition}
\label{prp:P}
If $\cX$ contains a Cliffordian quartet $(L,R,\pp,\qq)$,
then
\[
\PseudoType\cX\in\{\Piii,\,\Piv,\,\Pv\}.
\]
\end{proposition}

\begin{proof}
We consider the complex stereographic projection~$\pr_\pp\c\S^3\dto\P^3$
and observe that the complex double lines $L, R\subset\Sing\cX$ are contracted to
complex points in $\cZ_\pp$.
We deduce from \Cref{prp:pr} that $\bigl(\deg\cZ_\pp,\theta_\pp(\cX)\bigr)=(5,2)$.
By \Cref{prp:delta} there exist sectional delta invariants for $\cX$ and $\cZ_\pp$.
Recall from \Cref{lem:delta} that $g(\cX)=g(\cZ_\pp)$, $\delta(\cX)=8$ and $\delta(\cZ_\pp)=5$.
It follows from \AXMS{a3}{a4} that
\begin{equation}
\label{eqn:LR}
\Delta(L)=\Delta(\oL)=\Delta(\pr_\pp(\oL))
\qquad\text{and}\qquad
\Delta(R)=\Delta(\oR)=\Delta(\pr_\pp(\oR)).
\end{equation}

Suppose by contradiction that $\Sing \cX\subset \E$
so that $\Sing \cX=L\cup \oL\cup R\cup \oR$.
It follows from \Cref{prp:pr} that
$\Sing \cZ_\pp=\pr_\pp(\Sing \cX)=\pr_\pp(\oL)\cup\pr_\pp(\oR)$.
We apply \EQN{LR} with \AXM{a1} and find that
$2\,\Delta(L)+2\,\Delta(R)=\delta(\cX)=8$ and $\Delta(L)+\Delta(R)=\delta(\cZ_\pp)=5$.
We arrived at a contradiction and thus we established that $\Sing \cX\nsubseteq\E$.

Recall from \Cref{prp:BGE8} that all complex irreducible components of $\Sing \cX$ are complex curves.
Let $V:=V_1\cup\cdots\cup V_r$
denote the decomposition of $\Sing \cX\setminus\E$ into
complex irreducible curve components so that
\[
\Sing \cX=V_1\cup\cdots\cup V_r\cup L\cup \oL\cup R\cup \oR.
\]
We set $W_i:=\pr_\pp(V_i)$ for $1\leq i\leq r$.
Since $\theta_\pp(\cX)=2$ we find that $W_i$ is not a complex point.
Thus the decomposition of the 1-dimensional part of $\Sing \cZ_\pp=\pr_\pp(\Sing \cX)$
into complex irreducible components is as follows:
\[
\Sing \cZ_\pp=W_1\cup\cdots\cup W_r\cup\pr_\pp(\oL)\cup\pr_\pp(\oR).
\]
We set $\Psi:=\Delta(L)+\Delta(R)$ and apply
\EQN{LR} and \AXM{a1} so that
\begin{equation}
\label{eqn:VW}
\Delta(V_1)+\cdots+\Delta(V_r)=8-2\,\Psi
\quad\text{and}\quad
\Delta(W_1)+\cdots+\Delta(W_r)=5-\Psi.
\end{equation}
In particular, we find that $2\leq \Psi\leq 3$
and thus
\begin{equation}
\label{eqn:r}
(r,\Psi)\in\{(1,2),~(2,2),~(3,2),~(1,3),~(2,3)\}.
\end{equation}
Notice that the choice of the complex incidence point~$\pp$
in the Clifford quartet
is not canonical,
and that the assertions leading to \EQN{LR}, \EQN{VW} and \EQN{r},
remain true
if we replace the five symbols in
$(\pp,L,\oL,R,\oR)$ with any one of following three options:
$(\op,\oL,L,\oR,R)$,
$(\qq,L,\oL,\oR,R)$ or
$(\oq,\oL,L,R,\oR)$.
We shall refer to this observation as \df{quartet symmetry}.

If $\sigma_\cX(V_i)=V_i$ for some $1\leq i\leq r$, then $V_i\subset \S^3$
intersects the hyperplane section~$\E$ of~$\S^3$ in
complex conjugate points. The following claim is therefore
a consequence of B\'ezout's theorem.

{\bf Claim~1.}
If $\deg V_i$ is odd for some $1\leq i\leq r$, then $\sigma_\cX(V_i)\neq V_i$.

We make a case distinction on $(r,\Psi)$ as listed at \EQN{r}.

Suppose that $(r,\Psi)=(1,2)$ so that by \EQN{VW} we have $\Delta(V_1)=4$ and $\Delta(W_1)=3$.
\AXMS{a2}{a4} imply that $\deg V=\deg V_1=4$ and $\deg W_1=3$.
From this we deduce that $\pp\in V$ and thus
$\op,\oq,\qq\in V$ as well by quartet symmetry.
We conclude that~$\PseudoType \cX=\Piii$.

Suppose that $(r,\Psi)=(2,2)$ so that by \EQN{VW} we have
$\Delta(V_1)+\Delta(V_2)=4$ and $\Delta(W_1)+\Delta(W_2)=3$.
Suppose by contradiction that $\deg V_1=1$ or $\deg V_2=1$.
In this case $\sigma_\cX(V_1)=V_2$ by Claim~1 and thus
$\Delta(V_1)=\Delta(V_2)=2$ by \AXM{a3}.
We arrived at a contradiction, since $\Delta(W_1)=\Delta(W_2)=2$ by \AXM{a4}.
Therefore, $2\leq\deg V_i\leq\Delta(V_i)$ for $i\in\{1,2\}$ by \AXM{a2} so that $\deg V_1=\deg V_2=2$.
We deduce from \AXM{a4} that $\deg W_1=\Delta(W_1)=1$ and $\deg W_2=\Delta(W_2)=2$.
Hence, $\pp\in V$ and thus
$\op,\oq,\qq\in V$ as well by quartet symmetry
so that $(V_1\cup V_2)\cap\E=\{\pp,\op,\qq,\oq\}$ by B\'ezout's theorem.
It now follows from \Cref{lem:A4} that $\PseudoType \cX=\Piv$.

Suppose by contradiction that $(r,\Psi)=(3,2)$
so that by \EQN{VW} we have $\Delta(V_1)+\Delta(V_2)+\Delta(V_3)=4$
and $\Delta(W_1)+\Delta(W_2)+\Delta(W_3)=3$.
By applying Claim~1 and \AXM{a2}, we find that
$\sigma_\cX(V_1)=V_2$, $\sigma_\cX(V_3)=V_3$ and $(\deg V_1,\deg V_2,\deg V_3)=(1,1,2)$.
It follows from \AXM{a2} that $\Delta(W_i)=\deg W_i=1$ for all $1\leq i\leq 3$.
Hence, $\deg W_3=\deg V_3-1$, $\deg V_1=\deg W_1$ and $\deg V_2=\deg W_2$.
This implies that $\pp,\op\in V_3\setminus(V_1\cup V_2)$.
We deduce from the quartet symmetry that $\qq,\oq\in V_1\setminus (V_2\cup V_3)$.
We arrived at a contradiction with \Cref{prp:pr}, since $\theta_\qq(\cX)=3$.

Suppose that $(r,\Psi)=(1,3)$
so that by \EQN{VW} we have $\Delta(V_1)=2$
and $\Delta(W_1)=2$.
It follows from Claim~1 and \AXMS{a2}{a4} that $\deg V_1=\deg W_1=2$,
which implies that~$\pp\notin V_1$.
By quartet symmetry, we find that $V_1\cap\{\pp,\op,\qq,\oq\}=\varnothing$.
The conic~$V_1$ intersects the hyperplane section $\E\subset\S^3$
in two complex conjugate points.
Therefore, we may assume \Wlog that $V\cap L=V\cap \oL=\varnothing$
and $|V\cap R|=|V\cap \oR|=1$ so that $\PseudoType \cX=\Pv$.

Suppose by contradiction that $(r,\Psi)=(2,3)$
so that by \EQN{VW} we have $\Delta(V_1)+\Delta(V_2)=2$ and $\Delta(W_1)+\Delta(W_2)=2$.
It follows from \AXM{a2} that $\deg C=1$ for all $C\in\{V_1,V_2,W_1,W_2\}$,
which implies that $\pp\notin V$.
Therefore, $\op,\qq,\oq\notin V$ by quartet symmetry.
We have $|V_1\cap\E|=1$ by B\'ezout's theorem.
Hence, may assume \Wlog that $V_1\cap L=V_1\cap R=\varnothing$
so that $[V_1]\cdot [L]=[V_1]\cdot [R]=0$ by \RL{C}{d}, where $[V_1]=\alpha\,\l_0+\beta\,\l_1$
for some $\alpha,\beta\geq 0$.
We arrived at a contradiction as $[L]=\l_0$ and $[R]=\l_1$ by assumption
so that either $[V_1]\cdot [L]>0$ or $[V_1]\cdot [R]>0$.

We considered all possible values for $(r,\Psi)$ and thus concluded the proof.
\end{proof}

\begin{proposition}
\label{prp:D1D2}
If $\cX$ contains a Bohemian quartet $(L,R,\vv)$, then
\[
\SingType\cX\in\{\Di,\,\Dii\}.
\]
\end{proposition}

\begin{proof}
Recall from \Cref{lem:d} that $\cZ_\vv$ is a dP surface.
It follows from the assumption that $\theta_\vv(\cX)\geq 4$
and thus $(\deg \cZ_\vv, \theta_\vv(\cX))=(4,4)$ by \Cref{prp:pr}.

Suppose that $M$, $M'$, $F$, $F'$ are as in \Cref{lem:dpencil}
and let
$\aa:=\pr_\vv(L)$,
$\oa:=\pr_\vv(\oL)$ and
$\bb:=\pr_\vv(R)$,
$\ob:=\pr_\vv(\oR)$
be pairs of complex conjugate points.
It follows from \RL{C}{d} that $M\cap R\neq\varnothing$ and $M\cap\oR\neq\varnothing$.
Hence, $\aa$ and $\oa$ are base points of the pencil~$F$.
We apply \RLS{C}{a}{C}{b}, and deduce that there exists
a component $W_\aa\subset B(\cZ_p)$ \st
\[
\Gamma(W_\aa)=\aa,\qquad
\Lambda(F)\cdot W_\aa\succ 0
\quad\text{and}\quad
\sigma_*(W_\aa)\neq W_\aa.
\]
Similarly, there exists
a component $W_\bb\subset B(\cZ_p)$ \st
\[
\Gamma(W_\bb)=\bb,\qquad
\Lambda(F')\cdot W_\bb\succ 0
\quad\text{and}\quad
\sigma_*(W_\bb)\neq W_\bb.
\]
The hypothesis of \Cref{lem:dp4} is satisfied so that
up to $\aut N(\cZ_\vv)$, we have
\begin{gather*}
\sigma_*(\l_0)=\l_0,\quad
\sigma_*(\l_1)=\l_1,\quad
\sigma_*(\p_1)=\p_2,\quad
\sigma_*(\p_3)=\p_4,\quad
\\
B(\cZ_\vv)=\{b_{13},b_{24},\bp_{14},\bp_{23}\},\quad
G(\cZ_\vv)=\{g_0,g_1,g_{12},g_{34}\},\quad
E(\cZ_\vv)=\{e_1,e_2,e_3,e_4\}.
\end{gather*}
See \cite{2024celest} for a
\href{https://github.com/niels-lubbes/celestial-surfaces#intersection-numbers-for-proposition-33}{table of intersection numbers}
between these classes.

Let $U$ denote the one-dimensional part of~$\Sing \cZ_\vv$.
Since $\delta(\cZ_\vv)=2$ by \Cref{lem:delta} it follows
from \AXMS{a1}{a2} at \Cref{def:delta} that $\deg U\leq 2$.

We observe that the two pairs of base points $\aa$, $\oa$ and $\bb$, $\ob$ lie on the absolute conic~$\cU_\vv$.
By \RL{C}{c} the surface $\cZ_p$ contains two pairs of complex
conjugate lines with classes $e_1$, $e_2$ and $e_3$, $e_4$.
By \RL{C}{d} these four complex lines intersect the base points and thus
lie in the hyperplane at infinity~$\cH_\vv$.
Hence, the two pairs of complex conjugate lines with classes $e_1$, $e_2$ and $e_3$, $e_4$
intersect additionally at points~$\rs$ and~$\rt$, \resp.
As $e_1\cdot\sigma_*(e_1)=e_3\cdot\sigma_*(e_3)=0$,
we deduce from \RL{C}{d} that
$\rs, \rt\in\Sing (\cZ_\vv)_\R$.
Notice that the stereographic projection~$\pr_\vv$ is an isomorphism outside~$\U_\vv$ and
the hyperplane at infinity~$\cH_\vv$.
Since $\cX$ does not contain isolated singularities by \Cref{prp:BGE8},
it follows from \RL{C}{a} that $\aa$, $\oa$, $\bb$ and $\ob$
are the only complex isolated singularities in~$\cZ_\vv$.
This implies that $\{\rs,\rt\}=U\cap \cH_\vv$ and thus $\deg U=2$.
Let $V$ denote the Zariski closure of~$\Sing\cX\setminus\{L,\oL,R,\oR\}$
so that $\pr_\vv(V)=U$. It follows from \Cref{prp:pr} that a $V$ is a double curve
and thus its stereographic projection~$U$ is a double curve as well.

The intersections between the
base points, complex lines and the components of~$U$
are illustrated by diagram $\Ei$ or $\Eii$ in \Cref{tab:E}, if $U$ is an irreducible and reducible double conic, \resp.
\begin{table}[!ht]
\caption{The incidences between the base points, complex lines and singular components of the quartic dP surface~$\cZ_\vv\subset\P^3$.
The 1-dimensional part $U\subset\Sing \cZ_\vv$ is either an irreducible conic ($\Ei$) or a reducible conic ($\Eii$).
The two pairs of complex conjugate lines are represented by dashed line segments and lie in the hyperplane at infinity~$\cH_\vv$.}
\label{tab:E}
\centering
\setlength{\tabcolsep}{7mm}
\begin{tabular}{cc}
{\bf E1} & {\bf E2}
\\
\begin{tikzpicture}[scale=0.6]
\draw[thick, colBrown] (-2,-2) to [out=45, in=220] (2,2) to [out=0, in=-90] (-2,-2);
\draw[red,densely dotted] (-3,2)    -- (4,2) node[black,right] {$e_1$};
\draw[red,densely dotted] (-3,-0.5) -- (4,3) node[black,right] {$e_2$};
\draw[red,densely dotted] (-2,3)    -- (-2,-4) node[black,below] {$e_3$};
\draw[red,densely dotted] ( 0.5,3)  -- (-3,-4) node[black,below] {$e_4$};
\draw[draw=black, fill=colP] (-2,2)     circle [radius=0.15] node[black,above left] {$b_{13}$};
\draw[draw=black, fill=colP] (-2,0)     circle [radius=0.15] node[black,above left] {$\bp_{23}$};
\draw[draw=black, fill=colP] (0,2)      circle [radius=0.15] node[black,above left] {$\bp_{14}$};
\draw[draw=black, fill=colP] (-0.7,0.7) circle [radius=0.15] node[black,above left] {$b_{24}$};
\draw[draw=black, fill=colO] (-2,-2)    circle [radius=0.15] node[black,left] {$\rt$};
\draw[draw=black, fill=colO] (2,2)      circle [radius=0.15] node[black,above] {$\rs$};
\end{tikzpicture}
&
\begin{tikzpicture}[scale=0.6]
\draw[thick, colGray] (-3,-2) -- (2.7,-2) node[right] {$g_{34}$};
\draw[thick, colGreen] (2,3) -- (2,-2.7) node[below] {$g_{12}$};
\draw[red, densely dotted] (-3,2)    -- (4,2) node[black,right] {$e_1$};
\draw[red, densely dotted] (-3,-0.5) -- (4,3) node[black,right] {$e_2$};
\draw[red, densely dotted] (-2,3)    -- (-2,-4) node[black,below] {$e_3$};
\draw[red, densely dotted] ( 0.5,3)  -- (-3,-4) node[black,below] {$e_4$};
\draw[draw=black, fill=colP] (-2,2)     circle [radius=0.15] node[black, above left] {$b_{13}$};
\draw[draw=black, fill=colP] (-2,0)     circle [radius=0.15] node[black, above left] {$\bp_{23}$};
\draw[draw=black, fill=colP] (0,2)      circle [radius=0.15] node[black, above left] {$\bp_{14}$};
\draw[draw=black, fill=colP] (-0.7,0.7) circle [radius=0.15] node[black, above left] {$b_{24}$};
\draw[draw=black, fill=colO] (-2,-2)    circle [radius=0.15] node[black, above left] {$\rt$};
\draw[draw=black, fill=colO] (2,2)      circle [radius=0.15] node[black, above left] {$\rs$};
\draw[draw=black, fill=colO] (2,-2)     circle [radius=0.15] node[black, above left] {};
\end{tikzpicture}
\end{tabular}
\end{table}

If $U$ is irreducible, then by B\'ezout's theorem this double conic must be a hyperplane
section of~$\cZ_\vv$ so that $[U]=-\k=2\,\l_0+2\,\l_1-\p_1-\p_2-\p_3-\p_4$.

Now suppose that $U$ is reducible and
let $U_1$ and $U_2$ denote its two complex double line components.
It follows from \RL{C}{c} and $|E(\cZ_\vv)|=4$ that the anticanonical projections of
the four complex lines in the anticanonical model~$\bA(\cZ_\vv)$
form the hyperplane section $\cZ_\vv\cap\cH_\vv$.
Hence, the complex double line~$U_i$ must for all $i\in\{1,2\}$ be the anticanonical projection of a complex irreducible conic
in~$\bA(\cZ_\vv)$ so that the class~$[U_i]$ belongs to~$G(\cZ_\vv)=\{g_0,g_1,g_{12},g_{34}\}$.
We observe that $\aa,\oa,\bb,\ob\notin U$
and we may assume \Wlog that $\rt\notin U_1$ and $\rs\notin U_2$.
It follows from \RL{C}{a} and \RL{C}{d} that $[U_1]\cdot b=[U_2]\cdot b=0$ for all
$b\in \{b_{13},b_{24},\bp_{14},\bp_{23}\}$,
$[U_1]\cdot e_3=[U_1]\cdot e_4=0$ and
$[U_2]\cdot e_1=[U_2]\cdot e_2=0$.
We find that $[U_1]=g_{12}$ and $[U_2]=g_{34}$.
Notice that $U_1\cap U_2\neq\varnothing$ by \RL{C}{d} and thus $U$ is a double conic.

Recall from \Cref{prp:BGE8} and \RL{C}{b} that $\cX$ is 2-circled
and thus \ref{i} at \Cref{def:st} holds.
By B\'ezout's theorem, the double conic~$U\subset \cZ_\vv$ is a hyperplane section,
and thus its preimage $V\subset\cX$ is a hyperplane section as well.
As $\cX$ is a dP surface, we find that $[V]=-\k_{\cX}=2\,\l_0+2\,\l_1$
and thus $V$ is a either a reducible or irreducible double curve of degree~$4$.
Recall from \Cref{lem:dpencil} that
$[M_i]=[F_i]=\l_0$ and $[M'_i]=[F'_i]=\l_1$ for general $i\in\P^1$.
By assumption, $[L]=[\oL]=\l_0$, $[R]=[\oR]=\l_1$
and thus the first three items of \ref{iv} are satisfied.

First, we suppose that $V$ is reducible.
Let $V=V_1\cup V_2$ \st $U_1=\pr_\vv(V_1)$ and $U_2=\pr_\vv(V_2)$.
We find that \ref{ii} is satisfied for the $\Dii$-diagram in \Cref{tab:s}
and \ref{iii} is a consequence of \Cref{prp:pr}.
By \RL{C}{d} and B\'ezout's theorem, we have for all $j\in\{1,2\}$ and general $i\in\P^1$
that $|F_i\cap U_j|=|F'_i\cap U_j|=1$, which implies that $|M_i\cap V_j|=|M'_i\cap V_j|=1$.
Hence, $[V_1]=[V_2]=\l_0+\l_1$ and thus both \ref{iv} and \ref{v} hold so that
$\SingType\cX=\Dii$.

Finally, suppose that $V$ is irreducible.
We find that \ref{ii} is satisfied for the $\Di$-diagram in \Cref{tab:s}
and \ref{iii} is a consequence of \Cref{prp:pr}.
We already established that $V\subset\cX$ is
a hyperplane section so that \ref{iv} holds.
Therefore, \ref{v} is a consequence of B\'ezout's theorem so that $\SingType \cX=\Di$.
\end{proof}

We are now ready to prove the main result of this article.

\begin{proof}[Proof of \Cref{thm:s}]
See \Cref{rmk:s} for Assertions~\ref{thm:s:a} and~\ref{thm:s:b}.
For \ASN{thm:s:c} we assume \Cref{ntn:X} and thus $X=\cX$.
Since the singular type is a M\"obius invariant,
it follows from \Cref{thm:quartet} that $\cX$ contains
\Wlog either a Bohemian or Cliffordian quartet.
If $\cX$ contains a Cliffordian quartet, then
$\SingType\cX\in\{\Diii,\,\Div,\,\Dv\}$ by \Cref{prp:P,prp:D345}.
If $\cX$ contains a Bohemian quartet, then
$\SingType\cX\in\{\Di,\,\Dii\}$ by \Cref{prp:D1D2}.
\end{proof}

\begin{proof}[Proof of \Cref{cor:BC}]
The 1st item follows from \Cref{thm:s},
the 2nd and 3rd items follow from \citep[Propositions~27 and~30]{2024fact},
and the 4th item follows from \citep[Proposition~25]{2024great}.
Let us consider the remaining 5th item
and suppose by contradiction that $\SingType\cX\in\{\Dii,\,\Div\}$.
In this case, there exists a double circle $V\subset\Sing\cX$
\st $C_\R\cap V=\varnothing$.
We know that $[V]=\l_0+\l_1$ and $[C]\in\{\l_0,\l_1\}$ by \ref{iv} at \Cref{def:st}.
Hence, by \Cref{def:class} there exist $V'\subset \P^1\times\P^1$ of bidegree $(1,1)$
and $C'\subset\P^1\times\P^1$ of bidegree $(1,0)$ \st $\varphi(V')=V$, $\varphi(C')=C$
and $C'\cap V'$ consists of pairs of complex conjugate points.
We arrived at a contradiction as $[V]\cdot[C]=|C'\cap V'|=1$
and thus $C'_\R\cap V'\neq\varnothing$.
\end{proof}

\begin{proof}[Proof of \Cref{cor:deg}]
Up to M\"obius equivalence, we may assume that the degree of the Euclidean model~$\bU(X)$
is equal to the degree of the stereographic projection~$\pr_p(X)$ for some $p\in\S^3_\R$.
By \Cref{thm:s}\ref{thm:s:a}, we have $\deg X\in\{2,4,8\}$
and thus $1\leq \deg\pr_p(X)\leq 8$.
Suppose by contradiction that $\deg\pr_p(X)=5$ so that $\deg X=8$.
In this case $p\in\Sing X_\R$ has multiplicity $m=3$, since $\deg\pr_p(X)=\deg X-m$
as a direct consequence of the definitions.
It follows from \Cref{thm:s}\ref{thm:s:c} and \ref{iii} at \Cref{def:st}
that $p$ lies on $2$ complex double lines.
We arrived at a contradiction as the complex point~$p$ is non-real.
\end{proof}

\section{The classification of visible types}
\label{sec:VisType}

In this section, we prove \Cref{cor:VisType} by showing that
\begin{gather*}
\VisType\cX
\in
\{
\Di[=],\,
\Dii[+],\,
\Diii[~],\,
\Diii[\circ],\,
\Diii[=],\,
\Diii[-],\,
\\
\Div[\infty],\,
\Div[\alpha],\,
\Div[\sigma],\,
\Div[+],\,
\Div[=],\,
\Div[-],\,
\Dv[~],\,
\Dv[\circ]
\}.
\end{gather*}

We define $\VisPoints U$ as the intersection of the set of visible points of $\cX$ with the subset~$U\subseteq\cX$.
An \df{arc} is homeomorphic to the unit interval~$[0,1]\subset\R$.

\begin{lemma}
\label{lem:D12}
If $\SingType\cX\in\{\Di,\Dii\}$, then
\[
\VisType\cX\in\{\Di[=],\,\Dii[+]\}
\]
and there exist topological spheres $S$ and $S'$ \st
\[
\VisPoints\cX=S\cup S'
\quad\text{and}\quad
\VisPoints(\Sing\cX)=S\cap S'.
\]
\end{lemma}

\begin{proof}
By \Cref{cor:BC}, we may assume up to M\"obius equivalence that
the Euclidean model~$\bU(\cX)$ is equal to $A+B$
for some circles $A,B\subset\R^3$, where
$A+B$ denotes the Zariski closure of~$\set{a+b}{a\in A,~b\in B}$.
Let $\plane(A)$ denote the plane in $\R^3$ that is spanned by $A$.
If $b\in B$, then its \df{complementary index}~$b^*$ is defined as an element in $B$ \st
\[
\plane(A+\{b\})\cap\bU(\cX)=(A+\{b\})\cup (A+\{b^*\}).
\]
We observe that for all $b\in B$, we have that $A+\{b\}$ is a circle in~$\bU(\cX)$ \st
\begin{Mlist}
\item $\plane(A+\{b\})$ is parallel to $\plane(A)$,
\item $\radius(A+\{b\})=\radius(A)$,
\item $\cent(A+\{b\})$ lies on the circle $B+\{\cent(A)\}$, and
\item $\radius(B+\{\cent(A)\})=\radius(B)$.
\end{Mlist}
As a straightforward consequence, we find that the complementary index exists and is unique.
Moreover, there exists exactly two elements~$\rs,\rt\in B$ \st $\rs^*=\rs$ and $\rt^*=\rt$.
In \Cref{fig:D12}, we illustrated examples for $A+B$
together with circles $A+\{\beta\}$ \st $\beta\in\{\bb,\bb^*,\aa,\aa^*,\rs,\rt\}$
for some $\aa,\bb\in B$ \st $\bb\neq\bb^*$ and $\aa\neq\aa^*$.
\begin{figure}[!ht]
\centering
\csep{3mm}
\begin{tabular}{cccc}
\DIb &
\fig{3}{3}{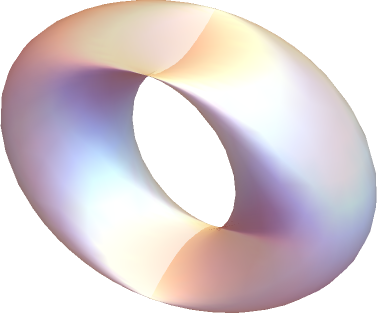} &
\fig{3}{3}{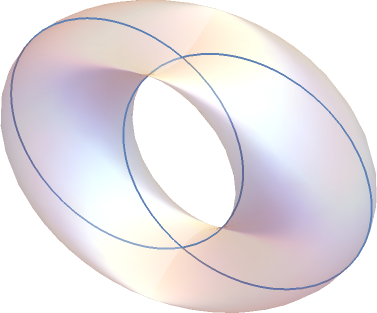} &
\fig{3}{3}{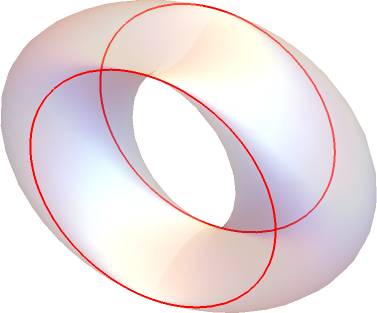} \\
\DIIb &
\fig{3}{3}{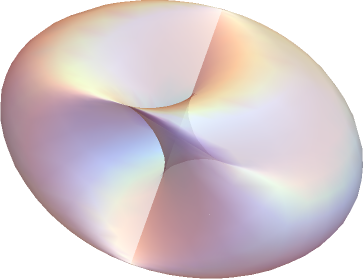} &
\fig{3}{3}{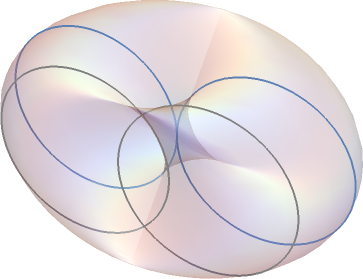} &
\fig{3}{3}{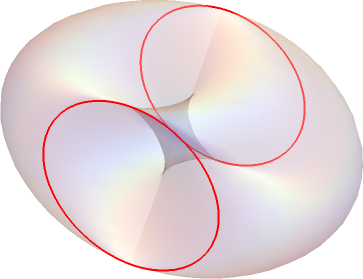}
\end{tabular}
\caption{The visible points in the surface $A+B$ for some circles $A,B\subset\R^3$
\st $\radius(A)>\radius(B)$ in the 1st row and $\radius(A)=\radius(B)$
in the 2nd row. The 3rd and 4th columns illustrate the circles $A+\{\beta\}$ in
case $\beta\in\{\bb,\bb^*,\aa,\aa^*\}$ and $\beta\in\{\rs,\rt\}$, \resp.
The circles $A+\{\bb\}$ and $A+\{\bb^*\}$ intersect in two points.
The circles $A+\{\aa\}$ and $A+\{\aa^*\}$ intersect tangentially in one point
and are only illustrated in the 2nd row.}
\label{fig:D12}
\end{figure}

Since $A+B$ contains many plane sections consisting of two circles, we deduce that $\deg(A+B)=4$.
Therefore, the center $\vv\in\S^3$ of stereographic projection is the unique point in $\Sing\cX$ of multiplicity $4$.
It follows that either
\begin{Mlist}
\item
$\SingType\cX=\Di$  and $\Sing\bU(\cX)$ consists of an irreducible conic, or
\item
$\SingType\cX=\Dii$ and $\Sing\bU(\cX)$ consists of two coplanar lines.
\end{Mlist}
First, let us suppose that $\radius(A)>\radius(B)$.
We already established that the center of the circle~$A+\{b\}$ lies for all $b\in B$
on a circle whose radius is equal to the radius of $B$.
Since $\radius(A+\{b\})=\radius(A)>\radius(B)$,
it follows that
for all $b\in B\setminus\{\rs,\rt\}$ we have
$
(A+\{b\})\cap (A+\{b^*\})=\{\pp_b,\qq_b\}.
$
We deduce that the sets~$\{\pp_b\}_{b\in B}$ and $\{\qq_b\}_{b\in B}$ form two disjoint arcs
with end points in $A+\{\rs,\rt\}$,
which implies that $\VisPoints(\Sing\bU(\cX))$ is homeomorphic to the symbol $=$.
Now suppose by contradiction that $\VisType\cX=\Dii[=]$ and
let $L\subset\VisPoints(\Sing\bU(\cX))$ be one of the two line segments.
We deduce using \Cref{fig:D12b} that there exists a circular arc~$C\subset A$
\st
for all $\alpha\in C$ there exists $\alpha^*\in C$ \st
\[
\plane(\{\alpha\}+B)=\plane(\{\alpha^*\}+B)
\quad\text{and}\quad
|(\{\alpha\}+B)\cap (\{\alpha^*\}+B)\cap L|=2.
\]
Since $\plane(\{\alpha\}+B)$ is parallel to $\plane(B)$
and meets $L$ in two points for all $\alpha\in C$,
we find that $L$ cannot be contained in a line, and thus we arrived at a contradiction.
The only remaining possibility is that $\VisType\cX=\Di[=]$.

\begin{figure}[!ht]
\centering
\csep{3mm}
\begin{tabular}{cccc}
\fig{3}{3}{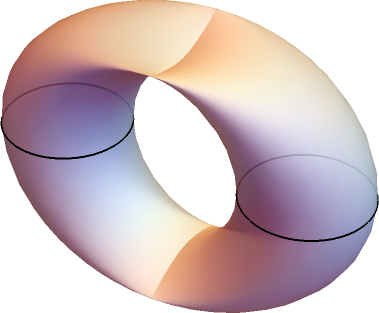} &
\fig{3}{3}{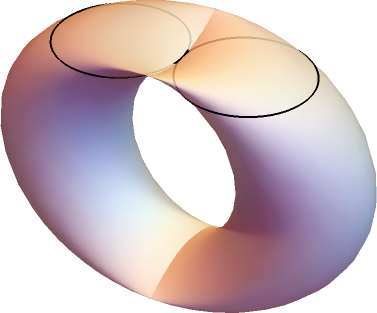} &
\fig{3}{3}{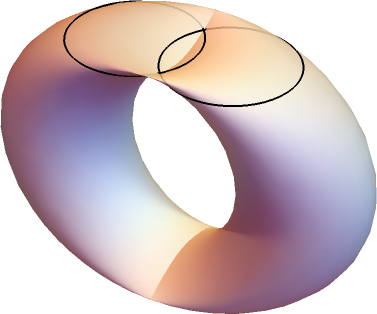} &
\fig{3}{3}{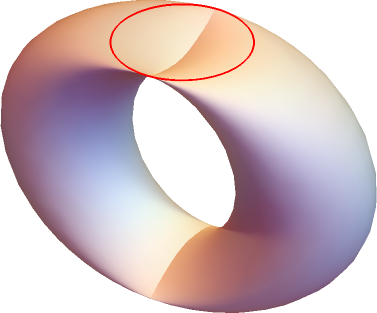}
\end{tabular}
\caption{The visible points in the surface $A+B$ for some circles $A,B\subset\R^3$
\st $\radius(A)>\radius(B)$ together with pairs of coplanar circles that
meet at an arc in the visible singular locus in either zero, one or two points.}
\label{fig:D12b}
\end{figure}

Finally, let us suppose that $\radius(A)=\radius(B)$.
In this case, there exists $\aa,\aa^*\in B$ \st
$\plane(A+\{\aa\})=\plane(A+\{\aa^*\})$
and $(A+\{\aa\})\cap(A+\{\aa^*\})=\{\qq\}$,
where $\qq\in\VisPoints(\Sing\bU(\cX))$ (see \Cref{fig:D12}).
If $b\in B$ \st $\plane(A+\{b\})$ lies between either
\begin{Mlist}
\item
$\plane(A+\{\rs\})$ and  $\plane(A+\{\aa\})$, or
\item
$\plane(A+\{\aa\})$ and  $\plane(A+\{\rt\})$,
\end{Mlist}
then $(A+\{b\})\cap (A+\{b^*\})$ consists of two points that trace out two arcs,
and these arcs meet at the point~$\qq$.
Hence, $\VisPoints(\Sing\bU(\cX))$ is homeomorphic to the symbol~$+$.
This implies that $\VisType\cX=\Dii[+]$.

We established that $\VisType\cX\in\{\Di[=],\,\Dii[+]\}$.
For the remaining assertion we observe that
the parallel plane sections of~$\bU(\cX)$ that lie between
$\plane(A+\{\rs\})$ and $\plane(A+\{\rt\})$
trace out two topological spheres $S$ and $S'$
so that $\VisPoints(\bU(\cX))=S\cup S'$ and $\VisPoints(\Sing\bU(\cX))=S\cap S'$.
Since $\VisPoints(\bU(\cX))$ is compact, we deduce that $\VisPoints\cX$
is homeomorphic to $\VisPoints(\bU(\cX))$. This concludes the proof.
\end{proof}

\begin{notation}
\label{ntn:X2}
In what follows, we assume in addition to \Cref{ntn:X} that
\begin{Mlist}
\item $\SingType\cX\in\{\Diii,\,\Div,\,\Dv\}$,
\item $V\subset \Sing\cX$ is a rational double curve,
\item $\varphi\c\P^1\times\P^1\to\cX$ is a desingularization, and
\item $W\subset\P^1\times\P^1$ is the 1-dimensional component in~$\varphi^{-1}(V)$.
\end{Mlist}
If $C\subset\P^1\times\P^1$ is a complex and possibly reducible curve,
then we denote by $[C]$ its divisor class in the N\'eron-Severi lattice~$N(\cX)$.
Notice that $[V]=[W]$ as a direct consequence of the definitions.
\END
\end{notation}

\begin{lemma}
\label{lem:vis}
We have
$\VisPoints\cX=\varphi(\P^1_\R\times\P^1_\R)$
and
\[
\VisPoints V=\varphi(W_\R)=V_\R\setminus\set{\varphi(\qq)}{\qq\in W\setminus W_\R,~\varphi(\qq)=\varphi(\oq) }.
\]
\end{lemma}

\begin{proof}
The preimages \wrt $\varphi$ of hyperplane sections of $\cX_\R$
correspond to curves in $\P^1_\R\times\P^1_\R$ without isolated points and thus
$\VisPoints\cX=\varphi(\P^1_\R\times\P^1_\R)$ by the definition of ``visible point''.
As a direct consequence, we obtain that
\[
\VisPoints V=\varphi(\varphi^{-1}(V)\cap \P^1_\R\times\P^1_\R).
\]
This implies that the preimage \wrt $\varphi$ of the points in~$V_\R$ that are not
visible must contain at least one pair of complex conjugate points.

Now suppose by contradiction that $\varphi(\varphi^{-1}(V)\cap \P^1_\R\times\P^1_\R)\neq \varphi(W_\R)$.
In this case there exists $\vv\in\varphi^{-1}(V)\setminus W$ \st $\varphi(\vv)\in V$.
Hence, $\vv\in \varphi^{-1}(V\cap M)$ for some singular component $M\subset\Sing\cX$ different from $V$.
Recall that $\SingType\cX\in\{\Diii,\,\Div,\,\Dv\}$ by assumption.
Since the complex double lines in $\cX$ do not have real points,
we find that $\SingType\cX=\Div$ so that $V$ and $M$ are both double circles
with class~$\l_0+\l_1$.
We deduce that $\varphi^{-1}(V\cap M)=W\cap \varphi^{-1}(M)$.
We arrived at a contradiction as $\vv\in\varphi^{-1}(V\cap M)$ and $\vv\notin W$.
\end{proof}

\begin{lemma}
\label{lem:ram}
Suppose that $W$ is irreducible with $|W_\R|=\infty$
and $\varphi|_W\c W\to V$ a 2:1 morphism
with $c$ complex ramification points.
\begin{claims}
\item\label{lem:ram:a}
$\pp\in W_\R$ is a ramification point of $\varphi|_W$
if and only if either
$\pp\in\Sing W$
or
$\varphi(\pp)$ is the end point of an arc in $\VisPoints V$.
\item\label{lem:ram:b}
$c\leq [W]\cdot ([W]-2\,\l_0-2\,\l_1)-|\Sing W|+4$.
\item\label{lem:ram:c}
If $[W]\cdot ([W]-2\,\l_0-2\,\l_1)=-2$, then $\Sing W=\varnothing$.
\end{claims}
\end{lemma}

\begin{proof}
\ref{lem:ram:a}
We first show the $\Leftarrow$ direction.
Notice that $\Sing W$ consists of ramification points by definition.
Now suppose that $\pp\notin\Sing W$ so that $\varphi(\pp)$ is an end point of an arc.
It follows from \Cref{lem:vis} that for all $v\in V_\R\setminus \VisPoints V$
the fiber~$\varphi^{-1}(v)$ consists of two complex conjugate points.
Moreover, for all $v\in \VisPoints V$ the preimage~$\varphi^{-1}(v)$
contains at least one point in $W_\R$.
We deduce that $\varphi(\pp)$
is a real branching point of $\varphi|_W$ and thus $\pp$ is a ramification point.

The $\Rightarrow$ direction is a straightforward consequence of the definitions.
For intuition, one could think of the projection of a circle in $\R^2$ to a line segment.
See also \citep[Figure~IV.2.13]{1977}.
We remark that unlike a nodal singularity of~$W$, a cuspidal singularity
may be send to an end point of an arc.

\ref{lem:ram:b}
Let $p_g(W)$ and $p_a(W)$ denote the \df{geometric genus} and \df{arithmetic genus}, \resp.
Now suppose that $f\c\widetilde{W}\to W$ is a desingularization of $W$ so that $\widetilde{W}$ is smooth,
$p_g(\widetilde{W})=p_g(W)$ and $f$ is a biregular isomorphism outside $\Sing W$.
Let $\tilde{c}$ denote the number of complex ramification points of the composition $\varphi|_W\circ f$.
As $V$ is rational by assumption, we have $p_g(V)=0$
and thus the Hurwitz formula at \citep[Corollary~IV.2.4]{1977} implies that
\[
\tilde{c}\leq 2\,p_g(W)+2.
\]
If $q$ is a complex ramification point of the morphism~$\varphi|_W\circ f$,
then $f(q)$ is a complex ramification point of~$\varphi|_W$ as a direct consequence of the definitions.
However, $p\in\Sing W$ may be a complex ramification point of $\varphi|_W$ but not of $\varphi|_W\circ f$.
Hence, we deduce that
\[
c\leq 2\,p_g(W)+2+|\Sing W|.
\]
It follows from the genus formula at \citep[Exercise~IV.1.8a]{1977} that
\[
0\leq p_g(W)\leq p_a(W)-|\Sing W|,
\]
and thus
\[
c\leq 2\,p_a(W)-|\Sing W|+2.
\]
Since $\k=-2\,\l_0-2\,\l_1$ is the canonical class of $\P^1\times\P^1$,
it follows from the arithmetic genus formula at \citep[Proposition~V.1.5 and Exercises~IV.1.8 and~V.1.3]{1977}
that
\[
p_a(W)=\tfrac{1}{2}\,[W]\cdot([W]+\k)+1.
\]
We substitute this formula for $p_a(W)$ in the previous
upper bound for $c$ and recover the upper bound of the main assertion.

\ref{lem:ram:c}
If $[W]\cdot([W]+\k)=-2$, then $\Sing W=\varnothing$ as
$p_a(W)=\tfrac{1}{2}\,[W]\cdot([W]+\k)+1=0$ and $0\leq p_g(W)\leq  p_a(W)-|\Sing W|=-|\Sing W|$.
\end{proof}

\begin{lemma}
\label{lem:Wfin}
~
\begin{claims}
\item\label{lem:Wfin:a}
If $|W_\R|<\infty$, then $W_\R=\VisPoints V=\varnothing$.
\item\label{lem:Wfin:b}
If $[V]\cdot\l_1=1$, then $|W_\R|=\infty$ and $\VisPoints V=\varphi(W_\R)$.
\end{claims}
\end{lemma}

\begin{proof}
Let $\rho_1,\rho_2\c\P^1\times\P^1\to\P^1$ denote the projections to the
first and second component.
For all $i\in\P^1$, we consider the fibers
$F_i:=\rho^{-1}_1(i)$ and $G_i:=\rho^{-1}_2(i)$.
We may assume \Wlog that $[F_i]=\l_0$ and $[G_i]=\l_1$ for all $i\in\P^1$.
Recall from \Cref{prp:BGE8} that $G(\cX)=\{\l_0,\l_1\}$
and thus the class of a circle is either $\l_0$ or $\l_1$.
As $\S^3_\R$ does not contain lines, we deduce
that $\varphi(\aa)\neq\varphi(\bb)$ for all~$\aa,\bb\in F_i$ and~$i\in\P^1_\R$.
This implies that $C\subset\cX$ is a circle \st $[C]=\l_0$
if and only if $C=\varphi(F_i)$ for some point~$i\in\P^1_\R$.
Similarly, $C\subset\cX$ is a circle \st $[C]=\l_1$
if and only if $C=\varphi(G_i)$ for some point~$i\in\P^1_\R$.

\ref{lem:Wfin:a}
Suppose by contradiction that there exists $\pp\in W_\R$.
Let $U_\pp\subset \P^1_\R\times\P^1_\R\cong S^1\times S^1$ be an analytic neighborhood of~$\pp$
\st $U_\pp$ is biregular isomorphic to $\varphi(U_\pp)$.
There exists a unique pair $(i,j)\in(\P^1_\R)^2$ \st $\pp\in F_i\cap G_j$ and
thus through the point~$\varphi(\pp)$ pass the circles $\varphi(F_i)$ and $\varphi(G_j)$.
Since $\varphi(\pp)\in\Sing\cX_\R$, we find that the tangent plane at $\varphi(\pp)$ is not uniquely defined
so that the circles $\varphi(F_i)$ and $\varphi(G_j)$ must have a common tangent line
at $\varphi(\pp)$.
Now observe that $\varphi(U_\pp)$ contains through each point
in $\varphi(U_\pp)\setminus\{\varphi(\pp)\}$ two transversal circular arcs.
We arrived at a contradiction as the two circular arcs passing through $\varphi(\pp)$
must be transversal as well.
We established that $W_\R=\varphi(W_\R)=\varnothing$ so that $\VisPoints V=\varnothing$ by \Cref{lem:vis}.

\ref{lem:Wfin:b}
Since $[W]\cdot\l_1=[W]\cdot [F_i]=|W\cap F_i|=1$ for all $i\in\P^1$,
it follows that $W_\R\cap F_i\neq\varnothing$ so that $|W_\R|=\infty$.
\end{proof}

\begin{lemma}
\label{lem:Wred}
If $W$ is reducible, then $\VisPoints V=V_\R$
with either
$W_\R=V_\R=\varnothing$
or
$|W_\R|=|V_\R|=\infty$.
\end{lemma}

\begin{proof}
As $V$ is a rational double curve by assumption, there exist complex curves~$W'$ and~$W''$
\st $W=W'\cup W''$ and $\varphi(W')=\varphi(W'')=V$.
Hence, the fiber $\varphi^{-1}(v)$ \wrt a general complex point~$v\in V$
consists of two complex points $\pp,\qq\in W$
\st $\pp\in W'$ and $\qq\in W''$.
It follows from \Cref{lem:vis} that if $v\in V_\R$,
then either
$\pp,\qq\in W_\R$ and $v$ is visible, or
$\qq=\op\notin W_\R$ and $v$ is not visible.
This implies that $\VisPoints V=\varphi(W_\R)=\varphi(W'_\R)$.
Now if $|W_\R|=\infty$, then $|W_\R'|=\infty$ and $\varphi(W_\R')=V_\R$
so that $\VisPoints V=V_\R$ with $|V_\R|=\infty$.
If $|W_\R|<\infty$, then $W_\R=\varnothing$ by \Cref{lem:Wfin}
and thus $\VisPoints V=V_\R=\varnothing$.
\end{proof}

\begin{lemma}
\label{lem:tri}
Exactly one of the following three scenarios holds true:
\begin{enumerate}[topsep=0pt,itemsep=0pt,leftmargin=10mm,label=(\Roman*),ref=(\Roman*)]
\item\label{tri:i}
$\VisPoints V=V_\R$ and $|V_\R|=\infty$.
\item\label{tri:ii}
$\VisPoints V=\varnothing$ and $[V]\cdot \l_1\neq 1$.
\item\label{tri:iii}
$\VisPoints V$ consists of $\alpha$ disjoint arcs, where
\[
1\leq \alpha\leq \tfrac{1}{2}\,[V]\cdot([V]-2\,\l_0-2\,\l_1)+2
\]
and $[V]\cdot(\l_0+\l_1)\neq\tfrac{1}{2}\deg V$.

\end{enumerate}
\end{lemma}

\begin{proof}
If $W$ is reducible, then either \ref{tri:i} or \ref{tri:ii} holds by \RL{Wfin}{b} and \Cref{lem:Wred}.
If $|W_\R|<\infty$, then \ref{tri:ii} holds by \Cref{lem:Wfin}.

In the remainder of this proof we assume that
$W$ is irreducible with
$|W_\R|=\infty$ and
$\VisPoints V\notin\{V_\R,\varnothing\}$.

Because $V$ is a double curve, the restriction~$\varphi|_W\c W\to V$
is a $d:1$ morphism \st $d\in\{1,2\}$.
Since $\VisPoints V\notin\{V_\R,\varnothing\}$, we deduce that $\alpha\geq 1$ and $d=2$.
It follows from \RL{ram}{b} with $[V]=[W]$ that
\[
2\alpha\leq [V]\cdot([V]-2\,\l_0-2\,\l_1)+4.
\]
Let $\eta\c\bA(\cX)\to\cX$ denote the anticanonical projection
and let $\psi_{-k}\c\P^1\times\P^1\to\P^8$ denote the map associated to the anticanonical class~$-\k=2\,\l_0+2\,\l_1$.
The singular component~$V\subset\P^4$ is via $\eta$
a linear projection of the curve~$U:=\psi_{-\k}(W)\subset\P^8$ in the anticanonical model~$\bA(\cX)$.
As a direct consequence of the definitions, we find that
$
\deg U=-\k\cdot [W]=2\cdot [V]\cdot(\l_0+\l_1).
$
Since $d=2$, we deduce that $\deg U\neq \deg V$ and thus
\[
[V]\cdot(\l_0+\l_1)\neq \tfrac{1}{2}\deg V.
\]
We established that \ref{tri:iii} holds.
This concludes the proof as the three cases do not overlap
due to their specifications.
\end{proof}

\begin{lemma}
\label{lem:circle}
Suppose that $V\subset\Sing\cX$ is a double circle \st $[V]=\l_0+\l_1$ and $\VisPoints V\neq V_\R$.
Then $\VisPoints V$ consists of an arc with an end point~$\pp$
if and only if
$\pp$ is a real branching point of $\varphi|_W$.
\end{lemma}

\begin{proof}
It follows from \RL{Wfin}{b} that $|W_\R|=\infty$.
Since $V$ is a double circle and $\VisPoints V=\varphi(W_\R)$ by \Cref{lem:vis},
we deduce that $\varphi|_W$ is 2:1 morphism.
It follows from \Cref{lem:Wred} that $W$ is irreducible.
We apply \RL{ram}{c} and find that $\Sing W=\varnothing$.
The proof is now concluded by \RL{ram}{a}.
\end{proof}

\begin{proof}[Proof of \Cref{cor:VisType}]
We apply \Cref{thm:s} and find that
\[
\SingType\cX\in\{\Di,\,\Dii,\,\Diii,\,\Div,\,\Dv\},
\]
where $X=\cX$ (see \Cref{ntn:X}).
If $\SingType\cX\in\{\Di,\,\Dii\}$, then the proof is concluded by \Cref{lem:D12}.
We make a case distinction on the remaining singular types for~$\cX$
and assume \Cref{ntn:X2}.

First, we suppose that $\SingType\cX=\Diii$.
Let $V\subset\Sing X$ be the double twisted quartic with $[V]=2\,\l_0+2\,\l_1$.
Since $V\cong\P^1$, we find that $V_\R$ is either the empty set or homeomorphic to the symbol~$\circ$.
We apply \Cref{lem:tri} to $V$ and deduce that
\[
\VisType\cX\in\{\Diii[~],\,\Diii[\circ],\,\Diii[=],\,\Diii[-]\}.
\]
Next, we suppose that $\SingType\cX=\Dv$.
Let $V\subset\Sing\cX$ denote the double conic with $[V]\in\{\l_0,\,2\,\l_0\}$.
As before, we apply \Cref{lem:tri} to $V$ and find that
\[
\VisType\cX\in\{\Dv[~],\Dv[\circ]\}.
\]
In the remainder of this proof, we assume that $\SingType\cX=\Div$.
Let $V_1,V_2\subset\Sing\cX$ be the double circles so that $[V_1]=[V_2]=\l_0+\l_1$ and
$V_1\cap V_2=\{\rr\}\subset \cX_\R$.
Suppose that $W_i\subset \P^1\times\P^1$
denotes the 1-dimensional component in the preimage~$\varphi^{-1}(V_i)$ so that $[W_i]=[V_i]$
where $i\in\{1,2\}$.
It follows from \RL{Wfin}{b} that $\VisPoints V_i\neq\varnothing$.
It follows from \Cref{lem:tri}
that $\VisPoints V_i$ is homeomorphic to either the symbol $\circ$ or the symbol $-$.
We make a case distinction.
\begin{Mlist}
\item
If both $\VisPoints V_1$ and $\VisPoints V_2$ are homeomorphic to the symbol $\circ$,
then
\[
\VisType\cX=\Div[\infty].
\]
\item
Suppose that $\VisPoints V_1$ is homeomorphic to the symbol~$-$
and $\VisPoints V_2$ is homeomorphic to the symbol~$\circ$
so that
\[
\VisType\cX\in\{\Div[\alpha],\,\Div[\sigma],\,\Div[-\,\circ]\}.
\]
Since $\rr\in\VisPoints V_2=(V_2)_\R$ it follows that $\rr\in\VisPoints V_1$ and thus
\[
\VisType\cX\neq\Div[-\,\circ].
\]
Notice that if $\VisType\cX=\Div[\sigma]$, then the arc $\VisPoints V_1$ has $\rr$ as an end point.
\item
Finally, suppose that both $\VisPoints V_1$ and $\VisPoints V_2$ are homeomorphic to the symbol~$-$
so that
\[
\VisType\cX\in\{\Div[+],\,\Div[=],\,\Div[-],\,\Div[\lambda]\}.
\]
Now suppose by contradiction that $\VisType\cX=\Div[\lambda]$.
We may assume \Wlog that $\rr$ is an end point of the arc~$\VisPoints V_1$.
It follows from \Cref{lem:circle} that $\rr$ is a branch point of $\varphi|_{W_1}$.
This implies that $|W_1\cap W_2|=1$ and thus $\rr$ must also be a branch point of $\varphi|_{W_2}$.
We arrived at a contradiction with \Cref{lem:circle}
as $\rr$ must also be an end point of the arc $\VisPoints V_2$.
Notice that if $\VisType\cX=\Div[-]$, then the arcs $\VisPoints V_1$ and $\VisPoints V_2$
have a common end point~$\rr$.
\end{Mlist}
We considered all possible singular types for $\cX$
and thus concluded the proof.
\end{proof}

\section{Identifying celestial surfaces}
\label{sec:id}

In this section, we present methods for computing
the singular and visible types of celestial surfaces.
As an application, we verify each row of \Cref{tab:D}
and clarify \Cref{tab:D123,tab:D45}.
Moreover, \Cref{prp:varphi} refines the characterization of the
birational morphism $\varphi$ at \ref{iv} of \Cref{def:st}.

We call a rational map $\P^1\times\P^1\dto\P^n$ \df{biquadratic} if its components
are forms of bidegree~$(2,2)$.
The \df{left images} and \df{right images} of~$\varphi$ are
defined as $\varphi(\{p\}\times\P^1)$ and $\varphi(\P^1\times\{p\})$ for some $p\in\P^1$, \resp.
Left and right images are real unless explicitly stated otherwise.

\begin{proposition}
\label{prp:varphi}
For all biquadratic birational morphisms~$\varphi\c \P^1\times\P^1\to X$
\st $X\subset\S^3$ is a celestial surface of degree~$8$
the following holds:
\begin{Mlist}
\item The left and right images of $\varphi$ are circles.
\item If $C\subset X$ is a circle \st $C\nsubseteq\Sing X$, then $C$ is a left or right image of $\varphi$.
\item
The morphism~$\varphi$ induces a biregular isomorphism
$\P^1\times\P^1\setminus U\cong X\setminus\Sing X$, where $U:=\varphi^{-1}(\Sing X)$.
\item If $\widetilde{\varphi}\c\P^1\times\P^1\to X$ is a biquadratic birational morphism,
then there exists a biregular automorphism~$\alpha\in\aut\P^1\times\P^1$ \st $\varphi=\widetilde{\varphi}\circ\alpha$.
\end{Mlist}
\end{proposition}

\begin{proof}
Let $\cV$ denote the $9$-dimensional vector space of bidegree~$(2,2)$ forms
and let $f\c \P^1\times\P^1\to \P^8$ be a rational map whose components form
a basis for $\cV$.
Since $f$ is associated to the anticanonical class of~$\P^1\times\P^1$,
it follows from \citep[Theorem~8.3.2(iii) and \textsection8.4.1]{2012dol}
that $f$ is a biregular isomorphism onto its anticanonical model $\bA(X):=f(\P^1\times\P^1)$.
By assumption $\varphi$ is a biquadratic and thus there exists a
linear projection $g\c \P^8\dto \P^4$ \st $\varphi=\eta\circ f$, where
$\eta:=g|_{\bA(X)}\c \bA(X)\to X$ is a birational morphism and $\deg \bA(X)=\deg X$.
The left and right images of $f$ are conics and because $\S^3$ does not contain lines
these conics cannot be projected by $\eta$ to double lines in~$X$.
Therefore, the left and right images of $\varphi$ must be circles.
Since $\varphi$ is biquadratic, the preimage~$\varphi^{-1}(C)$ of a circle~$C$ that is not a singular component
must be of bidegree either $(1,0)$ or $(0,1)$ and thus either a left or right image
as asserted.
As $\varphi$ is a morphism, it
induces an isomorphism~$\P^1\times\P^1\setminus U\cong X\setminus\Sing X$.
For the remaining assertion,
we observe that the composition~$\widetilde{\varphi}^{-1}\circ\varphi$
defines a birational automorphism~$\alpha\c\P^1\times\P^1\dto \P^1\times\P^1$.
It follows from \citep[Lemma~16]{2024fact} that $\varphi$, and thus $\alpha$,
does not contract complex curves to complex points.
By \citep[Theorem in \textsection4.2 at page~510]{1978}, birational maps between surfaces
that are not morphisms must contract a Zariksi open set of some complex curve
to a complex point.
It follows that $\alpha$ is a morphism and
the same argument for $\varphi^{-1}\circ\widetilde{\varphi}$ shows that $\alpha^{-1}$ is a morphism.
We conclude that $\alpha$ is a biregular automorphism of~$\P^1\times \P^1$
and $\varphi=\widetilde{\varphi}\circ\alpha$.
\end{proof}

\begin{lemma}
\label{lem:id}
Suppose that $\varphi\c \P^1\times\P^1\to X$ is a biquadratic birational morphism
\st $X\subset\S^3$ is a celestial surface of degree 8
and let $W\subset\P^1\times\P^1$ denote the Zariski closure of
\[
\set{p\in\P^1\times\P^1}{|(\varphi^{-1}\circ\varphi)(p)|\geq 2}.
\]
We consider the following set of bidegrees,
where $W_1,\ldots,W_k$ are the irreducible complex curve components of~$W$:
\[
D:=\set{\deg W_i}{1\leq i\leq k \text{ and } \deg W_i\notin\{(0,1),(1,0)\}}.
\]
If $\deg X\neq 4$ and $\SingType X\notin\{\Di,\Dii\}$,
then either one of the following holds:
\begin{Mlist}
\item $D=\{(2,2)\}$ and $\SingType X=\Diii$,
\item $D=\{(1,1)\}$ and either $\SingType X=\Div$ or $\VisType X\in\{\Diii[\circ],\Diii[~]\}$, or
\item $D=\varnothing$ and $\SingType X=\Dv$.
\end{Mlist}
\end{lemma}

\begin{proof}
We know from \Cref{thm:s} that $\SingType X\in\{\Diii,\Div,\Dv\}$.
By \Cref{prp:varphi}, $\varphi$ is unique up to composition
with elements of~$\aut \P^1\times\P^1$.
If $C\subset\Sing X$ is a singular component \st $\deg \varphi^{-1}(C)=(\alpha,\beta)$,
then up to $\aut N(X)$ its class is $[C]=\alpha\,\l_0+\beta\,\l_1$.
Recall from \Cref{prp:BGE8} that $X$ is a dP surface with $2\,\l_0+2\,\l_1$ as class of hyperplane sections.
Hence, $\varphi^{-1}(C)\subset W$ if and only if $(2\,\l_0+2\,\l_1)\cdot [C]>\deg C$.
If $\SingType X=\Diii$ and the preimage $\varphi^{-1}(V)$
of the twisted double quartic~$V\subset\Sing X$ is reducible,
then $\varphi^{-1}(V)$ must consists of two complex curves $W_1$ and $W_2$
\st $\varphi(W_1)=\varphi(W_2)=V$ and $\deg W_1=\deg W_2=(1,1)$.
Moreover, in this case $\VisType X\in\{\Diii[\circ],\Diii[~]\}$ by \Cref{lem:Wred}.
If $\SingType X=\Diii$ and $C\subset\Sing X$ the conic component with class~$2\,\l_0$,
then $\varphi^{-1}(C)$ is the union of two components each of bidegree~$(1,0)$ by the
genus formula (see \citep[Lemma~15(b)]{2024great}).
The proof is now concluded as a direct consequence of \ref{iv} at \Cref{def:st}.
\end{proof}

\begin{remark}
\label{rmk:deg4}
If $X\subset\S^3$ is a Cliffordian celestial surface \st $\deg X\neq 8$,
then it follows from \citep[Theorem~1(b)]{2024darboux} that $\deg X=4$ and
either
\begin{Mlist}
\item  $\Sing X_\R=\varnothing$ and $X$ is $\lambda$-circled for some $4\leq \lambda\leq 5$, or
\item $|\Sing X_\R|=1$ and there exists a M\"obius transformation $\alpha\in\aut\S^3$
\st the Euclidean model $\bU(\alpha(X))$ is a one-sheeted circular hyperboloid.
\END
\end{Mlist}
\end{remark}

\begin{remark}
\label{rmk:cst}
Suppose that $A,B\subset S^3$ are circles and let $X\subset \S^3$ denote the celestial surface
whose spherical model~$\bS(X)\subset S^3$ is equal to $\set{a\star b}{a\in A,~b\in B}$.
We describe an algorithm that computes $\SingType X$ from the circles $A$ and $B$
under the assumption that $A$ and $B$ are defined over $\Q$.
Let $\rho_A\c \C\dto A$ and $\rho_B\c \C\dto B$ be birational maps.
Recall from \cref{rmk:pmz} the specification of
the Hamiltonian product $\_\star\_\c S^3\times S^3\to S^3$
and let $\iota\c \R^4\hookrightarrow \P^4$ send $(y_1,\ldots,y_4)$ to $(1:y_1:\ldots:y_4)$.
The parametrization $\C^2\dto \bS(X)$ that sends $(s,t)$ to $\rho_A(s)\star\rho_B(t)$
corresponds via the embedding $\iota$ to the birational map $f\c \C^2\dto X$ that sends $(s,t)$ to
$(f_0(s,t):\cdots:f_4(s,t))$
where $f_i\in \Q[s,t]$ is for all~$0\leq i\leq 4$ a polynomial defined over~$\Q$.
If the ideal generated by $\{f_i\}_i$ is nontrivial, then $\deg X\neq 8$ and
thus $\SingType X$ is not defined (see \Cref{rmk:deg4}).
In the remainder, we assume that $\deg X=8$ so that $f$ extends to a biquadratic birational morphism~$\P^1\times\P^1\to X$.
We would like to compute the ideal of the algebraic set
\[
U:=\set{(s,t)\in\C^2}{|(f^{-1}\circ f)(s,t)|>1}.
\]
If $\pi\c \C^2\times\C^2\times\C\to \C^2$
is the projection to the first $\C^2$ component, then
\[
\begin{array}{@{}l@{~}l@{~}l@{~}l}
U&=\pi(\{(s,t,u,v,y)\in\C^2\times\C^2\times\C &:& f(s,t)=f(u,v),~ (s,t)\neq (u,v) \})\\
    &=\pi(\{(s,t,u,v,y)\in\C^2\times\C^2\times\C &:&
    f_i(s,t)\cdot f_j(u,v)-f_j(s,t)\cdot f_i(u,v)=0\\
    &&&\text{for all }0\leq i<j\leq 4,~(s\cdot v-t\cdot u)\cdot y=1\}).
\end{array}
\]
Hence, the ideal of $U$ is equal to the elimination ideal~$ I\cap \C[s,t]$, where
$I\subset\C[s,t,u,v,y]$ is the ideal generated by the polynomials
\[
\set{f_i(s,t)\cdot f_j(u,v)-f_j(s,t)\cdot f_i(u,v)}{0\leq i<j\leq 4}\cup\{s\cdot v-t\cdot u)\cdot y-1\}.
\]
We compute the elimination ideal $I\cap \C[s,t]$
by computing a Gr\"obner basis for~$I$.
Suppose that $\gamma\c \C^2\hookrightarrow \P^1\times\P^1$ is the embedding that
sends~$(s,t)$ to $(1:s;1:t)$ and let $W\subset\P^1\times\P^1$ denote the Zariski closure
of $\gamma(U)$.
Since $X$ is Cliffordian,
we know from \Cref{cor:BC} that $\SingType X\in\{\Diii,\Div,\Dv\}$
and thus $W$ is a hypersurface. Hence, the ideal~$I$ is generated by a single polynomial~$p\in\Q[s,t]$.
If $p_1,\ldots,p_k\in\C[s,t]$ denote the irreducible factors of~$p$,
then the bidegree of the form~$p_i\circ\gamma^{-1}$ for $1\leq i\leq k$
corresponds to the bidegree of the corresponding irreducible component of~$W$.
We now recover the possible values for $\SingType X$ from \Cref{lem:id}.
See \cite{2024maple} for a
\href{https://github.com/niels-lubbes/celestial-singularities}{Maple implementation}.
\END
\end{remark}

\begin{remark}
\label{rmk:cvt}
Suppose that $A,B\subset S^3$ are circles and let $X\subset \S^3$ denote the celestial surface
\st $\bS(X)=\set{a\star b}{a\in A,~b\in B}$.
We describe a method for determining $\VisType X$ from the circles $A$ and $B$.
We first compute $\SingType X$ using \Cref{rmk:cst}.
Next, we apply \Cref{cor:VisType}
and reduce the possible
values of $\VisType X$ to finitely many cases.
We now recover $\VisType X$ from observing a rendering
of the visible points of the Euclidean model~$\bU(X)$ in~$\R^3$
together with the parameter lines that are defined by stereographic projections of the circles
$A\cdot\{b\}$
and
$\{a\}\cdot B$
for some $a\in A$ and $b\in B$.
See \Cref{tab:D123,tab:D45} for examples of such renderings.
We refer to the \href{https://github.com/niels-lubbes/celestial-surfaces#experiment-with-cliffordian-surfaces}{software}
at \cite{2024celest}
for observing the surface rendering from different view angles,
refining the net of circular parameter lines, and
perturbing the circles~$A$ and~$B$.
\END
\end{remark}

\begin{remark}
Our goal is to verify each row of \Cref{tab:D}.
Let $X\subset\S^3$ be the celestial surface
defined by the parametric type in the second column of \Cref{tab:D}.
In the first two rows of \Cref{tab:D}, the celestial surface $X$ is Bohemian by construction and thus
$\VisType X$ is either $\Di[=]$ or $\Dii[+]$ by \Cref{cor:BC,cor:VisType}.
For the remaining rows of \Cref{tab:D}
we first use \Cref{rmk:pmz} to recover from the parametric type the circles $A,B\subset S^3$
\st $\bS(X)=\set{a\star b}{a\in A,~b\in B}$.
Next, we compute $\SingType X$ and $\VisType X$ using \Cref{rmk:cst,rmk:cvt}.
We refer to \cite{2024maple,2024celest} for the software implementations.
\END
\end{remark}

\begin{remark}
The topological types of $\Dia$ and $\Diia$ as diagrammatically illustrated in
the third column of \Cref{tab:D123} follow from \Cref{lem:D12}.
Since $\Dva$, $\Dvb$ and $\Dvc$ are by construction great celestial surfaces,
their topological types as illustrated in the third column of \Cref{tab:D45} are confirmed by
\citep[Corollaries~I and~II]{2024great}.
The diagrammatic descriptions of the topological types of the remaining celestial surfaces
illustrated in \Cref{tab:D123,tab:D45} are conjectural.
\END
\end{remark}

\section{Acknowledgements}

Financial support was provided by the Austrian Science Fund (FWF): P33003 and P36689.

\bibliography{sing-8}

\begin{thebibliography}{10}

\bibitem{1987}
F.~Ap{\'e}ry.
\newblock Models of the real projective plane. {Computer} graphics of {Steiner}
  and {Boy} surfaces, 1987.

\bibitem{1980}
R.~Blum.
\newblock Circles on surfaces in the {Euclidean} space.
\newblock Geometry and differential geometry, {Proc}. {Conf}., {Haifa} 1979,
  {Lect}. {Notes} {Math}. 792, 213-221, 1980.

\bibitem{2011arc}
P.~Bo, H.~Pottmann, M.~Kilian, W.~Wang, and J.~Wallner.
\newblock Circular arc structures.
\newblock {\em ACM Trans. Graph.}, 30:101, 07 2011.
\newblock \href {https://doi.org/10.1145/2010324.1964996}
  {\path{doi:10.1145/2010324.1964996}}.

\bibitem{2012ort}
A.I. Bobenko and E.~Huhnen-Venedey.
\newblock Curvature line parametrized surfaces and orthogonal coordinate
  systems: discretization with {Dupin} cyclides.
\newblock {\em Geom. Dedicata}, 159:207--237, 2012.
\newblock \href {https://doi.org/10.1007/s10711-011-9653-5}
  {\path{doi:10.1007/s10711-011-9653-5}}.

\bibitem{2017sup}
A.I. {Bobenko}, E.~{Huhnen-Venedey}, and T.~{R\"orig}.
\newblock {Supercyclidic nets}.
\newblock {\em {Int. Math. Res. Not.}}, 2017(2):323--371, 2017.
\newblock \href {https://doi.org/10.1093/imrn/rnv328}
  {\path{doi:10.1093/imrn/rnv328}}.

\bibitem{1979}
J.W. Bruce and C.T.C. Wall.
\newblock On the classification of cubic surfaces.
\newblock {\em J. Lond. Math. Soc., II. Ser.}, 19:245--256, 1979.
\newblock \href {https://doi.org/10.1112/jlms/s2-19.2.245}
  {\path{doi:10.1112/jlms/s2-19.2.245}}.

\bibitem{1996}
A.~Coffman, A.J. Schwartz, and C.~Stanton.
\newblock The algebra and geometry of {Steiner} and other quadratically
  parametrizable surfaces.
\newblock {\em Comput. Aided Geom. Des.}, 13(3):257--286, 1996.
\newblock \href {https://doi.org/10.1016/0167-8396(95)00026-7}
  {\path{doi:10.1016/0167-8396(95)00026-7}}.

\bibitem{2012dol}
I.V. Dolgachev.
\newblock {\em Classical algebraic geometry. {A} modern view}.
\newblock Cambridge University Press, 2012.
\newblock \href {https://doi.org/10.1017/CBO9781139084437}
  {\path{doi:10.1017/CBO9781139084437}}.

\bibitem{1822}
C.~Dupin.
\newblock Applications de g\'eom\'etrie et de m\'echanique.
\newblock 1822.

\bibitem{2024fact}
J.~Frischauf, N.~Lubbes, and H.-P. Schr\"ocker.
\newblock Bivariate quaternionic factorizations and surfaces that decompose
  into two circles.
\newblock 2024.
\newblock submitted preprint,
  \href{https://nielslubbes.com/pp/preprint3-decompose.pdf}{nielslubbes.com/pp/preprint3-decompose.pdf}.

\bibitem{1978}
P.~Griffiths and J.~Harris.
\newblock Principles of algebraic geometry, 1978.

\bibitem{1977}
R.~Hartshorne.
\newblock {\em Algebraic geometry}.
\newblock Springer-Verlag, 1977.

\bibitem{1990}
D.~Hilbert and S.~Cohn-Vossen.
\newblock Geometry and the imagination.
\newblock {Chelsea Publishing Company}, 1990.

\bibitem{2023vision}
E.~Hoxhaj, J.M. Menjanahary, and J.~Schicho.
\newblock Using algebraic geometry to reconstruct a darboux cyclide from a
  calibrated camera picture.
\newblock {\em Appl. Algebra Eng. Commun. Comput.}, 2023.
\newblock \href {https://doi.org/10.1007/s00200-023-00600-y}
  {\path{doi:10.1007/s00200-023-00600-y}}.

\bibitem{1995}
T.~Ivey.
\newblock Surfaces with orthogonal families of circles.
\newblock {\em Proc. Am. Math. Soc.}, 123(3):865--872, 1995.
\newblock \href {https://doi.org/10.2307/2160812} {\path{doi:10.2307/2160812}}.

\bibitem{2018kol}
J.~Koll{\'a}r.
\newblock Quadratic solutions of quadratic forms.
\newblock In {\em Local and global methods in algebraic geometry}, pages
  211--249. 2018.
\newblock \href {https://doi.org/10.1090/conm/712/14348}
  {\path{doi:10.1090/conm/712/14348}}.

\bibitem{2020kin}
R.~Krasauskas and S.~Zub{\.e}.
\newblock Kinematic interpretation of {Darboux} cyclides.
\newblock {\em Comput. Aided Geom. Des.}, 83:16, 2020.
\newblock \href {https://doi.org/10.1016/j.cagd.2020.101945}
  {\path{doi:10.1016/j.cagd.2020.101945}}.

\bibitem{2021circle}
N.~Lubbes.
\newblock Surfaces that are covered by two pencils of circles.
\newblock {\em Math. Z.}, 299(3-4):1445--1472, 2021.
\newblock \href {https://doi.org/10.1007/s00209-021-02713-x}
  {\path{doi:10.1007/s00209-021-02713-x}}.

\bibitem{2021web}
N.~Lubbes.
\newblock Webs of rational curves on real surfaces and a classification of real
  weak del {Pezzo} surfaces.
\newblock {\em J. Lond. Math. Soc., II. Ser.}, 103(2):398--448, 2021.
\newblock \href {https://doi.org/10.1112/jlms.12379}
  {\path{doi:10.1112/jlms.12379}}.

\bibitem{2024maple}
N.~Lubbes.
\newblock Maple library for computing the singular loci of celestial surfaces,
  2024.
\newblock
  \href{https://github.com/niels-lubbes/celestial-singularities}{github.com/niels-lubbes/celestial-singularities}.

\bibitem{2024celest}
N.~Lubbes.
\newblock Mathematica code for celestial surfaces, 2024.
\newblock
  \href{https://github.com/niels-lubbes/celestial-surfaces}{github.com/niels-lubbes/celestial-surfaces}.

\bibitem{2024great}
N.~Lubbes.
\newblock Shapes of surfaces that contain a great and a small circle through
  each point.
\newblock 2024.
\newblock submitted preprint,
  \href{https://nielslubbes.com/pp/preprint2-great.pdf}{nielslubbes.com/pp/preprint2-great.pdf}.

\bibitem{2024darboux}
N.~Lubbes.
\newblock Translational and great {D}arboux cyclides.
\newblock {\em C. R. Math. Acad. Sci. Paris}, 362:413--448, 2024.
\newblock \href {https://doi.org/10.5802/crmath.603}
  {\path{doi:10.5802/crmath.603}}.

\bibitem{2018kin}
N.~Lubbes and J.~Schicho.
\newblock Kinematic generation of {Darboux} cyclides.
\newblock {\em Comput. Aided Geom. Des.}, 64:11--14, 2018.
\newblock \href {https://doi.org/10.1016/j.cagd.2018.06.001}
  {\path{doi:10.1016/j.cagd.2018.06.001}}.

\bibitem{2005}
R.~Piene.
\newblock Singularities of some projective rational surfaces.
\newblock In {\em Computational methods for algebraic spline sufaces
  (COMPASS)}, pages 171--182. Springer, 2005.

\bibitem{1985}
U.~Pinkall.
\newblock Regular homotopy classes of immersed surfaces.
\newblock {\em Topology}, 24:421--434, 1985.
\newblock \href {https://doi.org/10.1016/0040-9383(85)90013-8}
  {\path{doi:10.1016/0040-9383(85)90013-8}}.

\bibitem{2012web}
H.~{Pottmann}, L.~{Shi}, and M.~{Skopenkov}.
\newblock {Darboux cyclides and webs from circles}.
\newblock {\em {Comput. Aided Geom. Des.}}, 29(1):77--97, 2012.
\newblock \href {https://doi.org/10.1016/j.cagd.2011.10.002}
  {\path{doi:10.1016/j.cagd.2011.10.002}}.

\bibitem{2001}
J.~Schicho.
\newblock The multiple conical surfaces.
\newblock {\em Beitr. Algebra Geom.}, 42(1):71--87, 2001.

\bibitem{2019sko}
M.~Skopenkov and R.~Krasauskas.
\newblock Surfaces containing two circles through each point.
\newblock {\em Math. Ann.}, 373(3-4):1299--1327, 2019.
\newblock \href {https://doi.org/10.1007/s00208-018-1739-z}
  {\path{doi:10.1007/s00208-018-1739-z}}.

\bibitem{2000}
N.~Takeuchi.
\newblock Cyclides.
\newblock {\em Hokkaido Math. J.}, 29(1):119--148, 2000.
\newblock \href {https://doi.org/10.14492/hokmj/1350912960}
  {\path{doi:10.14492/hokmj/1350912960}}.

\bibitem{1986}
T.~Urabe.
\newblock Classification on non-normal quartic surfaces.
\newblock {\em Tokyo J. Math.}, 9:265--295, 1986.
\newblock \href {https://doi.org/10.3836/tjm/1270150719}
  {\path{doi:10.3836/tjm/1270150719}}.

\bibitem{1669}
C.~Wren.
\newblock Generatio corporis cylindroidis hyperbolici, elaborandis lenti bus
  hyperbolicis accommodati.
\newblock 4:961--962, 1669.

\end{thebibliography}

\textbf{address:}
Institute for Algebra, Johannes Kepler University, Linz, Austria
\\
\textbf{email:}
info@nielslubbes.com

\end{document}